\documentclass[12pt,centertags,oneside]{amsart}

\usepackage{amsmath,amstext,amsthm,amscd,typearea,hyperref,stmaryrd,mathabx}
\usepackage{amssymb}
\usepackage{a4wide}
\usepackage[mathscr]{eucal}
\usepackage{mathrsfs}
\usepackage{typearea}
\usepackage{charter}
\usepackage{pdfsync}
\usepackage[a4paper,width=16.2cm,top=3cm,bottom=3cm,lines=50]{geometry}
\usepackage[all]{xy}
\usepackage{tikz}
\usepackage{enumitem}
\usepackage[foot]{amsaddr}
\usepackage{academicons}
\usepackage{xcolor}

\numberwithin{equation}{section}
\newcommand{\orcid}[1]{\href{https://orcid.org/#1}{\textsc{orc}i\textsc{d}}}

\title{Finiteness of hyperbolic entropy for holomorphic foliations with non-degenerate singularities}

\author{Fran\c cois Bacher}
\address{Université Bourgogne Europe, CNRS, IMB UMR 5584, F-21000 Dijon, France}
\email{francois.bacher@ube.fr}

\date{\today}

\def\restriction#1#2{\mathchoice
              {\setbox1\hbox{${\displaystyle #1}_{\scriptstyle #2}$}
              \restrictionaux{#1}{#2}}
              {\setbox1\hbox{${\textstyle #1}_{\scriptstyle #2}$}
              \restrictionaux{#1}{#2}}
              {\setbox1\hbox{${\scriptstyle #1}_{\scriptscriptstyle #2}$}
              \restrictionaux{#1}{#2}}
              {\setbox1\hbox{${\scriptscriptstyle #1}_{\scriptscriptstyle #2}$}
              \restrictionaux{#1}{#2}}}
\def\restrictionaux#1#2{{#1\,\smash{\vrule height .8\ht1 depth .85\dp1}}_{\,#2}}

\theoremstyle{plain}
\newtheorem{thm}{Theorem}[section]
\newtheorem{lem}[thm]{Lemma}

\newtheorem{prop}[thm]{Proposition}

\newtheorem*{thm*}{Theorem}
\newtheorem*{conj*}{Conjecture}

\theoremstyle{definition}
\newtheorem{defn}[thm]{Definition}

\newtheorem*{exmp*}{Example}

\theoremstyle{remark}
\newtheorem{rem}[thm]{Remark}

\makeatletter

\@addtoreset{equation}{section}
\makeatother

\DeclareMathOperator{\id}{id}

\DeclareMathOperator{\sing}{sing}

\DeclareMathOperator{\card}{card}

\DeclareMathOperator{\Hol}{Hol}

\DeclareMathOperator{\cst}{cst}

\DeclareMathOperator{\Area}{Area}

\newcommand{\cjg}[1]{\overline{#1}}

\newcommand{\adh}[1]{\overline{#1}}

\newcommand{\PC}{P}

\newcommand{\eps}{\varepsilon}

\newcommand{\fol}{\mathscr{F}}

\newcommand{\leafatlas}{\mathscr{L}}

\newcommand{\plfol}{\left(\mani{M},\leafatlas,\mani{E}\right)}

\newcommand{\zlogz}[1]{#1\log^{\star}#1}

\newcommand{\set}[1]{\mathbb{#1}}

\newcommand{\der}[2]{\frac{\partial#1}{\partial#2}}

\newcommand{\bder}[2]{\der{#1}{\cjg{#2}}}

\newcommand{\leaf}{L}

\newcommand{\leafu}[1]{\leaf_{#1}}

\newcommand{\norm}[1]{\left\Vert#1\right\Vert}

\newcommand{\intcc}[2]{\left[#1,#2\right]}

\newcommand{\intoo}[2]{\left(#1,#2\right)}

\newcommand{\intco}[2]{\left[#1,#2\right)}

\newcommand{\intoc}[2]{\left(#1,#2\right]}

\newcommand{\intent}[2]{\left\llbracket#1,#2\right\rrbracket}

\newcommand{\sentp}[1]{\left\lceil#1\right\rceil}

\newcommand{\Cmod}[1]{\left\vert#1\right\vert}

\newcommand{\proj}[1]{\set{P}^{#1}}

\newcommand{\class}[1]{\mathscr{C}^{#1}}

\newcommand{\flot}[1]{\varphi_{#1}}

\newcommand{\mani}[1]{#1}

\newcommand{\manis}[2]{\mani{#1}\setminus\mani{#2}}

\newcommand{\dhimpsing}[2]{\dhimpnov(#2,\mani{#1})}

\newcommand{\rD}[1]{#1\set{D}}

\newcommand{\DR}[1]{\set{D}_{#1}}

\newcommand{\adhDR}[1]{\adh{\set{D}}_{#1}}

\newcommand{\wo}[2]{#1\backslash#2}

\newcommand{\dherm}[3]{d_{\mani{#1}}(#2,#3)}

\newcommand{\dhimp}[2]{\dhimpnov(#1,#2)}

\newcommand{\dhermfnov}[1]{d_{\leafu{#1}}}

\newcommand{\dhermf}[3]{\dhermfnov{#1}(#2,#3)}

\newcommand{\dhimpnov}{d}

\newcommand{\dhimpsnov}[1]{\dhimpnov_{#1}}

\newcommand{\dhimps}[3]{\dhimpsnov{#1}(#2,#3)}

\newcommand{\dPC}[2]{\dPCnov(#1,#2)}

\newcommand{\dPCs}[3]{d_{\PC,#1}(#2,#3)}

\newcommand{\dPCnov}{d_{\PC}}

\newcommand{\metPC}{g_{\PC}}

\newcommand{\metm}[1]{g_{\mani{#1}}}

\newcommand{\textfol}{holomorphic foliation}

\newcommand{\textsingfol}{singular \textfol{}}

\newcommand{\ballleafu}[2]{\leafu{#1}[#2]}

\newcommand{\lPC}{\ell_{\PC}}

\newcommand{\seg}[2]{[#1,#2]}

\newcommand{\wt}[1]{\widetilde{#1}}

\newcommand{\wh}[1]{\widehat{#1}}

\newcommand{\disk}[2]{D\left(#1,#2\right)}

\newcommand{\foldPnC}[2]{\fol_{#1}(\proj{#2})}

\begin{document}

\theoremstyle{plain}

\begin{abstract}  Consider $\fol$ a Brody-hyperbolic foliation on a compact complex surface $\mani{M}$. Suppose that the singularities of $\fol$ are all non-degenerate. We show that the hyperbolic entropy of $\fol$ is finite.
\end{abstract}

\maketitle

\section{Introduction}

There have been a lot of progress in the dynamical theory of laminations by Riemann surfaces during the last two decades. More precisely, much of attention has been focused on building an ergodic theory when the leaves are hyperbolic. To have such a setup, the case of the projective spaces is very typical. Indeed, every polynomial vector field on $\set{C}^n$ can be compactified naturally into a holomorphic foliation on $\proj{n}$. This foliation is always singular. Let $d,n\in\set{N}$ with $n\geq2$, denote by $\foldPnC{d}{n}$ the space of singular holomorphic foliations of degree~$d$ on $\proj{n}$. Lins~Neto~ and Soares~\cite{LNS}, using a work of Jouanolou~\cite{Jou}, show that a generic foliation $\fol\in\foldPnC{d}{n}$ has only non-degenerate singularities. Moreover, by a result of Lins~Neto~\cite{LN2} and Glutsyuk~\cite{Glu}, such a foliation is hyperbolic if $d\geq2$. It is even Brody-hyperbolic in the sense of~\cite{DNSII}. Loray and Rebelo~\cite{LorReb} also build a non-empty open subset of these foliations, the leaves of which are all dense in $\proj{n}$. When $n=2$, Nguy\^{e}n~\cite{NguLyap} uses the integrability of the holonomy cocycle in~\cite{NguHolo} to compute the Lyapunov exponent of a generic foliation $\fol\in\foldPnC{d}{2}$. We recall briefly some recent studies and refer the reader to the survey articles~\cite{surDinhSib,surForSib,surVANG18,surVANG21} for a more detailed exposition.

By solving heat equations with respect to harmonic currents, Dinh, Nguy\^{e}n and \mbox{Sibony} are able in~\cite{DNS12} to prove abstract ergodic theorems for laminations and foliations. This new approach enables them to develop an effective ergodic theory for laminations and foliations, and in particular, geometric versions of Birkhoff's theorem in this context. In two articles~\cite{DNSI,DNSII}, the three authors study a modulus of continuity for the leafwise Poincar\'e metric. More precisely, they show that it is H\"{o}lder in the case of a compact regular hyperbolic foliation, and H\"{o}lder with a logarithmic slope towards the singularities in the case of linearizable singularities. Somehow, their work on the heat equation implicitly studies the dynamics of foliations in a canonical time, which is measured by the Poincar\'e distance in the universal covering. From this viewpoint, they introduce a canonical notion of hyperbolic entropy. Heuristically, this is a way to measure how exponentially fast the leaves get apart, with respect to the hyperbolic time. This definition is inspired by Bowen's notion of entropy. See Subsection~\ref{subsecentropy} for a more precise definition. The three authors are able to prove the following finiteness results.

\begin{thm}\label{thmDNSintro}
  \begin{enumerate}[label=(\arabic*),ref=\arabic*]
  \item\label{thmDNSintro1} \emph{(Dinh--Nguy\^{e}n--Sibony~{\cite[Theorem~3.10]{DNSI}})} Let $\fol=(X,\leafatlas)$ be a smooth compact lamination by hyperbolic Riemann surfaces. Then, the hyperbolic entropy of~$\fol$ is finite.
  \item\label{thmDNSintro2} \emph{(Dinh--Nguy\^{e}n--Sibony~{\cite[Theorem~1.1]{DNSII}})} Let $\fol=\plfol$ be a Brody-hy\-per\-bo\-lic \textsingfol{} on a compact complex surface~$\mani{M}$. Suppose that all the singularities of $\fol$ are linearizable. Then, the hyperbolic entropy of $\fol$ is finite.
  \end{enumerate}
\end{thm}

Above, we use the notation of the survey articles~\cite{surDinhSib,surForSib,surVANG18,surVANG21}, where $\fol=(X,\leafatlas)$ means that~$\leafatlas$ is an atlas of flow boxes of $X$, and $\fol=\plfol$ means that~$\leafatlas$ is an atlas of flow boxes of~$\manis{M}{E}$, where~$\mani{E}$ is a closed set of~$\mani{M}$ which is called the set of singularities of~$\fol$. The finiteness of the hyperbolic entropy in this theorem is strongly dependent on their previous result on the modulus of continuity of the leafwise Poincaré metric, as can be seen in~\cite[Theorem~2.1]{DNSI} and~\cite[Theorem~3.2]{DNSII}. In our previous work~\cite{Bac1,Bac2}, we generalize this regularity result to foliations with non-degenerate singularities. In this article, we obtain the following generalization of Theorem~\ref{thmDNSintro}~\eqref{thmDNSintro2}.

\begin{thm}\label{mainthm} Let $\fol=\plfol$ be a Brody-hyperbolic \textsingfol{} on a compact complex surface. Suppose that all the singularities of $\fol$ are non-degenerate. Then, the hyperbolic entropy of $\fol$ is finite.
\end{thm}

Let us explain briefly the method of our proof. We follow the general strategy of the three authors in~\cite{DNSII} for linearizable singularities. They are able to ensure that two points are at small Bowen distance by solving a Beltrami equation for a map that is obtained by gluing local orthogonal projections from a leaf to another. More precisely, let $x,y\in\mani{M}$ and $\leafu{x},\leafu{y}$ be the leaves through these points. Denote by $\phi_x\colon\set{D}\to\leafu{x}$ a universal covering of~$\leafu{x}$ such that $\phi_x(0)=x$. If we can construct a map $\psi\colon\set{D}\to\leafu{y}$, close to~$\phi_x$ on a large hyperbolic disk, and close to holomorphic (in the sense of the $\class{1}$-norm of the Beltrami coefficient of a lifting), then~$x$ and~$y$ are at small Bowen distance. This criterion involves solving a Beltrami equation to make~$\psi$ holomorphic, and showing that this new holomorphic map is close to a rotation. Unfortunately, the Beltrami coefficient of the local orthogonal projection could explode near the singularities. Therefore, the authors need to correct this map and make it holomorphic when approaching the singular set. They do so by replacing the orthogonal projection by the flow. Actually, the three authors do not name it as such, since everything is explicit for linearizable singularities.

To have small cells on which we can do the above process, they construct transversals on which we can control precisely the flow and the orthogonal projection in small hyperbolic time. Carrying a covering on a transversal to another by holonomy and using a crucial refinement lemma (see Lemma~\ref{lemref} below), they construct a covering such that they can define an orthogonal projection for two points in the same cell, up to a hyperbolic time~$R$. Thanks to this  lemma,  they are able to control the growth of the cardinality of coverings all along their refinement process. 
This is one of their main ingredients in proving the finiteness of the entropy.


To adapt their proof, we use a classification of non-degenerate singularities in dimension~$2$ in three types.
\begin{itemize}
\item The linearizable singularities.
\item The singularities with two separatrices and real negative characteristic number.
\item The resonant singularities.
\end{itemize}
By Poincar\'e linearization Theorem and Briot--Bouquet Theorem, every non-degenerate singularity is of one of these types. Moreover, Poincar\'e--Dulac Theorem enables us to have an explicit form for resonant singularities. We show the same kind of estimates as~\cite{DNSII} for both cases with separatrices, and stronger ones for the resonant case. This gives us an initial covering on which we control the behaviour of close leaves in small hyperbolic time. To obtain such estimates, we use a generalization of Gr\"{o}nwall Lemma for non-linear differential equations, due to Lins~Neto and Canille~Martins~\cite{MLN}. We also use their estimation of the Poincar\'e metric near the singularities to compare the flow time and the hyperbolic time.


It remains to check and adapt each technical element of~\cite{DNSII}. In our case, we have to deal with the fact that our estimates on the flow are only local and non-explicit. Moreover, there are new phenomena emerging from our setting. First, two points could have different flow monodromies near the singularities. This could be a problem when replacing the orthogonal projection by the flow. We need to show that up to choosing well the vector field representing the foliation, two distinct close points can not have distinct close monodromy flow times. Second, we need disks to use the refinement lemma and they are not preserved by holonomy. This was already the case in~\cite{DNSII}. Here, we need to control the distortion of holonomy mappings near the more general singularities. Sometimes the resonant case and sometimes the two separatrices case need new arguments. Some of our techniques are slightly different from the three authors', but the main structure of our proof is very similar. Most of the time, our work on the linearizable case is just reproving what is already there in~\cite{DNSII} but in a slightly different setting. This enables us to clarify our work and statements in both other cases.

The article is organized as follows. In Section~\ref{secprel}, we introduce the hyperbolic entropy following~\cite{DNSI}. Moreover, we recall our previous work on local orthogonal projections from a leaf to another and the generalization of the Gr\"{o}nwall Lemma. In Section~\ref{seclocal}, we study the flow in a small step of hyperbolic time for the three types of singularities. We obtain a first cell decomposition. In Section~\ref{secreduc}, we show a sufficient condition for the entropy to be finite. This criterion involves the orthogonal projections and their corrections near the singularities. This is where we need to control the monodromy. In Section~\ref{sectrans}, we build a hyperbolically dense mesh of transversals and the initial covering that is refined later to obtain the Bowen cells. In Section~\ref{secholo}, we study the holonomy mappings and their distortion in small hyperbolic time. This enables us to carry information during the refinement process. In Section~\ref{sechyp}, we consider trees that encode the dynamics on the universal cover $\set{D}$, and are compatible with our mesh of transversals. Section~\ref{secproof} ends the proof by exposing more precisely our refinement algorithm and building the orthogonal projection that is needed for our criterion.


\subsection*{Notations}Throughout this paper, we denote by $\set{D}$ the unit disk of $\set{C}$, and $\rD{r}$ the open disk of radius $r\in\set{R}_+^*$ for the standard Euclidean metric of $\set{C}$. For $R\in\set{R}_+^*$, we denote by $\DR{R}$ the open disk of hyperbolic radius $R$ in $\set{D}$, so that $\DR{R}=\rD{r}$ with $r=\frac{e^R-1}{e^R+1}$, or if $r\in\intco{0}{1}$, with $R=\ln\frac{1+r}{1-r}$. More generally, for $\rho\in\set{R}_+^*$ and $U$ a subset of a vector space with a marked point~$z_0$, $\rho U$ denotes the image of $U$ by the homothety $z\mapsto z_0+\rho(z-z_0)$. In particular, if $D\subset\set{C}$ is a disk of radius $r$, $\rho D$ is the disk of same center and radius $r\rho$.

We consider several distances on a complex manifold~$\mani{M}$. For $\metm{M}$ a Hermitian metric on $\mani{M}$, the distance induced by $\metm{M}$ is denoted by $\dhimpnov$. Consider $\fol=\plfol$ a \textsingfol{}.  If $\leaf$ is a leaf of $\fol$, $\metm{M}$ induces a distance on $\leaf$ that we denote by $d_{\leaf}$. If $\leaf$ is hyperbolic, then $\leaf$ is endowed with the Poincar\'e metric denoted $\metPC$ and the induced distance denoted $\dPCnov$. We use the same notation for the Poincar\'e metric and distance on $\set{D}$. If $x\in\manis{M}{E}$ is such that the leaf through $x$, denoted $\leafu{x}$, is hyperbolic, we note $\phi_x\colon\set{D}\to\leafu{x}$ a uniformization of $\leafu{x}$ such that $\phi_x(0)=x$. For $u,v$ two functions from~$K$ to $\mani{M}$, to a leaf or to $\set{D}$, we denote by
\[\dhimps{K}{u}{v}=\underset{x\in K}{\sup}\,\dhimp{u(x)}{v(x)},\quad\dPCs{K}{u}{v}=\underset{x\in K}{\sup}\,\dPC{u(x)}{v(x)}.\]

We try to make our notations different for different contexts. If $z\in\set{C}$ and $r\in\set{R}_+^*$, we note $D(z,r)$ the disk of center $z$ and radius $r$. Inside a metric space, we denote by $B(x,r)$ the ball of center $x$ and radius $r$. We try to keep this notation for ambient Hermitian distances. Inside a leaf, we denote by $\ballleafu{x}{r}=\{y\in\leafu{x};\,\dhermf{x}{x}{y}<r\}$. Inside the Poincar\'e disk, if $\xi\in\set{D}$ and $R\in\set{R}_+$, we note $\DR{R}(\xi)=\{\zeta\in\set{D};\,\dPC{\xi}{\zeta}<R\}$.

The term ``constant'' means a real positive number that does not depend on a point $x\in\manis{M}{E}$, nor on the hyperbolic radius $R$ that will go to $+\infty$. Most of our proof relies on the fact that some constants $h,h_1,\hbar$ are sufficiently small, independently on $R$, given that it is sufficiently large. We may have forgotten to say it somewhere and the reader can suppose it is in the hypotheses of every statement. When we do not care about constants, we simply denote them by $C,C',C''$, \dots{} When we want to keep track of them to clarify our arguments, we denote them by $C_0,C_1,C_2$, \dots{}

Finally, we denote by $\sentp{a}$ the smallest integer $k$ such that $k\geq a$, for $a\in\set{R}$. We also denote $\Re(a)$ (resp. $\Im(a)$) the real (resp. imaginary) part of a complex number $a$.

\subsection*{Acknowledgments} The  author would like to thank the referee for his interesting remarks, which helped in making the exposition clearer. The author is supported by the Labex CEMPI (ANR-11-LABX-0007-01) and by the project QuaSiDy (ANR-21-CE40-0016). 

\section{Preliminaries}\label{secprel}

\subsection{Leafwise Poincar\'{e} metric}

In all this section, we let $\fol=\plfol$ be a \textsingfol{} on a complex manifold $\mani{M}$. Suppose that $\mani{M}$ is endowed with a Hermitian metric $\metm{M}$ and for $x\in\manis{M}{E}$, consider
\[\eta(x)=\sup\left\{\norm{\alpha'(0)}_{\metm{M}};\,\alpha\colon\set{D}\to\leafu{x}~\text{holomorphic}~\text{such}~\text{that}~\alpha(0)=x\right\}.\]
Above, $\norm{v}_{\metm{M}}$ is the norm of a vector $v\in T_x\leafu{x}$ with respect to the Hermitian metric $\metm{M}$. That is, $\norm{v}_{\metm{M}}=\left(g_{\mani{M},x}(v,v)\right)^{1/2}$. The map $\eta$ was introduced by Verjovsky in~\cite{Ver}. It is designed to satisfy the following facts.

\begin{prop}\label{propeta}
  \begin{enumerate}
  \item For $x\in\manis{M}{E}$, $\eta(x)<+\infty$ if and only if the leaf $\leafu{x}$ is hyperbolic, that is, it is uniformized by the Poincar\'e disk $\set{D}$.

   \item If $\leafu{x}$ is hyperbolic, we have $\eta(x)=\norm{\phi'(0)}_{g_{\mani{M}}}$, where $\phi\colon\set{D}\to\leafu{x}$ is any uniformization of~$\leafu{x}$ such that $\phi(0)=x$.
  \item If $\leafu{x}$ is hyperbolic, then $\frac{4g_{\mani{M}}}{\eta^2}$ induces the Poincar\'e metric on $\leafu{x}$.
  \end{enumerate}
\end{prop}

In this article, we are interested in the case of hyperbolic leaves and we need to specify our global setting. We follow~\cite{DNSII} in our vocabulary.

\begin{defn} If all the leaves of $\fol$ are hyperbolic, we say that $\fol$ is \emph{hyperbolic}. If moreover there exists a constant $c_0>0$ such that $\eta(x)<c_0$ for all $x\in\manis{M}{E}$, we say that~$\fol$ is \emph{Brody-hyperbolic}.
\end{defn}

From now on, we suppose that $\fol$ is hyperbolic. We also need to define the type of singularities we deal with.

  \begin{defn}
    Near a singularity $a\in E$, there exists a vector field $X$ defining $\fol{}$. In coordinates $(z_1,\dots{},z_n)$ centered at $a$, we can write
    \[X(z)=\sum\limits_{j=1}^nF_j(z)\der{}{z_j}.\]
    The functions $F_j$ can be developed as power series $F_j=\sum_{\alpha\in\set{N}^n}c_{\alpha,j}z^{\alpha}$. The \emph{1-jet} of $X$ at $a$ is defined in the chart $(U,z)$ as $X_1=\sum_{j=1}^n\sum_{\Cmod{\alpha}\leq1}c_{\alpha,j}z^{\alpha}\der{}{z_j}$. See~\cite[Chapter I]{Ilya} for more details. If the 1-jet of $X$ has an isolated singularity at $a$, we say that $a$ is a \emph{non-degenerate singularity of $\fol{}$}.
  \end{defn}

We use the following estimate of the map~$\eta$ for non-degenerate singularities. It can be found in~\cite[Proposition~4.2]{Bac1} and its proof is basically the same as Dinh, Nguy\^{e}n and Sibony's one in the context of linearizable singularities~\cite[Proposition~3.3]{DNSII}, together with a local estimate that is due to Lins~Neto and Canille~Martins~\cite[Theorem~2]{MLN}.

\begin{prop}\label{etazlogz} Note $\dhimpnov$ the distance induced by $\metm{M}$. Suppose that $\mani{M}$ is compact and that~$\fol$ is Brody-hyperbolic with only non-degenerate singularities. Then, there is a constant $C\geq1$ such that
  \[C^{-1}\zlogz{\dhimpsing{E}{x}}\leq\eta(x)\leq C\zlogz{\dhimpsing{E}{x}},\qquad x\in\manis{M}{E},\]
  where $\log^{\star}=1+\Cmod{\log}$ is a $\log$-type function.
\end{prop}

\subsection{Hyperbolic entropy}\label{subsecentropy}

For $x\in\manis{M}{E}$, denote by $\phi_x\colon\set{D}\to\leafu{x}$ a uniformization of $\leafu{x}$ such that $\phi_x(0)=x$. To unify notations, set also $\phi_a(\zeta)=a$, for $a\in\mani{E}$ and $\zeta\in\set{D}$. The idea of Dinh, Nguy\^{e}n and Sibony~\cite{DNSI} is to consider the Poincar\'e distance in $\set{D}$ to be a canonical time. More precisely, for $R\geq0$, consider the Bowen distance
\[d_R(x,y)=\inf\limits_{\theta\in\set{R}}\sup\limits_{\xi\in\adhDR{R}}\dhimp{\phi_x(\xi)}{\phi_y(e^{i\theta}\xi)},\quad x,y\in\mani{M}.\]
It measures the distance between the orbits of $x$ and $y$ up to time $R$. It is clear that it is independent on the choice of $\phi_x$. This enables us to define the entropy of $\fol$. For $x\in\mani{M}$, $R,\eps>0$, denote by $B_R(x,\eps)=\{y\in\mani{M}\,;\,d_R(x,y)<\eps\}$ the \emph{Bowen ball} of radius~$\eps$ and center~$x$ up to time $R$. For $Y\subset\mani{M}$, $R,\eps>0$ and $F\subset Y$, we  say that $F$ is \emph{$(R,\eps)$-dense} in~$Y$ if $Y\subset\cup_{x\in F}B_R(x,\eps)$. Denote by $N(Y,R,\eps)$ the minimal cardinality of an $(R,\eps)$-dense subset in $Y$. The \emph{hyperbolic entropy} of $Y$ is defined as
\[h(Y)=\sup\limits_{\eps>0}\limsup\limits_{R\to+\infty}\frac{1}{R}\log N(Y,R,\eps).\]
For $Y=\mani{M}$, we denote it by $h(\fol)$. If $\mani{M}$ is compact, then it does not depend on the choice of $\metm{M}$. A similar and equivalent definition can be made with maximal $(R,\eps)$-separated sets, but we do not need it. The interested reader can see~\cite{DNSI} for more details.

\subsection{Local orthogonal projection}

In order to show Theorem~\ref{mainthm}, we need to build cells in sufficiently small cardinality, such that two points $x,y\in\mani{M}$ in the same cell are close up to time $R$. To ensure such a proximity, we build a smooth map $\psi\colon\DR{R}\to\leafu{y}$ close to $\phi_x$, take its lifting $\Psi\colon\DR{R}\to\set{D}$ \emph{via} $\phi_y$, slightly correct it into a close holomorphic map $v\colon\DR{R}\to\set{D}$, and finally correct $v$ into a close rotation $r_{\theta}\colon\xi\mapsto e^{i\theta}\xi$. That way, we are able to show that $\phi_x$ and $\phi_y\circ r_{\theta}$ are close up to large time. Actually, the last two steps of this proof are hidden behind a result of Dinh, Nguy\^{e}n and Sibony~\cite[Proposition~3.6]{DNSII}. For the first step of this construction, we need to recall our previous work~\cite{Bac1} on local orthogonal projections from a leaf to another, which generalizes the estimates established in~\cite[p.~602]{DNSII}.

For $x\in\manis{M}{E}$, the metric $\metm{M}$ can be restricted to $\leafu{x}$ and induces a distance $\dhermfnov{x}$ on it. For $x\in\manis{M}{E}$ and $r>0$, denote by $\ballleafu{x}{r}=\{x'\in\leafu{x}\,;\,\dhermf{x}{x}{x'}<r\}$. Suppose that $\mani{M}$ is compact and that all the singularities of $\fol$ are non-degenerate.

      \begin{lem}[{\cite[Lemma~4.3]{Bac1}}]\label{orthproj} There exist constants $\eps_0$, $\eps_1$, $k$ and $K$ such that for two points $x,y\in\manis{M}{E}$, if $\dhimp{x}{y}\leq\eps_1\dhimpsing{E}{x}$, then there exists a local orthogonal projection
        \[\Phi_{xy}\colon\ballleafu{x}{\eps_0\dhimpsing{E}{x}}\to\ballleafu{y}{k\eps_0\dhimpsing{E}{y}},\]
        satisfying
        \begin{enumerate}
        \item $\dhermf{y}{y}{\Phi_{xy}(x)}\leq k\dhimp{x}{y}$,
        \item for $x_1,x_2\in\ballleafu{x}{\eps_0\dhimpsing{E}{x}}$, $\dhermf{y}{\Phi_{xy}(x_1)}{\Phi_{xy}(x_2)}\leq k\dhermf{x}{x_1}{x_2}.$
        \item $\Phi_{xy}$ is smooth and in a finite set of charts,
        \[\norm{\Phi_{xy}-\id}_{\infty}\leq e^K\dhimp{x}{\Phi_{xy}(x)},\quad\norm{\Phi_{xy}-\id}_{\class{1}}\leq e^K\frac{\dherm{}{x}{\Phi_{xy}(x)}}{\dhimpsing{E}{x}},\]
        \[\norm{\Phi_{xy}-\id}_{\class{2}}\leq e^K\frac{\dhimp{x}{\Phi_{xy}(x)}}{\dhimpsing{E}{x}^2}.\]
      \item If $x'\in\ballleafu{x}{\eps_0\dhimpsing{E}{x}}$, $y'\in\ballleafu{y}{k\eps_0\dhimpsing{E}{x}}$ and $\dhimp{x'}{y'}\leq\eps_1\dhimpsing{E}{x'}$, then $\Phi_{x'y'}=\Phi_{xy}$ on the intersection of their domains of definition.
      \end{enumerate}
    \end{lem}

    More precisely, we build the local orthogonal projection by solving an implicit equation on the flow. That way, in singular charts, we are able to estimate the flow time that is needed to join $y$ and $\Phi_{xy}(x)$, with the notations of the previous lemma.

    \begin{lem}[{\cite[Lemma~3.1]{Bac1}}] \label{floworthproj}Consider $X$ a holomorphic vector field on a neighbourhood of $\adh{\set{D}}^2$ with a non-degenerate singularity at the origin. Suppose that $\set{D}^2$ is endowed with the standard Hermitian metric on $\set{C}^2$. Let $x,y\in\wo{\frac{3}{4}\set{D}^2}{\{0\}}$ be such that $\Phi_{xy}$ exists in the sense of Lemma~\ref{orthproj}. Denote by $\flot{y}$ the flow of $X$ starting at $y$ from an open neighbourhood of $0$ in $\set{C}$ to $\leafu{y}$. Then, there exists a unique $t\in\set{C}$, with $t=O\left(\norm{x-y}\norm{x}^{-1}\right)$ such that $\Phi_{xy}(x)=\flot{y}(t)$.
    \end{lem}

    \subsection{Variations on the Gr\"{o}nwall Lemma}

    In order to obtain thorough estimates, we use several generalizations of the Gr\"{o}nwall Lemma, including some non-linear cases. First, let us state it in a form that contains these various versions and then do some remarks.
    
    \begin{prop}[Lins~Neto--Canille~Martins~{\cite[Proposition~6]{MLN}}]\label{Gronwall} Let $F\colon\set{R}_+\times\set{R}_+\to\set{R}_+$ be a continuous function such that $F(t,x)\leq F(t,y)$ if $x\leq y$. Suppose that for any $x_0\in\set{R}_+$, the Cauchy problem
      \begin{equation}\label{CauchypbMLN}x'(t)=F(t,x(t)),\qquad x(0)=x_0,\end{equation}
      has a unique maximal solution in a neighbourhood of $t=0$. Let $x\colon\intco{0}{r_x}\to\set{R}_+$ be a continuous function that satisfies
      \[x(t)\leq x_0+\int_0^tF(s,x(s))ds,\quad t\in\intco{0}{r_x},\]
      and $y\colon\intco{0}{r_y}\to\set{R}_+$ be the unique maximal solution of~\eqref{CauchypbMLN} starting at $x_0$. Then, for $t\in\intco{0}{\min(r_x,r_y)}$, $x(t)\leq y(t)$.
    \end{prop}

    \begin{rem} Actually, Lins~Neto and Canille~Martins prove a stronger estimate for partial orders on $\set{R}_+^n$, but we only need it for $n=1$. We use this result in two contexts. The first one is of an autonomous system, sometimes in the non-linear case and sometimes in the linear case (i.e. $F(t,x)=Cx$, which gives the classical Gr\"{o}nwall Lemma). The second one is a linear but non-autonomous system. More precisely, for $F(t,x)=Cx+f(t)$. In that case, note that we have
      \[y(t)=\left(x_0+\int_0^tf(s)e^{-Cs}ds\right)e^{Ct}.\]
    \end{rem}

 \section{Local cell decomposition}\label{seclocal}

\subsection{First estimates}

We begin by some local work near a non-degenerate singularity. We want to establish estimates of the divergence of orbits in a (small) step of hyperbolic time. More precisely, we want to decompose the singular open sets into small cells in which we have a good control of the flow in this hyperbolic time. Of course, we need a bound on the cardinality of this covering by cells. Our main result in this section (see Proposition~\ref{propcelldec}) is close to~\cite[Proposition~2.7]{DNSII}.

We consider $\fol=\plfol$ a Brody-hyperbolic \textsingfol{} on a compact complex surface $\mani{M}$ (that is $\dim_{\set{C}}\mani{M}=2$) with only non-degenerate singularities. Take $a\in\mani{E}$ and $U_a\simeq\set{D}^2$ a neighbourhood of $a$ on which $\fol$ is generated by a vector field~$X$. We suppose that the coordinates of $U_a$ and $X$ extend to a neighbourhood of $\adh{\set{D}}^2$. We need to clarify our vocabulary.

\begin{defn} For $z\in\wo{\frac{1}{2}\set{D}^2}{\{0\}}$, note $\flot{z}$ the flow of the vector field $X$ starting at $z$, defined on a maximal open subset of $\set{C}$ such that $\flot{z}(t)$ stays in $\set{D}^2$.

  A \emph{flow path} for $z$ and $X$ is a $\class{1}$ map $\gamma\colon\intcc{0}{T}\to\set{C}$, where $T\in\set{R}_+$, $\gamma(0)=0$ and $\flot{z}(\gamma(t))$ is well defined for all $t\in\intcc{0}{T}$ and belongs to $\frac{3}{4}\set{D}^2$. Most of the time, $T$ is implicit. The length of $\gamma$ as a path in $\set{C}$ will be denoted by $\ell(\gamma)$. The \emph{Poincar\'e length} $\lPC(\gamma)$ of $\gamma$ is by definition the Poincar\'e length of $\flot{z}\circ\gamma$ in $\leafu{z}$. The notation $\lPC(\gamma)$ is used only if there is no confusion possible for the point $z$ and the vector field $X$.

  Let $\delta\colon\intcc{0}{T}\to\leafu{z}\cap\frac{3}{4}\set{D}^2$ be a $\class{1}$ map such that $\delta(0)=z$. A flow path $\gamma\colon\intcc{0}{T}\to\set{C}$ for $z$ and~$X$ is said to \emph{correspond to $\delta$} if $\delta(t)=\flot{z}(\gamma(t))$. Fix a uniformization $\phi_z$ of $\leafu{z}$ such that $\phi_z(0)=z$. Let $\xi\in\set{D}$ be such that $\phi_z(\seg{0}{\xi})\subset\frac{3}{4}\set{D}^2$. We say that a flow path $\gamma\colon\intcc{0}{1}\to\leafu{z}$ \emph{represents $\xi$} if $\flot{z}(\gamma(t))=\phi_z(t\xi)$. Since the flow is a local biholomorphism, it is clear that for such a $\xi$, there is a unique flow path representing it. Similarly, for such a $\delta$, there is a unique flow path corresponding to it.

  Let $z,w\in\wo{\frac{1}{2}\set{D}^2}{\{0\}}$ and $R,\sigma>0$. We say that $z$ and $w$ are \emph{$(R,\sigma)$-relatively close following the flow of $X$} if for all $\xi\in\DR{R}$, $\phi_z(\xi)\in\frac{3}{4}\set{D}^2$ and for $\gamma$ the flow path for $z$ and $X$ representing~$\xi$, $\flot{w}(\gamma(t))$ belongs to $\set{D}^2$ and for all $t\in\intcc{0}{1}$, $\norm{\flot{z}(\gamma(t))-\flot{w}(\gamma(t))}_{\infty}\leq\sigma\norm{\flot{z}(\gamma(t))}_{\infty}$; and if we also have the same properties when switching the roles of $z$ and $w$. Here, we have denoted by $\norm{z}_{\infty}=\max(\Cmod{z_1},\Cmod{z_2})$ for $z=(z_1,z_2)\in\set{D}^2$.
\end{defn}

We need a classification of non-degenerate singularities in dimension $2$.

\begin{thm}[Briot--Bouquet, Poincar\'e--Dulac]\label{thmBBPD} Let $a$ be a non-degenerate singularity of a foliation $\fol$ on a compact complex surface. There exist local coordinates $(z_1,z_2)\in\set{D}^2$ centered at~$a$, such that $\fol$ is generated on $\set{D}^2$ by one of the vector fields
  \[\begin{aligned} X_1&=z_1\der{}{z_1}+\lambda z_2\der{}{z_2},&\qquad &\lambda\in\set{C}^*;\\
      X_2&=z_1\der{}{z_1}+\left(mz_2+\mu z_1^m\right)\der{}{z_2},&\qquad &m\in\set{N}^*,~\mu\in\set{C}^*,~\Cmod{\mu}<\frac{1}{2};\\
      X_3&=z_1\der{}{z_1}-\alpha z_2\left(1+z_1z_2^{q+1}g(z_1,z_2)\right)\der{}{z_2},&\qquad &q\in\set{N},~\norm{g}_{\infty}<1,\\ &&&\alpha\in\intoc{0}{1}\cap\intco{(q+1)^{-1}}{q^{-1}}.\end{aligned}\]
\end{thm}
In $X_3$, if $q=0$, then we just have $\alpha=1$.

\begin{proof} Let $\beta$ be a characteristic number of the singularity. If $\beta\in\wo{\set{C}^*}{\left(\set{R}_-\cup\set{N}^*\cup\frac{1}{\set{N}^*}\right)}$, then Poincar\'e linearization Theorem implies that the singularity is linearizable and $\fol$ is generated by some $X_1$. If $\beta\in\set{N}^*\cup\frac{1}{\set{N}^*}$, then by Poincar\'e--Dulac Theorem (see for both \cite[Chapter~I, Section~5]{Ilya}) gives either $X_1$ or $X_2$, whether the singularity is linearizable or not. Finally, if $\beta=\alpha\in\set{R}_-$, taking $\alpha^{-1}$ if necessary, we can suppose that $\alpha\in\intoc{0}{1}\cap\intco{(q+1)^{-1}}{q^{-1}}$ for some $q\in\set{N}$. By Briot--Bouquet Theorem and a refinement by Camacho--Kuiper--Palis~\cite[Lemma~7]{CKP}, we have the form $X_3$. The estimates on $\mu$ and~$g$ can be obtained by homothety or transformations of the type $(z_1,z_2)\mapsto(Az_1,z_2)$ for some $A\in\set{C}^*$.
\end{proof}

In what follows, we suppose that $\fol$ is generated on a neighbourhood of $\adh{\set{D}}^2$ by one of the vector fields $X=X_j$, for $j\in\{1,2,3\}$. If $j=1$, we talk about the \emph{linearizable case}, if $j=2$ about the \emph{Poincar\'e--Dulac case} and if $j=3$ about the \emph{Briot--Bouquet case}. For $X_3$, we also consider
\[\wh{X}_3=-\frac{1}{\alpha}z_1\left(1+z_1z_2^{q+1}g(z_1,z_2)\right)^{-1}\der{}{z_1}+z_2\der{}{z_2}=-\frac{1}{\alpha}z_1\left(1+z_1z_2^{q+1}h(z_1,z_2)\right)\der{}{z_1}+z_2\der{}{z_2}.\]
Making an other homothety, we can still suppose that $\norm{h}_{\infty}<1$. Exchanging $z_1$ and $z_2$ if necessary, note that both $X_3$ and $\wh{X}_3$ are of the form
\begin{equation}\label{defwtX3}\wt{X}_3=z_1\der{}{z_1}-\alpha z_2\left(1+z_1^kz_2f(z_1,z_2)\right)\der{}{z_2},\end{equation}
with $\alpha\in\set{R}_+$, $k\in\set{N}^*$ and $k\geq\alpha$. That is, with $f=z_2^qg$ and $k=1$ for $X_3$ ; and $f=h$ and $k=q+1$ for $\wh{X}_3$. We often use only the hypothesis $k\geq\alpha$. In that case, we are able to make a unified proof of our estimates for $X_3$ and $\wh{X}_3$. The Briot--Bouquet case is basically the only one for which we specify that flow and flow paths are for $X_3$, $\wh{X}_3$ or $\wt{X}_3$ (that is, for any of $X_3$ or $\wh{X}_3$).

Now, we want to establish some first useful results that we use uniformly in all three cases and throughout our proof. Since $0$ is a non-degenerate singularity of the vector field $X$, we know that there exist constants $C_0,C_1,C_2>0$ such that
\begin{equation}\label{X(z)=z}C_0^{-1}\norm{z}_{\infty}\leq \norm{X(z)}_{\infty}\leq C_0\norm{z}_{\infty},\qquad z\in\set{D}^2;\end{equation}
\begin{equation}\label{X(z-w)=z-w}C_1^{-1}\norm{z-w}_{\infty}\leq\norm{X(z)-X(w)}_{\infty}\leq C_1\norm{z-w}_{\infty},\qquad z,w\in\set{D}^2;\end{equation}
\begin{equation}\label{eta=zlnz}C_2^{-1}\norm{z}_{\infty}\Cmod{\ln\norm{z}_{\infty}}\leq\eta(z)\leq C_2\norm{z}_{\infty}\Cmod{\ln\norm{z}_{\infty}},\qquad z\in\frac{3}{4}\set{D}^2.\end{equation}
The last inequality is a consequence of Proposition~\ref{etazlogz}.

\begin{lem}\label{Cmodgammat} Let $z$ be a point in $\wo{\frac{1}{2}\set{D}^2}{\{0\}}$ and $\gamma$ be a flow path for $z$. Suppose that $\lPC(\gamma)\leq R$. There exists a constant $C_3>0$ such that for all $t$,
  \[\Cmod{\gamma(t)}\leq C_0^{-1}\Cmod{\ln\norm{z}_{\infty}}\left(e^{C_3R}-1\right).\]
\end{lem}

\begin{proof} Consider a reparametrization $\wt{\gamma}\colon\intcc{0}{T}\to\set{C}$ of $\gamma$ such that $\Cmod{\wt{\gamma}'(t)}=1$ for all $t\in\intcc{0}{T}$. It is clear that $T\geq\sup_{u\in\intcc{0}{T}}\Cmod{\wt{\gamma}(u)}\geq\Cmod{\gamma(t)}$ for all $t$. Let us translate the bound $\lPC(\wt{\gamma})\leq R$ in terms of an integral.
  \[R\geq2\int_0^T\frac{\norm{X(\flot{z}(\wt{\gamma}(t)))}}{\eta(\flot{z}(\wt{\gamma}(t)))}dt\geq C\int_0^T\frac{dt}{\Cmod{\ln\norm{\flot{z}(\wt{\gamma}(t))}_{\infty}}},\]
  where we used first the identity $\eta^2\metPC=4\metm{M}$ and the fact that $\Cmod{\wt{\gamma}'(t)}=1$, and second the equivalence of Hermitian metrics, \eqref{X(z)=z} and~\eqref{eta=zlnz}. Now, by~\eqref{X(z)=z}, we have $\norm{\left(\flot{z}\circ\wt{\gamma}\right)'(t)}_{\infty}\leq C_0\norm{\left(\flot{z}\circ\wt{\gamma}\right)(t)}_{\infty}$. Then, by Gr\"{o}nwall Lemma, $\norm{\flot{z}(\wt{\gamma}(t))}_{\infty}\leq\norm{z}_{\infty}e^{C_0t}$. Moreover, the same argument on the reverse path $\wt{\gamma}^{-1}(t)=\wt{\gamma}(T-t)$, $t\in\intcc{0}{T}$, ensures that $\norm{\flot{z}(\wt{\gamma}(t))}_{\infty}\geq\norm{z}_{\infty}e^{-C_0t}$. Hence,
  \[R\geq C\int_0^T\frac{dt}{\Cmod{\ln\norm{z}_{\infty}}+C_0t}=\frac{C}{C_0}\ln\left(1+\frac{C_0T}{\Cmod{\ln\norm{z}_{\infty}}}\right).\]
  Thus, $\Cmod{\gamma(t)}\leq T\leq C_0^{-1}\Cmod{\ln\norm{z}_{\infty}}\left(e^{C_0C^{-1}R}-1\right)$.
\end{proof}

The next result describes the Bowen ball of a singularity. It is close to~\cite[Lemma~2.5]{DNSII}.

\begin{lem}\label{Bowcella} Let $R>0$ and $\eps\in\intoo{0}{\frac{1}{2}}$. If $0<\norm{z}_{\infty}<\exp\left(\ln(\eps)e^{C_3R}\right)$, then $\phi_z(\DR{R})\subset\eps\set{D}^2$.
\end{lem}

\begin{proof} The proof is by contradiction. Take $z'\in\phi_z(\DR{R})$ such that $\frac{1}{2}>\norm{z'}_{\infty}>\eps$. Let $\gamma\colon\intcc{0}{1}\to\set{C}$ be a flow path with respect to $z'$, of Poincar\'e length less than $R$ and such that $\flot{z'}(\gamma(1))=z$. By Lemma~\ref{Cmodgammat}, $\Cmod{\gamma(1)}\leq C_0^{-1}\Cmod{\ln\norm{z'}_{\infty}}\left(e^{C_3R}-1\right)$. On the other hand, using the same arguments as the previous lemma, we have the contradiction
  \[\norm{z}_{\infty}\geq\norm{z'}_{\infty}e^{-C_0\Cmod{\gamma(1)}}\geq\norm{z'}_{\infty}\exp\left(\ln\norm{z'}_{\infty}\left(e^{C_3R}-1\right)\right)\geq\exp\left(\ln(\eps)e^{C_3R}\right).\qedhere\]
\end{proof}

\begin{defn}For now on, fix $\eps\in\intoo{0}{\frac{1}{2}}$ and denote by $r_{\sing}(R)=\exp\left(\ln\left(\frac{\eps}{4}\right)e^{C_3R}\right)$ and $U_{\sing}(R)=r_{\sing}(R)\set{D}^2$.
\end{defn}

\subsection{Statement of the cell decomposition} \label{statementcell}We are now ready to introduce our cell decomposition. We wish to cover the bidisk by smaller bidisks such that we have a good control of the divergence of the flow in small hyperbolic time of two points in a cell. Let us take any $j\in\{1,2,3\}$ and $X=X_j$ if $j\neq3$, $X=\wt{X}_3$ if $j=3$. Fix $\eps>0$ and let $R>0$ be sufficiently large (depending on $\eps$). Let $C_4$ be a strictly positive constant that is specified by our further computations. Name $r_0=e^{-\exp(C_4R)}$, $r_n=r_0e^{n\exp(-C_4R)}$, for $n\in\intent{1}{N}$, $N=\sentp{e^{2C_4R}}$, and $\theta_k=\frac{2k\pi}{N}$, for $k\in\intent{1}{N'}$, and $N'=\sentp{4\pi e^{C_4R}}$. Recall that we denote by $\sentp{a}$ the smallest integer $k$ such that $k\geq a$, for $a\in\set{R}$. Let $D_0=r_0\set{D}$ and $D_{nk}=\disk{r_{n-1}e^{i\theta_k}}{r_n-r_{n-1}}$, for $n\in\intent{1}{N}$ and $k\in\intent{1}{N'}$. Define also the collection $\mathcal{D}=\{D_0\}\cup\{D_{nk};n\in\intent{1}{N},k\in\intent{1}{N'}\}$. It is easy to see that $\mathcal{D}$ is a covering of the unit disk if $R$ is sufficiently large. This is actually about the same covering as~\cite[p.~600]{DNSII}. See Figure~\ref{figcellBowen} for the skeleton of the cell decomposition, which we needed to transform into a covering by disks for further purposes. Similarly to~\cite[Proposition~2.7]{DNSII}, we have the following result.

\begin{figure}[htb]
  \centering
  \begin{tikzpicture}[scale=5.1,line cap=butt,line join=miter,miter limit=4.00,line width=0.4pt,draw=black]
    \begin{scope}
      \filldraw[black] (0,0) circle (.01);
      \path[fill=black, fill opacity=0.3] (0,0) circle (.11);
      \draw (0,0) circle (0.0316);
      \draw (0,0) circle (0.0422);
      \draw (0,0) circle (0.0563);
      \draw (0,0) circle (0.0752);
      \draw (0,0) circle (0.1005);
      \draw (0,0) circle (0.1342);
      \draw (0,0) circle (0.1792);
      \draw (0,0) circle (0.2393);
      \draw (0,0) circle (0.3196);
      \draw (0,0) circle (0.4269);
      \draw (0,0) circle (0.5702);
      \draw(-0.5702,0) -- (-0.0316,0);
      \draw(0.0316,0) -- (0.5702,0);
    \end{scope}
    \begin{scope}[rotate=36]
     \draw(-0.5702,0) -- (-0.0316,0);
      \draw(0.0316,0) -- (0.5702,0);  
    \end{scope}
    \begin{scope}[rotate=72]
     \draw(-0.5702,0) -- (-0.0316,0);
      \draw(0.0316,0) -- (0.5702,0);  
    \end{scope}
    \begin{scope}[rotate=108]
     \draw(-0.5702,0) -- (-0.0316,0);
      \draw(0.0316,0) -- (0.5702,0);  
    \end{scope}
    \begin{scope}[rotate=144]
      \draw(-0.5702,0) -- (-0.0316,0);
      \draw(0.0316,0) -- (0.5702,0); 
    \end{scope}
    
    \begin{scope}[xshift=25]
      \draw[->] (0,-0.6)--(0,0.6) node[above] {$\Cmod{z_2}$};
      \draw[->] (0,-0.6) -- (1.2,-0.6) node[right] {$\Cmod{z_1}$};
      \path[fill=black,fill opacity=0.3] (0,-0.6) -- (0.22,-0.6) -- (0.22,-0.38) -- (0,-0.38) --cycle;
      \draw (0.0632,-0.6) -- (0.0632,0.5404);
      \draw (0,-0.5368) -- (1.1404,-0.5368);
      \draw (0.0844,-0.6) -- (0.0844,0.5404);
      \draw (0,-0.5156) -- (1.1404,-0.5156);
      \draw (0.1126,-0.6) -- (0.1126,0.5404);
      \draw (0,-0.4874) -- (1.1404,-0.4874);
      \draw (0.1504,-0.6) -- (0.1504,0.5404);
      \draw (0,-0.4496) -- (1.1404,-0.4496);
      \draw (0.2010,-0.6) -- (0.2010,0.5404);
      \draw (0,-0.3990) -- (1.1404,-0.3990);
      \draw (0.2684,-0.6) -- (0.2684,0.5404);
      \draw (0,-0.3316) -- (1.1404,-0.3316);
      \draw (0.3584,-0.6) -- (0.3584,0.5404);
      \draw (0,-0.2416) -- (1.1404,-0.2416);
      \draw (0.4786,-0.6) -- (0.4786,0.5404);
      \draw (0,-0.1214) -- (1.1404,-0.1214);
      \draw (0.6392,-0.6) -- (0.6392,0.5404);
      \draw (0,0.03920) -- (1.1404,0.0392);
      \draw (0.8538,-0.6) -- (0.8538,0.5404);
      \draw (0,0.25380) -- (1.1404,0.2538);
      \draw (1.1404,-0.6) -- (1.1404,0.54040) -- (0,0.54040);
      \filldraw[black] (0,-0.6) node[below left] {$0$} circle (0.01); 

    \end{scope}
    \end{tikzpicture}
  \caption{Structure of the cell decomposition. On the left hand side, a slice $z_1=\cst$. The gray part is $\Cmod{z_2}\leq r_{\sing}(R)$. On the right hand side, the plane $(\Cmod{z_1},\Cmod{z_2})$. The gray part is $U_{\sing}(R)$. \label{figcellBowen}}
\end{figure}

\begin{prop}\label{propcelldec} Let $h$ be sufficiently small and $R$ sufficiently large. For $D^{(1)},D^{(2)}\in\mathcal{D}$, let $U=D^{(1)}\times D^{(2)}$ and $z,w\in2U\cap\left(\wo{\frac{1}{2}\set{D}^2}{\{0\}}\right)$. If $C_4$ is well chosen,
  \begin{enumerate}[label=(\arabic*),ref=\arabic*]
  \item\label{deccellsing} If $z$ and $w$ belong to $U_{\sing}(R)$, then $z$ and $w$ are $(R,\eps)$-close, that is $d_R(x,y)<\eps$;
  \item\label{deccellnsing}  If $z$ or $w$ does not belong to $U_{\sing}(R)$, then $z$ and $w$ are $(h,e^{-3R})$-relatively close following the flow.
  \end{enumerate}
\end{prop}

The next paragraphs are devoted to prove Proposition~\ref{propcelldec}. We need to distinguish three cases for the three vector fields $X_j$, for $j\in\{1,2,3\}$.
  
\subsection{Linearizable case}

What we show in this subsection is quite easy, but it clarifies our wishes, methods and notations for the two following subsections. Indeed, our idea for the Briot--Bouquet and Poincaré--Dulac cases is to compare them to the corresponding linearizable cases with $\lambda=-\alpha$ or $\lambda=m$.

Take the vector field $X=X_1$, with the notations of Theorem~\ref{thmBBPD}. For $z,w\in\wo{\frac{1}{2}\set{D}^2}{\{0\}}$, denote by $z(t)=(z_1(t),z_2(t))=\flot{z}(t)$ and $w(t)=(w_1(t),w_2(t))=\flot{w}(t)$ the coordinates of the flow trajectories. Set $\lambda_1=1$ and $\lambda_2=\lambda$. We have
\[z_j(t)=z_je^{\lambda_jt},\qquad w_j(t)=w_je^{\lambda_jt};\qquad j\in\{1,2\}.\]
The cell decomposition is a consequence of the following estimate, together with analogous ones in both non-linearizable cases. In this case, it is somehow a weaker version of~\cite[Proposition~2.7]{DNSII}. However, for the other two cases (see Lemmas~\ref{deccellPD} and~\ref{deccellBB} below), we would not have something as strong as what Dinh, Nguy\^{e}n and Sibony obtain in the linearizable case.

\begin{lem}\label{deccelllin} Let $h,\sigma\in\intoo{0}{1}$ be sufficiently small and~$R>0$ be sufficiently large. For two points $z,w\in\wo{\frac{1}{2}\set{D}^2}{\left(\frac{1}{2}U_{\sing}(R)\right)}$, if for each $j\in\{1,2\}$ we are in one of the following configurations,
  \begin{enumerate}[label=(C\arabic*),ref=C\arabic*]
  \item \label{Cseplin}$\Cmod{z_j},\Cmod{w_j}\leq\sigma r_{\sing}(R)^2$,
  \item \label{Cnseplin}$w_j,z_j\neq0$, $\Cmod{1-\frac{z_j}{w_j}}\leq\sigma$ and $\Cmod{1-\frac{w_j}{z_j}}\leq\sigma$,
  \end{enumerate}
  then $z$ and $w$ are $(h,\sigma)$-relatively close following the flow.
\end{lem}

\begin{proof} Since our hypotheses are symmetric in $z$ and $w$, it is sufficient to prove only the assertions about flow paths with respect to $z$. Fix $\xi\in\DR{h}$, $\gamma$ a flow path representing $\xi$ and $j\in\{1,2\}$. If~$z_j$ and $w_j$ are in configuration~\eqref{Cnseplin}, then
  \[\Cmod{z_j(\gamma(t))-w_j(\gamma(t))}=\Cmod{1-\frac{w_j}{z_j}}\Cmod{z_j(\gamma(t))}\leq\sigma\norm{z(\gamma(t))}_{\infty}.\]
  Next, suppose that $z_j$ and $w_j$ are in configuration~\eqref{Cseplin}. By Lemma~\ref{Cmodgammat} and since $h$ is small, we have $\Cmod{\gamma(t)}\leq Ch\Cmod{\ln\norm{z}_{\infty}}$. Using Gr\"{o}nwall Lemma as in the proof of Lemma~\ref{Cmodgammat}, we obtain
  \[\norm{z(\gamma(t))}_{\infty}\geq\norm{z}_{\infty}e^{-C_0\Cmod{\gamma(t)}}\geq\norm{z}_{\infty}^{1+CC_0h}\geq\left(\frac{1}{2}r_{\sing}(R)\right)^{1+CC_0h}\geq r_{\sing}(R)^{1+C'h},\]
  if~$R$ is sufficiently large. The reader should get used to this Gr\"{o}nwall Lemma in time $h\Cmod{\ln\norm{z}_{\infty}}$, for we will use it repeatedly. On the other hand, if $h$ is sufficiently small,
  \[\Cmod{z_j(\gamma(t))-w_j(\gamma(t))}\leq\left(\Cmod{z_j}+\Cmod{w_j}\right)e^{\Cmod{\lambda_j}\Cmod{\gamma(t)}}\leq2\sigma r_{\sing}(R)^{2-Ch}\leq\sigma\norm{z(\gamma(t))}_{\infty}.\qedhere\]
\end{proof}

\subsection{Poincar\'e--Dulac case} Now, suppose that $X=X_2$, with the notations of Theorem~\ref{thmBBPD}. With the same notations as before, we have the explicit flow
\[z_1(t)=z_1e^t,~z_2(t)=\left(z_2+\mu tz_1^m\right)e^{mt},\qquad w_1(t)=w_1e^t,~w_2(t)=\left(w_2+\mu tw_1^m\right)e^{mt}.\]
The analogous of Lemma~\ref{deccelllin} is the following.

\begin{lem}\label{deccellPD} Let $h,\sigma\in\intoo{0}{1}$ be sufficiently small and~$R>0$ be sufficiently large. For two points $z,w\in\wo{\frac{1}{2}\set{D}^2}{\left(\frac{1}{2}U_{\sing}(R)\right)}$, if 
  \begin{enumerate}[label=(C1.\arabic*),ref=C1.\arabic*]
  \item \label{CsepPD1}$\Cmod{z_1},\Cmod{w_1}\leq\sigma r_{\sing}(R)^2$ or
  \item \label{CnsepPD1}$w_1,z_1\neq0$, $\Cmod{1-\frac{z_1}{w_1}}\leq\Cmod{\ln\norm{w}_{\infty}}^{-1}\frac{\sigma}{4}$ and $\Cmod{1-\frac{w_1}{z_1}}\leq\Cmod{\ln\norm{z}_{\infty}}^{-1}\frac{\sigma}{4}$,
  \end{enumerate}
  and
  \begin{enumerate}[label=(C2.\arabic*),ref=C2.\arabic*]
  \item \label{cascentrPD2} $\Cmod{z_2-w_2}\leq\frac{\sigma}{2}\max\left(\Cmod{z_1}^m,\Cmod{w_1}^m\right)$ or
  \item \label{casperiphPD2} $z_2,w_2\neq0$, $\Cmod{1-\frac{z_2}{w_2}}\leq\Cmod{\ln\norm{w}_{\infty}}^{-1}\frac{\sigma}{4}$ and $\Cmod{1-\frac{w_2}{z_2}}\leq\Cmod{\ln\norm{z}_{\infty}}^{-1}\frac{\sigma}{4}$,
  \end{enumerate}
  then $z$ and $w$ are $(h,\sigma)$-relatively close following the flow.
\end{lem}

\begin{proof} Similarly to Lemma~\ref{deccelllin}, it is sufficient to prove the assertions concerning flow paths for $z$. Take also the same notations for $\xi$, $\gamma$. By the same arguments as in the linearizable case, we have $\Cmod{z_1(\gamma(t))-w_1(\gamma(t))}\leq\frac{\sigma}{4}\norm{z(\gamma(t))}_{\infty}$. Next, consider
  \begin{equation}\label{majz2-w2PD}\Cmod{z_2(\gamma(t))-w_2(\gamma(t))}\leq\Cmod{z_2-w_2}e^{m\Re(\gamma(t))}+\Cmod{\mu\gamma(t)\left(z_1^m-w_1^m\right)}e^{m\Re(\gamma(t))}.\end{equation}
  First, focus on the second term in the right hand side. If $z_1$ and $w_1$ are in configuration~\eqref{CsepPD1}, the same arguments as in the linearizable case show that it is bounded above by $\frac{\sigma}{4}\norm{z(\gamma(t))}_{\infty}$. On the other hand, if $z_1$ and $w_1$ are in configuration~\eqref{CnsepPD1}, we get
  \[\begin{aligned}\Cmod{\mu\gamma(t)}\Cmod{z_1^m-w_1^m}e^{m\Re(\gamma(t))}&\leq Ch\Cmod{\ln\norm{z}_{\infty}}\Cmod{1-\frac{w_1^m}{z_1^m}}\Cmod{z_1(\gamma(t))}^m\\
  &\leq\frac{1}{2}\Cmod{\ln\norm{z}_{\infty}}\Cmod{1-\frac{w_1}{z_1}}\Cmod{z_1(\gamma(t))},\end{aligned}\]
  if $h$ is sufficiently small. Here, we used that in configuration~\eqref{CnsepPD1}, $\Cmod{w_1}\leq2\Cmod{z_1}$. In any case, the second term of~\eqref{majz2-w2PD} is bounded above by $\frac{\sigma}{4}\norm{z(\gamma(t))}_{\infty}$. Now, consider the first term. If $z_2$ and $w_2$ are in configuration~\eqref{cascentrPD2} and $z_1$ and $w_1$ are in configuration~\eqref{CnsepPD1}, then
  \[\Cmod{z_2-w_2}e^{m\Re(\gamma(t))}\leq\frac{\Cmod{z_2-w_2}}{\Cmod{z_1}^m}\Cmod{z_1(\gamma(t))}^m\leq\frac{3\sigma}{4}\norm{z(\gamma(t))}_{\infty}.\]
  For the last inequality, we used that $\Cmod{z_1}^m\geq\frac{2}{3}\Cmod{w_1}^m$ in configuration~\eqref{CnsepPD1}, if~$\sigma$ is sufficiently small. If $z_2$ and $w_2$ are in configuration~\eqref{cascentrPD2} and $z_1$, $w_1$ are in configuration~\eqref{CsepPD1}, we have
  \[\Cmod{z_2-w_2}e^{m\Re(\gamma(t))}\leq r_{\sing}(R)^{(2-Ch)m}\leq\frac{3\sigma}{4}\norm{z(\gamma(t))}_{\infty},\]
  if~$h$ is  small enough and~$R$ is large enough, and by the usual Gr\"{o}nwall Lemma in time~$h\Cmod{\ln\norm{z}_{\infty}}$.

  Next, suppose that $z_2$ and $w_2$ are in configuration~\eqref{casperiphPD2}. We distinguish two cases.
  \begin{enumerate}[label=(\roman*),ref=\roman*]
  \item $\Cmod{z_2+\mu\gamma(t)z_1^m}\leq\frac{\Cmod{z_2}}{2}$. In particular, $z_1\neq0$ and $\Cmod{\mu\gamma(t)z_1^m}\geq\frac{\Cmod{z_2}}{2}$. Hence,
    \[\Cmod{z_2-w_2}e^{m\Re(\gamma(t))}=\Cmod{z_2-w_2}\frac{\Cmod{z_1(\gamma(t))}^m}{\Cmod{z_1}^m}\leq2\Cmod{1-\frac{w_2}{z_2}}\Cmod{\mu\gamma(t)}\Cmod{z_1(\gamma(t))}^m\leq\frac{\sigma}{2}\Cmod{z_1(\gamma(t))},\]
    if $h$ is sufficiently small, using Lemma~\ref{Cmodgammat} and $\Cmod{z_1(\gamma(t))}\leq\frac{3}{4}$.
  \item $\Cmod{z_2+\mu\gamma(t)z_1^m}>\frac{\Cmod{z_2}}{2}$. In particular, $\Cmod{z_2(\gamma(t))}\geq\frac{\Cmod{z_2}}{2}e^{m\Re(\gamma(t))}$. Thus,
    \[\Cmod{z_2-w_2}e^{m\Re(\gamma(t))}\leq2\Cmod{1-\frac{w_2}{z_2}}\Cmod{z_2(\gamma(t))}.\]
  \end{enumerate}
  In any case, we get $\Cmod{z_2(\gamma(t))-w_2(\gamma(t))}\leq\sigma\norm{z(\gamma(t))}_{\infty}$.
\end{proof}

\subsection{Briot--Bouquet case} \label{subsecBB}

This is the most delicate one. With the notations of~\eqref{defwtX3}, consider $X=\wt{X}_3$, i.e. both $X_3$ and $\wh{X}_3$. Note $z(t)=(z_1(t),z_2(t))=\flot{z}(t)$, and also $\wt{z}_2(t)=z_2(t)e^{\alpha t}$ to compare it with the corresponding linearizable case, where we would have $\wt{z}_2(t)$ being constant equal to $z_2$. Take also the same notations for $w$. We begin by some study of $z$ alone.

\begin{lem}\label{majwtz2'} With the notations above, if $\Cmod{z_1(t)}\leq\frac{3}{4}$,
  \[\Cmod{\wt{z}_2'(t)}\leq\alpha\left(\frac{3}{4}\right)^{k-\alpha}\Cmod{z_1}^{\alpha}\Cmod{\wt{z}_2(t)}^2.\]
\end{lem}

\begin{proof} By a simple computation, we get
  \begin{equation}\label{exprwtz2'}\wt{z}_2'(t)=-\alpha\wt{z}_2(t)z_1(t)^kz_2(t)f(z_1(t),z_2(t))=-\alpha\wt{z}_2(t)^2z_1^ke^{(k-\alpha)t}f(z_1(t),z_2(t)).\end{equation}
  Now, since $k\geq\alpha$ and $\frac{3}{4}\geq\Cmod{z_1(t)}=\Cmod{z_1}e^{\Re(t)}$, we obtain our result.
\end{proof}

\begin{lem}\label{majwtz2} If $h>0$ is sufficiently small, then for any $z\in\wo{\frac{1}{2}\set{D}^2}{\{0\}}$ and any flow path $\gamma$ such that $\lPC(\gamma)\leq h$, $\frac{1}{2}\Cmod{z_2}\leq\Cmod{\wt{z}_2(\gamma(t))}\leq2\Cmod{z_2}$.
\end{lem}

\begin{proof} Since $h$ is small and by Lemma~\ref{Cmodgammat}, $\Cmod{\gamma(t)}\leq Ch\Cmod{\ln\norm{z}_{\infty}}$. Integrating the inequality of Lemma~\ref{majwtz2'} along a radius gives for $t\neq0$,
  \[\Cmod{\wt{z}_2(t)}\leq\Cmod{z_2}+C\int_0^{\Cmod{t}}\Cmod{z_1}^{\alpha}\Cmod{\wt{z}_2\left(s\frac{t}{\Cmod{t}}\right)}^2ds.\]
    We want to apply Proposition~\ref{Gronwall}. With its notations, $F(s,x)=C\Cmod{z_1}^{\alpha}x^2$ and the unique solution~$y(s)$ is given by $\frac{x_0}{1-C\Cmod{z_1}^{\alpha}x_0s}$. Hence,
    \[\Cmod{\wt{z}_2(\gamma(t))}\leq\frac{\Cmod{z_2}}{1-C\Cmod{z_1}^{\alpha}\Cmod{z_2}\Cmod{\gamma(t)}}\leq\frac{\Cmod{z_2}}{1-Ch\Cmod{\ln\norm{z}_{\infty}}\Cmod{z_1}^{\alpha}\Cmod{z_2}}.\]
    Since $x\mapsto x\ln x$ is bounded on $\intcc{0}{\frac{3}{4}}$, we get $\Cmod{\wt{z}_2(\gamma(t))}\leq2\Cmod{z_2}$ for $h$ small enough. We argue the same on the reverse path from $z_2(\gamma(t))$ to~$z_2$ to obtain the other inequality.
  \end{proof}

  We can control the distance between $z_2(\gamma(t))$ and $w_2(\gamma(t))$ in small hyperbolic time.

  \begin{lem}\label{majzt-wt} Let $h>0$ be small enough and $z,w\in\wo{\frac{1}{2}\set{D}^2}{\{0\}}$ be such that $\norm{z}_{\infty}\leq2\norm{w}_{\infty}$. Let $\gamma\colon\intcc{0}{T}\to\set{C}$ be a flow path for $z$ with $\lPC(\gamma)\leq h$. Then,
    \[\Cmod{\wt{z}_2(\gamma(t))-\wt{w}_2(\gamma(t))}\leq2\Cmod{z_2-w_2}+\Cmod{z_2}\sup\limits_{u\in\intcc{0}{T}}\Cmod{z_1(\gamma(u))-w_1(\gamma(u))}.\]
  \end{lem}

  \begin{proof} Let us bound $\Cmod{\wt{z}_2'(t)-\wt{w}_2'(t)}$ by~\eqref{exprwtz2'}.
    \[\Cmod{\wt{z}_2'(t)-\wt{w}_2'(t)}=\alpha e^{\alpha\Re(t)}\Cmod{z_1(t)^kz_2(t)^2f(z_1(t),z_2(t))-w_1(t)^kw_2(t)^2f(w_1(t),w_2(t))}.\]
    Name $g(a,b)=a^kb^2f(a,b)$. It is clear that there exists a constant $C>0$ such that $\Cmod{\der{g}{a}}\leq C\Cmod{a}^{k-1}\Cmod{b}^2$ and $\Cmod{\der{g}{b}}\leq C\Cmod{a}^k\Cmod{b}$. Integrating along direction $a$ and then $b$, we get
    \[\begin{aligned}\Cmod{\wt{z}_2'(t)-\wt{w}_2'(t)}\leq&Ce^{\alpha\Re(t)}(\Cmod{z_2(t)}^2\max\left(\Cmod{z_1(t)},\Cmod{w_1(t)}\right)^{k-1}\Cmod{z_1(t)-w_1(t)}\\
          &+\Cmod{w_1(t)}^k\max\left(\Cmod{z_2(t)},\Cmod{w_2(t)}\right)\Cmod{z_2(t)-w_2(t)}),\\
          \leq&C\left(\Cmod{z_2(t)}\Cmod{\wt{z}_2(t)}\Cmod{z_1(t)-w_1(t)}+\Cmod{w_1}^{\alpha}\max\left(\Cmod{z_2},\Cmod{w_2}\right)\Cmod{\wt{z}_2(t)-\wt{w}_2(t)}\right),\end{aligned}\]
      if $\Cmod{t}\leq Ch\Cmod{\ln\norm{z}_{\infty}}$ and $h$ is sufficiently small. Here, we used first that $\flot{z}(t),\flot{w}(t)$ stay in $\frac{3}{4}\set{D}^2$ if~$h$ is sufficiently small, since $\ln\norm{z}_{\infty}$ and $\ln\norm{w}_{\infty}$ have bounded quotient; and the same trick as in Lemma~\ref{majwtz2'}. Furthermore, by Gr\"{o}nwall Lemma in logarithmic time, $\Cmod{z_2(t)}\leq\Cmod{z_2}^{\frac{1}{2}}$ and by Lemma~\ref{majwtz2}, $\Cmod{\wt{z}_2(t)}\leq2\Cmod{z_2}$ if $h$ is sufficiently small. Thus,
      \begin{equation}\label{majwtz'-wtw'}\Cmod{\wt{z}_2'(t)-\wt{w}_2'(t)}\leq C\left(\Cmod{z_2}^{\frac{3}{2}}\Cmod{z_1(t)-w_1(t)}+\Cmod{w_1}^{\alpha}\max\left(\Cmod{z_2},\Cmod{w_2}\right)\Cmod{\wt{z}_2(t)-\wt{w}_2(t)}\right).\end{equation}
        Applying the last inequality to a reparametrization $\wt{\gamma}$ of $\gamma$ such that $\Cmod{\wt{\gamma}'(u)}=1$, integrating along~$\wt{\gamma}$ and applying Proposition~\ref{Gronwall}, we obtain for $\beta=C\Cmod{w_1}^{\alpha}\max\left(\Cmod{z_2},\Cmod{w_2}\right)$,
      \[\Cmod{\wt{z}_2(\gamma(t))-\wt{w}_2(\gamma(t))}\leq\left(\Cmod{z_2-w_2}+C\Cmod{z_2}^{\frac{3}{2}}\sup\limits_{\gamma}\Cmod{z_1(\cdot)-w_1(\cdot)}\int_{0}^{\Cmod{\gamma(t)}}e^{-\beta s}ds\right)e^{\beta\Cmod{\gamma(t)}}.\]
      Now, since $\norm{w}_{\infty}\leq 2\norm{z}_{\infty}$, we have $\beta\Cmod{\gamma(t)}\leq Ch\Cmod{w_1}^{\alpha}$. Hence, if~$h$ is sufficiently small, $e^{\beta\Cmod{\gamma(t)}}\leq2$. For $I=e^{\beta\Cmod{\gamma(t)}}\int_0^{\Cmod{\gamma(t)}}e^{-\beta s}ds$, we also deduce $I\leq C\Cmod{\gamma(t)}$. Finally, it is quite clear that $\Cmod{z_2}^{\frac{1}{2}}\Cmod{\gamma(t)}\leq Ch$ so we conclude by putting together all these observations.
    \end{proof}

    This enables us to prove a similar result to the other cases.

    \begin{lem}\label{deccellBB} Let $h,\sigma\in\intoo{0}{1}$ be sufficiently small and $R>0$ be sufficiently large. For $z,w\in\wo{\frac{1}{2}\set{D}^2}{\left(\frac{1}{2}U_{\sing}(R)\right)}$, if for each $j\in\{1,2\}$ we are in one of the following configurations,
  \begin{enumerate}[label=(C\arabic*),ref=C\arabic*]
  \item \label{CsepBB}$\Cmod{z_j},\Cmod{w_j}\leq\sigma r_{\sing}(R)^2$,
  \item \label{CnsepBB}$w_j,z_j\neq0$, $\Cmod{1-\frac{z_j}{w_j}}\leq\frac{\sigma}{8}$ and $\Cmod{1-\frac{w_j}{z_j}}\leq\frac{\sigma}{8}$,
  \end{enumerate}
  then $z$ and $w$ are $(h,\sigma)$-relatively close following the flow of $\wt{X}_3$.
\end{lem}

\begin{proof} As usual, we only deal with flow paths for $z$. We also keep the same notations as Lemmas~\ref{deccelllin} and~\ref{deccellPD}. By the same arguments as in the linearizable case, we still have the inequalities $\Cmod{z_1(\gamma(t))-w_1(\gamma(t))}\leq\frac{\sigma}{4}\norm{z(\gamma(t))}_{\infty}$, and $\Cmod{z_2(\gamma(t))-w_2(\gamma(t))}\leq\sigma\norm{z(\gamma(t))}_{\infty}$ if $z_2$ and $w_2$ are in configuration~\eqref{CsepBB}. Now, suppose that $z_2$ and $w_2$ are in configuration~\eqref{CnsepBB}. By Lemmas ~\ref{majwtz2} and~\ref{majzt-wt},
  \[\Cmod{\frac{z_2(\gamma(t))-w_2(\gamma(t))}{z_2(\gamma(t))}}=\Cmod{\frac{\wt{z}_2(\gamma(t))-\wt{w}_2(\gamma(t))}{\wt{z}_2(\gamma(t))}}\leq4\Cmod{1-\frac{w_2}{z_2}}+\frac{\sigma}{2}\leq\sigma.\qedhere\]
\end{proof}

\subsection{Proof of the cell decomposition} \label{subseccell}

\begin{proof}[Proof of Proposition~\ref{propcelldec}] Point~\eqref{deccellsing} is a direct consequence of Lemma~\ref{Bowcella} and the definition of $U_{\sing}(R)$. Let us suppose by symmetry that $z\notin U_{\sing}(R)$. It is clear that if $C_4>C_3$ and $R$ is sufficiently large, then $w\notin\frac{1}{2}U_{\sing}(R)$. Similarly, if $C_4>C_3$ and $D^{(j)}=D_0$, $z_j$ and $w_j$ are in configuration~\eqref{Cseplin} of Lemmas~\ref{deccelllin},~\ref{deccellPD} and~\ref{deccellBB}. On the other hand, if $D^{(j)}$ is $D_{nk}$, then $\Cmod{1-\frac{z_j}{w_j}},\Cmod{1-\frac{w_j}{z_j}}\leq Ce^{-C_4R}$. Since $z,w\notin\frac{1}{2}U_{\sing}(R)$, $z_j$ and $w_j$ are in configuration~\eqref{Cnseplin} of the three lemmas if $C_4>C_3+3$. We conclude by applying them.
\end{proof}

\begin{rem} Actually, configuration~\eqref{cascentrPD2} in the Poincar\'e--Dulac case is most of the time far weaker than the configurations corresponding to separatrices in any other cases. On some sublevel $\{\Cmod{z_2}\leq C\Cmod{z_1}^m\}$, we can replace the disks in the second coordinate by disks of radius $e^{-3R}\Cmod{z_1}^m$. We do so later.
\end{rem}    

\section{General strategy and reductions}\label{secreduc}

\subsection{Geometric setup}

First, let us describe the general geometric assumptions we make to simplify our arguments. Since the entropy $h(\fol)$ does not depend on the choice of the Hermitian metric $\metm{M}$, we build one that satisfies some suitable conditions. Let $(U_r,U_a)_{r\in\mathcal{R},a\in\mani{E}}$ be a finite open covering of $\mani{M}$ by
\begin{itemize}
\item Singular flow boxes $U_a\simeq\set{D}^2$ such that $\fol$ is generated on a neighbourhood of $\adh{U}_a$ by one of the vector fields $X_j$, for $j\in\{1,2,3\}$. We also suppose that the Hermitian metric $\metm{M}$ is given on $U_a$ by $\norm{dz}^2$.
\item Regular flow boxes $U_r\simeq\set{D}\times\set{T}_r$ such that $2U_r$ is still a flow box and $2\adh{U}_r\cap\mani{E}=\emptyset$. Here $\set{T}_r\simeq\set{D}$ is just a transversal. We often identify $\set{T}_r$ with $\{0\}\times\set{T}_r$. We also suppose that the regular flow boxes cover $\wo{\mani{M}}{\left(\cup_{a\in\mani{E}}\rho U_a\right)}$. Here, $\rho>0$ is fixed below, sufficiently small to simplify some arguments. For this section, we need $\rho<\frac{1}{4}$.
\end{itemize}

\subsection{Reduction to studying orthogonal projections}

We want to do some reductions to a criterion involving an orthogonal projection. Let us begin by the following.

\begin{prop}[Dinh--Nguy\^{e}n--Sibony~{\cite[Proposition~4.1]{DNSII}}] Denote by $\set{T}=\cup_{r\in\mathcal{R}}\set{T}_r$. If $h(\set{T})<\infty$, then $h(\fol)<\infty$.
\end{prop}

Whereas the three authors prove it in the setup of linearizable singularities, it is implicit in the proof of~\cite[Lemma~4.3]{DNSII} that it is enough to prove the following lemma.

\begin{lem} Let $R,\eps>0$ and $x\in\manis{M}{E}$ be such that $\phi_x\left(\DR{2R}\right)\subset\frac{1}{2}U_a$. If $R$ is sufficiently large, then $\phi_x\left(\DR{R}\right)\subset\frac{\eps}{2}U_a$.
\end{lem}

\begin{proof} We prove that for $\eps>0$, there exists $K_{\eps}>0$ such that for any $z\in U_a$ with $\norm{z}_{\infty}\geq\eps$, there is $z'\in\left(\wo{U_a}{\frac{1}{2}U_a}\right)\cap\leafu{z}$ with $\dPC{z}{z'}\leq K_{\eps}$. It is easy to see that it implies the lemma. We have to distinguish the vector field we are dealing with.

  \emph{Briot--Bouquet or linearizable case.} Note $z=(z_1,z_2)$. Let $j\in\{1,2\}$ be the coordinate such that $\Cmod{z_j}=\norm{z}_{\infty}$. By symmetry, considering $\wh{X}_3$ if necessary, we can suppose that $j=1$. Then, $\flot{z}(t)$, for $t\in\set{R}_+$, stays in $\leafu{z}$, and has its first coordinate equal to $z_1e^t$. Hence, it reaches $\wo{U_a}{\frac{1}{2}U_a}$ in flow time at most $\ln\frac{1}{2\Cmod{z_1}}=\ln\frac{1}{2\norm{z}_\infty}\leq\ln\frac{1}{2\eps}$ and by~\eqref{eta=zlnz} in hyperbolic time less than $\frac{C}{\eps}$, for some $C>0$.
  
  \emph{Poincar\'e--Dulac case.} With the same notations, if $\Cmod{z_1}\geq\Cmod{z_2}^2$, then $\Cmod{z_1}\geq\eps^2$ and we argue similarly. If $\Cmod{z_2}^2>\Cmod{z_1}$, then $\frac{1}{m}\Cmod{\ln\Cmod{z_2}}\Cmod{\mu}\Cmod{z_1}^m\leq\frac{1}{2}\Cmod{z_2}$. In particular, $\flot{z}(t)$ escapes~$\frac{1}{2}U_a$ in positive real time $\frac{1}{m}\Cmod{\ln\Cmod{z_2}}\leq\frac{1}{m}\Cmod{\ln\eps}$. We conclude by the same observations.
\end{proof}

Let us recall a notion of~\cite[p.~573]{DNSI} that clarifies our work on the orthogonal projection.
\begin{defn}
  Let $R,\delta>0$ be such that $\delta\leq e^{-2R}$. Two points $x,y\in\manis{M}{E}$ are said to be \emph{$(R,\delta)$-conformally close} if the following properties, and the same when exchanging the roles of~$x$ and~$y$, are satisfied.
  \begin{enumerate}[label=(\alph*),ref=\alph*]
  \item \label{condpsiRdel}There exists a smooth function $\psi\colon\adhDR{R}\to\leafu{y}$ without critical points such that $\dhermf{y}{\psi(0)}{y}\leq\delta$ and $\dhimps{\adhDR{R}}{\psi}{\phi_x}\leq\delta$.
  \item \label{condPsiRdel}$\norm{d\psi}_{\infty}\leq2c_0$ for a constant $c_0$ such that $\eta\leq c_0$ on $\mani{M}$, and the norm is considered for the Poincar\'e metric at the source $\adhDR{R}$ and $\metm{M}$ on the goal $\leafu{y}$.
  \item \label{condy'Rdel}Denote by $y'=\psi(0)$. There exists a map $\Psi\colon\adhDR{R}\to\set{D}$ such that $\Psi(0)=0$, $\phi_{y'}\circ\Psi=\psi$ and the Beltrami coefficient $\mu_{\Psi}$ satisfies $\norm{\mu_{\Psi}}_{\class{1}}\leq\delta$.
  \end{enumerate}
  Recall that the \emph{Beltrami coefficient} is defined to have $\bder{\Psi}{t}=\mu_{\Psi}\der{\Psi}{t}$.
\end{defn}

What interests us with this notion is the following lemma. What is hidden behind is that we correct~$\Psi$ into a close holomorphic map by solving a Beltrami equation, and then correct this holomorphic map into a rotation.

\begin{lem}[Dinh--Nguy\^{e}n--Sibony~{\cite[Proposition~3.6]{DNSII}}]\label{confcloseimpRclose} There exists a constant $C>0$ such that if $R$ is large enough and $x,y\in\set{T}$ are $(R,e^{-2R})$-conformally close, then they are $\left(R/3,Ce^{-R/3}\right)$-close, \emph{i.e.} $d_{R/3}(x,y)\leq Ce^{-R/3}$.
\end{lem}



  Now, let us state the criterion we apply to show the finiteness of the entropy. Its proof occupies most of the end of the section.

\begin{prop}\label{reducorthproj} Let $h_1>0$ be sufficiently small, $R>0$ be sufficiently large and $(V_i)_{i\in I}$ be a covering of $\set{T}$ such that $\card(I)\leq e^{gR}$ and satisfying the following. For any $i\in I$ and $x,y\in V_i$,
  \begin{enumerate}[label=(\arabic*),ref=\arabic*]
  \item \label{dxyleqe2R} $\dhimp{x}{y}\leq e^{-2R}$,
  \item \label{Fh1dense}There exists a subset $F\subset\set{D}$ with $\DR{R}\subset\DR{h_1}(F)=\cup_{\xi\in F}\DR{h_1}(\xi)$,
  \item \label{critexorthproj}There exists a map $\psi\colon\DR{h_1}(F)\to\leafu{y}$ without critical points, such that locally near any $\xi\in\DR{h_1}(F)$, $\psi\circ\phi_x^{-1}$ is given as the orthogonal projection from $\leafu{x}$ near $\phi_x(\xi)$ to~$\leafu{y}$ near~$\psi(\xi)$ and $\dhimp{\phi_x(\xi)}{\psi(\xi)}\leq e^{-2R}$,
  \item \label{critrelclose}If $\xi\in F$ is such that $\phi_x(\xi)\in2\rho U_a$, then $\phi_x(\xi)$ and $\psi(\xi)$ are $\left(3h_1,e^{-3R}\right)$-relatively close following the flow. This property should hold for every of the three cases, but more precisely for $\wt{X}_3$ in the Briot--Bouquet case.
  \end{enumerate}
  Then, $h(\set{T})\leq3g$.
\end{prop}

Take $\alpha\in\intoo{0}{1}$ and an open covering $(V_i)_{i\in I}$ that satisfies the hypotheses of Proposition~\ref{reducorthproj} for $R(1+\alpha)$. It is enough to show that for $\eps>0$, if $x,y\in V_i$ for some $i\in I$, and~$R$ is sufficiently large, then $d_{R/3}(x,y)\leq\eps$. By Lemma~\ref{confcloseimpRclose}, it is even enough to show that $x,y$ are $(R,e^{-2R})$-conformally close. We fix $x,y\in V_i$ and $R>0$. Let us also take $F$ and~$\psi$ given by the hypotheses of Proposition~\ref{reducorthproj}. We need to control paths in singular flow boxes and therefore need to meet $F$ often enough. We use the two following lemmas.

\begin{lem}\label{critorthprojlambda} Let $a\in\mani{E}$ and $V_a$ be a connected component of $\DR{R(1+\alpha)}\cap\phi_x^{-1}\left(\rho U_a\right)$ for a sufficiently large~$R$. For $\zeta_1,\zeta_2\in V_a$, there exist paths $\lambda_1,\dots,\lambda_N\colon\intcc{0}{1}\to \DR{h_1}(F)\cap\phi_x^{-1}\left(2\rho U_a\right)$ with $\lambda_1(0)=\zeta_1$, $\lambda_N(1)=\zeta_2$, $\lambda_j(1)=\lambda_{j+1}(0)\in F$, for $j\in\intent{1}{N-1}$, $\lPC(\lambda_j)\leq3h_1$, $j\in\intent{1}{N}$ and $N\leq e^{R(1+3\alpha/2)}$.
\end{lem}

\begin{proof} Take~$G$ a maximal $\frac{h_1}{3}$-separated family in~$\DR{R(1+\alpha)}$ for the Poincaré distance. By maximality, we have $\DR{R(1+\alpha)}\subset\DR{h_1/3}(G)$. Here, we denote $\DR{h_1/3}(G)=\cup_{g\in G}\DR{h_1/3}(g)$. Moreover, since the elements of~$G$ are $\frac{h_1}{3}$-separated, the $\left(\DR{h_1/6}(g)\right)_{g\in G}$ must be pairwise disjoint and contained in $\DR{R(1+\alpha)+h_1/6}$. It follows that
  \[\card(G)\leq\frac{\Area\left(\DR{R(1+\alpha)+h_1/6}\right)}{\Area\left(\DR{h_1/6}\right)}\leq e^{R(1+3\alpha/2)},\]
  where we considered the hyperbolic area and used that~$R$ is sufficiently large. Indeed, an easy computation shows that $\Area(\DR{R})=\frac{\pi}{2}\left(e^R+e^{-R}-2\right)$. 

  Now, consider a path in~$V_a$ connecting~$\zeta_1$ and~$\zeta_2$. Cut the path into parts of length~$\frac{h_1}{3}$ and join the intermediate points to the closest points in~$G$. This gives $\lambda_1,\dots,\lambda_N$ parts of length at most~$h_1$ so that $\lambda_1$ joins $\zeta_1$ to $g_2\in G$, $\lambda_j$ joins $g_j\in G$ to $g_{j+1}\in G$, $j\in\intent{2}{N-1}$ and $\lambda_N$ joins $g_N$ to~$\zeta_2$. These paths stay at Poincaré distance at most $\frac{h_1}{3}$ of $V_a$. Since~$\fol$ is Brody-hyperbolic, if~$h_1$ is sufficiently small, then it stays in $\phi_x^{-1}\left(\frac{3\rho}{2}U_a\right)$. Moreover, getting rid of some parts if necessary, one can suppose that the $(g_i)_{i\in\intent{2}{N}}$ are all distinct. It follows that $N\leq\card(G)\leq e^{R(1+3\alpha/2)}$. Finally, we can join the $g_i$ to $\xi_i\in F$ in Poincaré distance at most $h_1$. We get $\lambda_1$ joining $\zeta_1$ to $g_2$ and then to~$\xi_2$, $\lambda_j$ joining~$\xi_j$ to $g_j$, $g_j$ to $g_{j+1}$ and $g_{j+1}$ to $\xi_{j+1}$, $j\in\intent{2}{N-1}$ and $\lambda_N$ joining $\xi_N$ to~$g_N$ and $g_N$ to~$\zeta_2$; all in Poincaré distance less than $3h_1$. Since~$\fol$ is Brody-hyperbolic, the longer paths are still contained in $\phi_x^{-1}(2\rho U_a)$ if~$h_1$ is sufficiently small and we get all the conditions of the lemma.
\end{proof}

\begin{lem}\label{critorthprojflowtime} Keep the notations of Lemma~\ref{critorthprojlambda}. Let $\lambda$ be the concatenation of $\lambda_1,\dots,\lambda_N$, $z=\phi_x(\zeta_1)$ and $w=\psi(\zeta_1)$. Suppose that~$\zeta_1\in F$. There is a flow path $\gamma$ (respectively~$\delta$) for~$z$ (respectively~$w$) such that $\phi_x(\lambda(t))=\flot{z}(\gamma(t))$ (respectively $\psi(\lambda(t))=\flot{w}(\delta(t))$) and $\Cmod{\gamma(t)-\delta(t)}\leq e^{-R(2+\alpha)}$.
\end{lem}

\begin{proof} The existence of the flow paths is a consequence of the fact that the flow is a local biholomorphism. We need to show the estimate $\Cmod{\gamma(t)-\delta(t)}\leq e^{-R(2+\alpha)}$. Let us take times $(t_{i,j})_{i\in\intent{1}{N},j\in\intent{1}{k_i}}$ so that $t_{i,k_i}=t_{i+1,1}$, $t_{i,j}<t_{i,j+1}$, $\restriction{\lambda}{\intcc{t_{i,1}}{t_{i,k_i}}}=\lambda_i$ and $\restriction{\psi\circ\lambda}{\intcc{t_{i,j}}{t_{i,j+1}}}$ is given as the local orthogonal projection from the leaf~$\leafu{x}$ near $z_{i,j}=\phi_x(\lambda(t_{i,j}))$ onto the leaf~$\leafu{y}$ near $w_{i,j}=\psi(\lambda(t_{i,j}))$. Denote this orthogonal projection as $\Phi_{i,j}$. Note $w'_{i,j}=\flot{w_{i,1}}\left(\gamma(t_{i,j})-\gamma(t_{i,1})\right)$ and $\Phi'_{i,j}$ the orthogonal projection from~$\leafu{x}$ near $z_{i,j}$ onto~$\leafu{y}$ near~$w'_{i,j}$. Since $w'_{i,1}=w_{i,1}$, $\Phi_{i,1}=\Phi'_{i,1}$ for all $i\in\intent{1}{N}$. We show by induction that if $\Phi_{i,j}$ and $\Phi'_{i,j}$ coincide, then
  \begin{equation}\label{indgamma-delta}\Cmod{(\gamma(t_{i,j})-\gamma(t_{i,1}))-(\delta(t_{i,j})-\delta(t_{i,1}))}\leq e^{-(3+5\alpha/2)R}\end{equation}
  and $\Phi_{i,j+1}$ and $\Phi'_{i,j+1}$ coincide. Recall that $z_{i,1}=\phi_x(\lambda_i(0))$, with $\lambda_i(0)\in F$ and that $w_{i,1}=\psi(\lambda_i(0))$. Then, $z_{i,1}$ and $w_{i,1}$ are $(3h_1,e^{-3(R+\alpha)})$-relatively close following the flow and $\lPC(\lambda_i)\leq3h_1$, so that $\Phi'_{i,j}$ and $\Phi'_{i,j+1}$ coincide by Lemma~\ref{orthproj}. Moreover, note that
  \[w_{i,j+1}=\flot{w_{i,1}}(\delta(t_{i,j+1})-\delta(t_{i,1}))=\flot{w'_{i,j+1}}((\delta(t_{i,j+1})-\delta(t_{i,1}))-(\gamma(t_{i,j+1})-\gamma(t_{i,1}))).\]
  Using~\eqref{indgamma-delta}, the uniqueness of the flow time in Lemma~\ref{floworthproj} and a continuity argument, we get
  \[\Cmod{(\gamma(t_{i,j+1})-\gamma(t_{i,1}))-(\delta(t_{i,j+1})-\delta(t_{i,1}))}=O\left(\frac{\norm{z_{i,j+1}-w'_{i,j+1}}}{\norm{z_{i,j+1}}}\right)\leq e^{-(3+5\alpha/2)R}.\]
  Above, the last inequality follows from the fact that $z_{i,1}$ and $w_{i,1}$ are $(3h_1,e^{-3(R+\alpha)})$-relatively close following the flow. The induction is complete. Finally, since $t_{i,k_i}=t_{i+1,1}$
  \[\Cmod{\delta(t)-\gamma(t)}\leq\sum_{i=1}^N\Cmod{(\delta(t_{i,k_i})-\delta(t_{i,1}))-(\gamma(t_{i,k_i})-\gamma(t_{i,1}))}\leq e^{-R(2+\alpha)},\]
  by~\eqref{indgamma-delta} and the bound for~$N$ in Lemma~\ref{critorthprojlambda}.

\end{proof}

We need to control some monodromy phenomena for the flow by the next result.

\begin{lem}\label{monoflot} There is $\eps_2>0$ such that if $z,w\in\wo{\set{D}^2}{\{0\}}$, $U_a\simeq\set{D}^2$ a linearizable or Poincar\'e--Dulac singular flow box, are such that $\norm{z-w}_{\infty}\leq\frac{1}{2}\norm{z}_{\infty}$, then the following holds. If~$t_1,t_2$, $u_1,u_2$ satisfy $\flot{z}(t_1)=\flot{z}(t_2)$, $\flot{w}(u_1)=\flot{w}(u_2)$ and $\Cmod{(t_1-t_2)-(u_1-u_2)}\leq\eps_2$, then we have $t_1-t_2=u_1-u_2$.

  For a Briot--Bouquet singularity, we have the same result for at least one of the vector fields $X_3$ or $\wh{X}_3$, depending on $z,w$ (but not on $t_1,t_2,u_1,u_2$).
\end{lem}

\begin{proof} Name $z=(z_1,z_2)$ and $w=(w_1,w_2)$. Note that the hypothesis $\norm{z-w}_{\infty}\leq\frac{1}{2}\norm{z}_{\infty}$ implies that there is a coordinate $j\in\{1,2\}$ with $z_j,w_j\neq0$. Without loss of generality, we suppose in the linearizable case that $j=1$. Since we can choose $X_3$ or $\wh{X}_3$, we do the same in the Briot--Bouquet case. The explicit form of the flow shows that $t_1-t_2,u_1-u_2\in2i\pi\set{Z}$. We conclude for $\eps_2<2\pi$. We argue similarly in the Poincar\'e--Dulac case. 
\end{proof}

\begin{proof}[End of proof of Proposition~\ref{reducorthproj}] First note that the conditions on~$\psi$ so that~$x$ and~$y$ are $(R,e^{-2R})$-conformally close are essentially satisfied far from the singular set and that our hypotheses are symmetric in~$x$ and~$y$. Indeed~\eqref{dxyleqe2R} and~\eqref{critexorthproj} give~\eqref{condpsiRdel} by Lemma~\ref{orthproj}, since~$\set{T}$ is far from the singular set. Condition~\eqref{condPsiRdel} follows from
  \[\norm{d\psi}_{\infty}\leq\norm{d\Phi_{\phi_x(\xi)\psi(\xi)}}_{\infty}\norm{d\phi_x}_{\infty}\leq2c_0,\]
  where the norm of $d\phi_x$ is computed with respect to the Poincaré metric at the source and the Hermitian metric at the goal, so that $\norm{d\phi_x}_{\infty}\leq c_0$ and the orthogonal projection is close to the identity, so that~$\norm{d\Phi_{\phi_x(\xi)\psi(\xi)}}_{\infty}\leq2$. The existence of~$\Psi$ in~\eqref{condy'Rdel} is just a lifting property. For the Beltrami coefficient, we refer the reader to~\cite[p.~19]{Bac1}, where it is explicitely computed. It is a $O\left(\frac{\dhimp{\phi_x(\xi)}{\psi(\xi)}}{\dhimpsing{E}{\phi_x(\xi)}^2}\right)$ in $\class{1}$-norm, which far from the singular set is a $O\left(e^{-2R(1+\alpha)}\right)$. Therefore, the only obstruction for being $(R,e^{-2R})$-conformally close is the control of the Beltrami coefficient near the singular set.

  We want to correct~$\psi$ into~$\wt{\psi}$, which coincides with~$\psi$ far from the singular set and is holomorphic near the singular set, and makes $x$ and $y$ $(R,e^{-2R})$-conformally close. Consider~$V_a$ a connected component of $\DR{R}\cap\phi_x^{-1}(\rho U_a)$ such that $V_a\cap\phi_x^{-1}\left(\frac{\rho}{4}U_a\right)\neq\emptyset$. For each of these~$V_a$, and each~$a\in\mani{E}$, we will make this correction so that $\wt{\psi}$ coincides with~$\psi$ on $V_a\cap\phi_x^{-1}\left(\wo{\rho U_a}{\frac{\rho}{2}U_a}\right)$ and is holomorphic on $V_a\cap\phi_x^{-1}\left(\frac{\rho}{4}U_a\right)$. Fix any $\xi\in F\cap V_a$. Such a $\xi$ exists because $\fol$ is supposed to be Brody-hyperbolic, $V_a\cap\phi_x^{-1}\left(\frac{\rho}{2}U_a\right)\neq\emptyset$ and~$h_1$ is small. Set $z=\phi_x(\xi)$ and $w=\psi(\xi)$. Let $\zeta\in V_a$ and take paths $\lambda\colon\intcc{0}{1}\to\set{D}$ from~$\xi$ to~$\zeta$ given by Lemma~\ref{critorthprojlambda} and $\gamma,\delta$ given by Lemma~\ref{critorthprojflowtime}. For the Briot--Bouquet case, we choose flow paths for the vector field that satisfies Lemma~\ref{monoflot} for $z$ and $w$. Let $\chi\colon\intcc{0}{1}\to\intcc{0}{1}$ be a smooth function such that $\chi=0$ on $\intcc{\frac{\rho}{2}}{1}$ and $\chi=1$ on $\intcc{0}{\frac{\rho}{4}}$. Define
    \[\wt{\psi}(\zeta)=\flot{w}\left(\chi\left(\norm{\flot{z}(\gamma(1))}\right)\left(\gamma(1)-\delta(1)\right)+\delta(1)\right).\]
    This definition does not depend on the choice of~$\lambda$. Indeed, if we had taken another path~$\lambda_2$ and flow paths~$\gamma_2,\delta_2$, then we would have $\flot{z}(\gamma(1))=\flot{z}(\gamma_2(1))=\phi_x(\zeta)$ and $\flot{w}(\delta(1))=\flot{w}(\delta_2(1))=\psi(\zeta)$. This implies $\gamma_2(1)-\delta_2(1)=\gamma(1)-\delta(1)$, by Lemmas~\ref{critorthprojflowtime} and~\ref{monoflot}. Since
    \[\wt{\psi}(\zeta)=\flot{\psi(\zeta)}\left(\chi\left(\norm{\phi_x(\zeta)}\right)\left(\gamma(1)-\delta(1)\right)\right),\]
    we would obtain the same for $\lambda_2,\gamma_2,\delta_2$.

    At distance $\geq\frac{\rho}{4}$ from the singularities, the map~$\wt{\psi}$ is $O\left(e^{-(2+\alpha)R}\right)$-close to~$\psi$ for $\class{2}$-norm. It is holomorphic near the singular set so that $\mu_{\Psi}=0$ on $\phi_x^{-1}\left(\frac{\rho}{4}U_a\right)$. Therefore, it satisies~\eqref{condy'Rdel}. Since it is $O\left(e^{-(2+\alpha)R}\right)$-close to~$\psi$ in $\class{0}$-norm, it satisfies~\eqref{condpsiRdel}. Condition~\eqref{condPsiRdel} is still clear. Therefore,~$x$ and~$y$ are $(R,e^{-2R})$-conformally close and we conclude.
  \end{proof}

  Above, the correction process for~$\psi$ is close to~\cite[Lemma~2.12]{DNSII}, but its proof is more difficult in our context, because it is less explicit. 
  To ensure that a $h_1$-dense subset $F$ has its image in an appropriate cell, we use the following refinement result, due to Dinh, Nguy\^{e}n and Sibony, on holonomy mappings. This result shows that the  cardinal of a refinement of coverings grows at most exponentially with respect to the number of coverings.

  \begin{lem}[Dinh--Nguy\^{e}n--Sibony~{\cite[Lemma~4.4]{DNSII}}]\label{lemref} Let $\Omega$ be a subset of a finite union of copies of $\set{C}$. Let $\mathcal{V}_1,\dots,\mathcal{V}_n$ be coverings of $\Omega$ by disks, all satisfying $\card\mathcal{V}_i\leq K$, for some uniform $K>0$. Then, there exists a covering $\mathcal{V}$ of $\Omega$ by disks, with $\card\mathcal{V}\leq200^nK$ such that for any $D\in\mathcal{V}$, there are $D_i\in\mathcal{V}_i$, $i\in\intent{1}{n}$, satisfying $2D\subset2D_1\cap\dots\cap2D_n$.
  \end{lem}

\section{Mesh of transversals and initial covering}\label{sectrans}

\subsection{Mesh of transversals} At the end of our argument, we wish to apply Proposition~\ref{reducorthproj}. Hence, we need to build a covering of $\set{T}$, the union of some regular transversals. However, to build such a covering, we also need to control how leaves from these transversals behave when approaching the singular set. Therefore, we also build a mesh of transversals, covering almost all the manifold, which is in some sense hyperbolically dense in the foliation. That way, we are able to control at the beginning what happens to leaves in small hyperbolic time near all of these transversals, and then refine to control what happens all along the orbit in large hyperbolic time.

Let us fix some constant $h_1>0$ which is sufficiently small to have all the results of Section~\ref{seclocal} for $h=3h_1$. We also take a constant $K\in\intoo{0}{1}$ and name $\hbar=K h_1$. The constant $K$ is fixed in the following sections. Note that since we are only concerned in what happens to points in $\set{T}$ in time~$R$, we let $\eps=\dhimpsing{E}{\set{T}}$ and $r_{\sing}(R)=\exp\left(\ln\left(\frac{\eps}{4}\right)e^{C_3R}\right)$. Hence, $\phi_x\left(\DR{R}\right)\cap B\left(a,r_{\sing}(R)\right)=\emptyset$, for $a\in\mani{E}$ and $x\in\set{T}$, by Lemma~\ref{Bowcella}. First, let us state what we ask of our mesh of transversals.

  \begin{prop}\label{propT} There exists a constant $K'>0$ such that for all $R>0$ sufficiently large, there exists a mesh of transversals $\wt{\set{T}}=\left(\set{T}_i\right)_{i\in I_{\set{T}}}$ satisfying
    \begin{enumerate}[label=(HT\arabic*),ref=HT\arabic*]
    \item \label{Tdense} For $x\in\manis{M}{E}$, if $\dhimpsing{E}{x}\geq r_{\sing}(R)$, there is $y\in\left(\cup_{i\in I_{\set{T}}}\set{T}_i\right)\cap\leafu{x}$ with $\dPC{x}{y}\leq\hbar$;
    \item \label{Tsep} For $i\in I_{\set{T}}$, there are at most $K'$ elements $j\in I_{\set{T}}$ such that there are $x\in\set{T}_i$, $y\in\leafu{x}\cap\set{T}_j$ with $\dPC{x}{y}\leq2h_1$.
    \end{enumerate}
  \end{prop}

  We are inspired by the  construction of transversals in~\cite[Sections~2 and~5]{DNSII}. However, there is a difference between the structure of our argument and that  of the three authors of~\cite{DNSII}.  Namely, they build transversals for each of the singular cells and have to consider some multiplicity of certain transversals. The property~\eqref{Tsep} for all transversals enables us to avoid this because the refinement Lemma~\ref{lemref} applies directly to a bounded number of transversals. 

  In regular flow boxes $U_r\simeq\set{D}\times\set{T}_r$, the Poincar\'{e} metric is equivalent to the standard Hermitian metric. Then, we consider a lattice $\Lambda=B\hbar\left(\set{Z}+i\set{Z}\right)\cap\set{D}$ and $\set{T}_{r,\lambda}=\{\lambda\}\times\set{T}_r$, $\lambda\in\Lambda$. Sometimes we will refer to these as \emph{regular} transversals. The union of these transversals clearly satisfy~\eqref{Tdense} and~\eqref{Tsep} if $B$ is sufficiently small.

  So, the difficulty of Proposition~\ref{propT} is for the singular flow boxes. We have to distinguish different cases for the different types of singularities. First, let $a\in\mani{E}$ be a singularity of linearizable or Briot--Bouquet type. Define $r_0^{\set{T}}=r_{\sing}(R)$ and for $j\in\set{N}$, $r_j^{\set{T}}=\left(r_0^{\set{T}}\right)^{\left(1-C_5\hbar\right)^j}$, for some $C_5\in\intoo{0}{\frac{1}{\hbar}}$ that our further computation specifies. Let also $\theta_1^{\set{T}},\dots,\theta_P^{\set{T}}$ be some angles that are $C_5\hbar$-dense in $\intcc{0}{2\pi}$, for $P=\sentp{\frac{\pi}{C_5\hbar}}$. Define
  \[\set{T}_{a,j,k,u}=\left\{z=(z_1,z_2)\in\set{D}^2\simeq U_a;\,z_u=r_j^{\set{T}}e^{i\theta_k^{\set{T}}},\,\norm{z}_{\infty}\leq\frac{3}{2}r_j^{\set{T}}\right\},\]
  for $j\in\intent{0}{N_a(R)}$, $k\in\intent{1}{P}$, $u\in\{1,2\}$ and $N_a(R)=\max\left\{j\in\set{N};\,r_j^{\set{T}}\leq2\rho\right\}$. These transversals will be called \emph{singular}. Here, we need $\rho<\frac{1}{4}$ so that the biggest transversal is still contained in~$\frac{3}{4}\set{D}^2$. Let us prove~\eqref{Tdense} and~\eqref{Tsep} for these transversals. This is what is done in the following statements.

  \begin{lem}\label{linTrdense} Let $a\in\mani{E}$ be a singularity as above. If $C_5$ is sufficiently small, then for any $z\in\set{D}^2$ with $\norm{z}_{\infty}\in\intcc{r_{j-1}^{\set{T}}}{r_j^{\set{T}}}$, for some $j\in\intent{1}{N_a(R)}$, there exists $z'\in\leafu{z}$ satisfying $\norm{z'}_{\infty}=r_j^{\set{T}}$ and $\dPC{z}{z'}\leq\frac{\hbar}{2}$.
  \end{lem}

  \begin{proof} Note that here, we only care about the foliation and not on the vector field. Hence, considering $\wh{X}_3$ if necessary, we may assume that $\norm{z}_{\infty}=\Cmod{z_1}$. For the flow $\flot{z}(t)$ starting from $z$ in time $t\in\set{R}_+$, let
    \[T=\inf\left\{t\in\set{R}_+;\,\norm{\flot{z}(t)}_{\infty}=r_j^{\set{T}}\right\}.\]
    It is clear that $T\leq\ln\frac{r_j^{\set{T}}}{r_{j-1}^{\set{T}}}$. Moreover, the path $\gamma\colon\intcc{0}{T}\to\leafu{z}$ defined by $\gamma(t)=\flot{z}(t)$ satisfies $\norm{\gamma(t)}_{\infty}\leq r_j^{\set{T}}$ and $\norm{\gamma(T)}_{\infty}=r_j^{\set{T}}$. Then, its Poincar\'e length is bounded by
    \[\lPC(\gamma)=2\int_0^T\frac{\norm{\gamma'(t)}}{\eta(\gamma(t))}dt\leq\frac{CT}{\Cmod{\ln r_j^{\set{T}}}}\leq C\left(\frac{\ln r_{j-1}^{\set{T}}}{\ln r_j^{\set{T}}}-1\right)=\frac{CC_5\hbar}{1-C_5\hbar}\leq\frac{\hbar}{2},\]
    by~\eqref{X(z)=z} and~\eqref{eta=zlnz} and if $C_5$ is sufficiently small.
  \end{proof}

  Once we have reached the right norm, we need to turn around to a transversal angle.

  \begin{lem}\label{linTadense} Let $a\in\mani{E}$ be a singularity as above. If $C_5$ is sufficiently small, $z\in\set{D}^2$ and $j\in\intent{1}{N_a(R)}$ are such that $\norm{z}_{\infty}=r_j^{\set{T}}$, then there exist $k\in\intent{1}{P}$, $u\in\{1,2\}$ and a $z'\in\leafu{z}\cap\set{T}_{a,j,k,u}$ with $\dPC{z}{z'}\leq\frac{\hbar}{2}$.
  \end{lem}

  \begin{proof} Similarly to the previous lemma, we can assume that $\norm{z}_{\infty}=\Cmod{z_1}$. Note $z_1=r_j^{\set{T}}e^{i\theta}$ and let $k\in\intent{1}{P}$ be such that $\Cmod{\theta-\theta_k^{\set{T}}}\leq C_5\hbar$. Consider the path $\gamma(t)=\flot{z}\left(it\left(\theta_k^{\set{T}}-\theta\right)\right)$. We have $\gamma(0)=z$ and $z'=\gamma(1)=\left(r_j^{\set{T}}e^{i\theta_k^{\set{T}}},z_2'\right)$. By Gr\"{o}nwall Lemma, it is easy to see that $\Cmod{z_2'}\leq\frac{3}{2}r_j^{\set{T}}$. Thus, $z'\in\set{T}_{a,j,k,1}$. On the other hand, if $C_5$ is sufficiently small,
    \[\lPC(\gamma)\leq\frac{C'\Cmod{\theta-\theta_k^{\set{T}}}}{\Cmod{\ln r_j^{\set{T}}}-C\Cmod{\theta-\theta_k^{\set{T}}}}\leq\frac{\hbar}{2}.\qedhere\]
    \end{proof}

    For these singularities, the condition~\eqref{Tdense} is clear, and we need to prove~\eqref{Tsep}.

    \begin{lem}\label{linsep} Let $a\in\mani{E}$ be a singularity as above. There is a constant $K'>0$, independent on $R$ such that for any $j\in\intent{0}{N_a(R)}$, $k\in\intent{1}{P}$, $u\in\{1,2\}$, there are at most $K'$ elements $(j',k',u')\in\intent{0}{N_a(R)}\times\intent{1}{P}\times\{1,2\}$ such that there exist $x\in\set{T}_{a,j,k,u}$, $y\in\set{T}_{a,j',k',u'}\cap\leafu{x}$ with $\dPC{x}{y}\leq2h_1$.
    \end{lem}

    \begin{proof} Note that $\norm{x}_{\infty}\in\intcc{r_j^{\set{T}}}{\frac{3}{2}r_j^{\set{T}}}$. Therefore, $\norm{y}_{\infty}\in\intcc{\left(r_j^{\set{T}}\right)^{1+Ch_1}}{\left(\frac{3}{2}r_j^{\set{T}}\right)^{1-Ch_1}}$ by Gr\"{o}nwall Lemma in logarithmic time. On the other hand, $\norm{y}_{\infty}\in\intcc{r_{j'}^{\set{T}}}{\frac{3}{2}r_{j'}^{\set{T}}}$. A direct computation implies that there is a bounded number of possible $j'$ if $j$ is given. The integers $k',u'$ are also in bounded quantity.
    \end{proof}

    We finish building the singular transversals, with a singularity $a\in\mani{E}$ of Poincar\'e--Dulac type. We take the same radii $r_j^{\set{T}}=\left(r_{\sing}(R)\right)^{(1-C_5\hbar)^j}$, but this time for $j\in\set{Z}$. We also keep the same angles $\theta_k^{\set{T}}$, for $k\in\intent{1}{P}$. Consider the transversals
    \[\begin{aligned}\set{T}_{a,j,k,1}&=\left\{(z_1,z_2)\in\set{D}^2\simeq U_a;\,z_1=r_j^{\set{T}}e^{i\theta_k^{\set{T}}},\,\Cmod{z_2}\leq\frac{3}{2}\left(r_j^{\set{T}}\right)^m\Cmod{\ln r_j^{\set{T}}}\right\};\\
        \set{T}_{a,j,k,2}&=\left\{(z_1,z_2)\in\set{D}^2\simeq U_a;\,z_2=\left(r_j^{\set{T}}\right)^m\Cmod{\ln r_j^{\set{T}}}e^{i\theta_k^{\set{T}}},\,\Cmod{z_1}\leq\left(\frac{3}{2}\right)^{1/m}r_j^{\set{T}}\right\};\end{aligned}\]
    for $j\in\set{Z}$ such that $r_{\sing}(R)\leq\max\left(r_j^{\set{T}},\left(r_j^{\set{T}}\right)^m\Cmod{\ln r_j^{\set{T}}}\right)$ and $\min\left(r_j^{\set{T}},\left(r_j^{\set{T}}\right)^m\Cmod{\ln r_j^{\set{T}}}\right)\leq2\rho$. We denote the maximal $j$ by $N_a(R)$ and the minimal $j$ by $N_a'(R)$. For a unification purpose, denote by $N_a'(R)=0$ for the other singularities. Now, we need $\rho$ so that
    \[2\max\left(r_{N_a(R)}^{\set{T}},\left(r_{N_a(R)}^{\set{T}}\right)^m\Cmod{\ln r_{N_a(R)}^{\set{T}}}\right)\leq1.\]
    That way, the biggest transversal is contained in $\frac{3}{4}\set{D}^2$. The next results are analogous to Lemmas~\ref{linTrdense},~\ref{linTadense} and~\ref{linsep}.

    \begin{lem}\label{PDTrdense} Let $a\in\mani{E}$ be as above and $C_5$ be small enough. If $j\in\intent{N_a'(R)+1}{N_a(R)}$ and $z=(z_1,z_2)\in\set{D}^2$ are such that we are in one of the following configurations,
      \begin{enumerate}[label=(\arabic*),ref=\arabic*]
      \item\label{config1PDrdense} Either $\Cmod{z_1}\in\intcc{r_{j-1}^{\set{T}}}{r_j^{\set{T}}}$ and $\Cmod{z_2}\leq\left(r_j^{\set{T}}\right)^m\Cmod{\ln r_j^{\set{T}}}$,
      \item\label{config2PDrdense} Or $\Cmod{z_1}\leq r_{j-1}^{\set{T}}$ and $\Cmod{z_2}\in\intcc{\left(r_{j-1}^{\set{T}}\right)^m\Cmod{\ln r_{j-1}^{\set{T}}}}{\left(r_j^{\set{T}}\right)^m\Cmod{\ln r_j^{\set{T}}}}$.
      \end{enumerate}
      then, there is $z'=(z_1',z_2')\in\leafu{z}$, with $\dPC{z}{z'}\leq\frac{\hbar}{2}$ and one of the following configurations.
      \begin{enumerate}[label=(\roman*),ref=\roman*]
      \item\label{configiPDrdense} Either $\Cmod{z_1}=r_j^{\set{T}}$ and $\Cmod{z_2}\leq\left(r_j^{\set{T}}\right)^m\Cmod{\ln r_j^{\set{T}}}$,
      \item\label{configiiPDrdense} Or $\Cmod{z_1}\leq r_j^{\set{T}}$ and $\Cmod{z_2}=\left(r_j^{\set{T}}\right)^m\Cmod{\ln r_j^{\set{T}}}$.
      \end{enumerate}
    \end{lem}

    \begin{proof} The situation is similar to Lemma~\ref{linTrdense}, the proof of which we keep some notations. Consider the flow $\flot{z}(t)$ starting at $z$ in positive real time. Note $\flot{z}(t)=(z_1(t),z_2(t))$. Set
      \[T=\inf\left\{t\in\set{R}_+;\,\Cmod{z_1(t)}=r_j^{\set{T}}\,\text{or}\,\Cmod{z_2(t)}=\left(r_j^{\set{T}}\right)^m\Cmod{\ln r_j^{\set{T}}}\right\}.\]
      Similarly, it is enough to prove that $T\leq C\ln\frac{r_j^{\set{T}}}{r_{j-1}^{\set{T}}}$. Indeed, for $t\leq T$, $\Cmod{z_2(t)}\leq\left(r_j^{\set{T}}\right)^{\frac{1}{2}}$. In configuration~\eqref{config1PDrdense}, the same arguments still work. Let us suppose the setup of configuration~\eqref{config2PDrdense}. Denote by $t=2\ln\frac{r_j^{\set{T}}}{r_{j-1}^{\set{T}}}$. We show that $t\geq T$, ensuring that $\Cmod{z_2(t)}\geq\left(r_j^{\set{T}}\right)^m\Cmod{\ln r_j^{\set{T}}}$.
      \[\Cmod{z_2(t)}\geq\Cmod{z_2}\left(1-\frac{\Cmod{\mu tz_1^m}}{\Cmod{z_2}}\right)\left(\frac{r_j^{\set{T}}}{r_{j-1}^{\set{T}}}\right)^{2m}\geq\left(r_j^{\set{T}}\right)^m\Cmod{\ln r_{j-1}^{\set{T}}}\left(1-\frac{\ln\left(r_j^{\set{T}}/r_{j-1}^{\set{T}}\right)}{\Cmod{\ln r_{j-1}^{\set{T}}}}\right)\left(\frac{r_j^{\set{T}}}{r_{j-1}^{\set{T}}}\right)^m.\]
      Here, we have used both hypotheses of configuration~\eqref{config2PDrdense}. Now, the right hand side is equal to $\left(r_j^{\set{T}}\right)^m\Cmod{\ln r_j^{\set{T}}}\left(\frac{r_j^{\set{T}}}{r_{j-1}^{\set{T}}}\right)^m$ and we conclude.     
    \end{proof}

    Again, once the right radius found, we have to turn to the right angle.

    \begin{lem}\label{PDTadense} Let $a\in\mani{E}$ be as above. If $C_5$ is small enough and $z\in\set{D}^2$ is in configuration~\eqref{configiPDrdense} or~\eqref{configiiPDrdense}, then there are $k\in\intent{1}{P}$, $j\in\{1,2\}$ and $z'\in\leafu{z}\cap\set{T}_{a,j,k,u}$ with $\dPC{z}{z'}\leq\frac{\hbar}{2}$.
    \end{lem}

    \begin{proof} In configuration~\eqref{configiPDrdense}, the proof is the same as for other singularities. Let us suppose we are in configuration~\eqref{configiiPDrdense}. Note $z_2=\left(r_j^{\set{T}}\right)^m\Cmod{\ln r_j^{\set{T}}}e^{i\theta}$ and $k\in\intent{1}{P}$ with $\Cmod{\theta-\theta_k^{\set{T}}}\leq C_5\hbar$. We search a small $t$ such that
      \begin{equation}\label{eqPDTadense}f(t)=\left(1+e^{-i\theta}\mu z_1^m\left(r_j^{\set{T}}\right)^{-m}\Cmod{\ln r_j^{\set{T}}}^{-1}t\right)e^{mt}=e^{i\left(\theta_k^{\set{T}}-\theta\right)}.\end{equation}
      It is easy to see that $\Cmod{f'(0)}\geq m-\Cmod{\mu}\Cmod{\ln r_j^{\set{T}}}^{-1}\geq\frac{1}{2}$, shrinking again $\rho$ if necessary. It is also clear that $\Cmod{f(t)-f(0)-tf'(0)}\leq C\Cmod{t}^2$ for $\Cmod{t}\leq1$ and some uniform constant $C>0$. Hence, $f$ is injective on a small uniform disk and by Koebe $\frac{1}{4}$-Theorem,~\eqref{eqPDTadense} admits a solution $t=O\left(C_5\hbar\right)$, if~$C_5$ is small enough. We conclude the same as in Lemma~\ref{linTadense}.
    \end{proof}

    So,~\eqref{Tdense} is also true for this type of singularity. The following statement finishes the proof of Proposition~\ref{propT}, by showing~\eqref{Tsep} for Poincar\'e--Dulac singularities.

    \begin{lem}\label{PDsep} Let $a\in\mani{E}$ be a singularity as above. There is a constant $K'>0$, independent on $R$ such that for any $j\in\intent{0}{N_a(R)}$, $k\in\intent{1}{P}$, $u\in\{1,2\}$, there are at most $K'$ elements $(j',k',u')\in\intent{0}{N_a(R)}\times\intent{1}{P}\times\{1,2\}$ such that there exist $x\in\set{T}_{a,j,k,u}$, $y\in\set{T}_{a,j',k',u'}\cap\leafu{x}$ with $\dPC{x}{y}\leq2h_1$.
    \end{lem}

    \begin{proof} The arguments are similar to those of Lemma~\ref{linsep}, using that
      \[\left(r_j^{\set{T}}\right)^m\Cmod{\ln r_j^{\set{T}}}\in\intcc{C^{-1}\left(r_j^{\set{T}}\right)^m}{C\left(r_j^{\set{T}}\right)^{\frac{1}{2}}}.\qedhere\]
    \end{proof}

    Actually, we correct a bit the movement that makes us end on a transversal to control monodromy processes below. For $u\in\{1,2\}$, we consider a vector field $X_j^u$, for $j\in\{1,2,3\}$. For the linearizable case, we put $X_1^1=\der{}{z_1}+\lambda z_2\der{}{z_2}$, $X_1^2=\lambda^{-1}z_1\der{}{z_1}+z_2\der{}{z_2}$. For the Poincar\'e--Dulac case, we just put $X_2^1=X_2^2=X_2$. For the Briot--Bouquet case, we put $X_3^1=X_3$ and $X_3^2=\wh{X}_3$. When the $j\in\{1,2,3\}$ is implicitly determined by a singularity~$a$, we simply denote the vector field by~$X^u$. This enables to have the following. We use it below to obtain Lemma~\ref{tpsdeflottotIm}. Somehow, this is a way to deal with the movement to reach a transversal. Since it is not so standard (compared with what we do in Section~\ref{sechyp}), we need to control the behaviour in both coordinates. Hence, we need it not to be too neutral, and avoid turning around inside a transversal level $\{\Cmod{z_u}=r_j^{\set{T}}\}$.

    \begin{lem}\label{flowtimesmallIm} Let $z\in2\rho U_a$ be such that $\dhimpsing{E}{x}\geq r_{\sing}(R)$. There are a transversal $\set{T}_{a,j,k,u}$ and flow paths $\gamma^v\colon\intcc{0}{1}\to\set{C}$ for $X^v$, $v\in\{1,2\}$, such that $\flot{z}^1(\gamma^1(t))=\flot{z}^2(\gamma^2(t))$, where $\flot{z}^v$ denotes the flow of $X^v$, $\flot{z}^v(\gamma^v(1))\in\set{T}_{a,j,k,u}$, $\lPC(\gamma)\leq \hbar$ and
      \begin{itemize}
      \item In the linearizable case, there is a constant $c>0$ such that
        \[\Cmod{\Im(\gamma^v(1))}\leq c\Cmod{\Re(\gamma^v(1))}+O\left(C_5\hbar\right),\qquad v\in\{1,2\}.\]
      \item In the Briot--Bouquet and Poincar\'e--Dulac cases, $\Cmod{\Im(\gamma^v(1))}=O(C_5\hbar)$.
      \end{itemize}
    \end{lem}

    \begin{proof} We change slightly the processes of Lemmas~\ref{linTrdense},~\ref{linTadense},~\ref{PDTrdense} and~\ref{PDTadense}. Note that they already give us a flow path $\gamma$, with $\gamma(1)=t_1+t_2$, where $t_1$ is given by Lemma~\ref{linTrdense} or~\ref{PDTrdense}, and $t_2$ by Lemma~\ref{linTadense} or~\ref{PDTadense}. In any case, $\Cmod{t_2}=O\left(C_5\hbar\right)$ so it is enough to show that
      \begin{itemize}
      \item We can choose $\Cmod{\Re(t_1)}\geq c\Cmod{\Im(t_1)}$ and $\Cmod{\Re(\lambda t_1)}\geq c\Cmod{\Im(\lambda t_1)}$ in the linearizable case.
      \item $\Im(t_1^v)=O(C_5\hbar)$, where $t_1^v$ is the flow time corresponding to $X^v$ and Lemma~\ref{linTadense} for the Briot--Bouquet or Poincar\'e--Dulac case.
      \end{itemize}
      For the Poincar\'e--Dulac case, $t_1$ is real. For the linearizable case, we consider a complex number~$\omega$ of modulus $1$ such that $\Re(\omega)>0$ and $\Re(\lambda\omega)\neq0$. For $\flot{z}(\omega t)$ and with notations of Lemma~\ref{linTrdense}, $T\leq\frac{1}{\Re(\omega)}\ln\frac{r_j^{\set{T}}}{r_{j-1}^{\set{T}}}$. We obtain our result if $C_5$ is small enough. For the Briot--Bouquet case, our process is in real flow time for $X^1$. We have
      \[\flot{z}^1(\gamma^1(t))=\left(z_1e^{\gamma^1(t)},z_2(\gamma^1(t))\right)=\flot{z}^2(\gamma^2(t))=\left(z_1(\gamma^2(t)),z_2e^{\gamma^2(t)}\right).\]
      Note $\wt{z}_2(t)=z_2(t)e^{\alpha t}$ and $\wh{z}_2(t)=\wt{z}_2(t)-z_2$. Locally, it is clear that 
\[\gamma^2(t)=-\alpha\gamma^1(t)+\ln\left(1+\frac{\wh{z}_2(\gamma^1(t))}{z_2}\right).\]
 Since $\wh{z}_2'(t)=\wt{z}_2'(t)$, by Lemmas~\ref{majwtz2'} and~\ref{majwtz2}, we get $\Cmod{\wh{z}_2'(t)}\leq C\Cmod{z_1}^{\alpha}\Cmod{z_2}^2$. It follows that $\Cmod{\wh{z}_2(t)}$ is a $O\left(C_5\hbar\Cmod{z_2}^2\right)$ and since $C_5\hbar$ is small, that $\Cmod{\Im\left(\gamma^1(t)\right)}=O\left(C_5\hbar\right)$.
    \end{proof}   
    
    \subsection{Initial covering}

    Now, we want to build an initial covering of $\wt{\set{T}}$ on which we control the orthogonal projection and the flow in hyperbolic time $3h_1$. This is the covering that we progessively refine by Lemma~\ref{lemref} in order to apply Proposition~\ref{reducorthproj}. As we did for the transversals, let us begin by stating what we wish for our covering. See also properties (H1) and (H1)' in~\cite[p.~615]{DNSII}.

    \begin{prop}\label{propD} There exists a constant $C_6>0$ such that for $R$ sufficiently large we have the following. For each $i\in I_{\set{T}}$, there exists a covering by disks $\mathcal{V}_i$ of $\set{T}_i$ such that
      \begin{enumerate}[label=(HD\arabic*),ref=HD\arabic*]
      \item \label{regHD} If $\set{T}_i$ is a regular transversal, $D\in\mathcal{V}_i$ and $x,y\in2D$, the orthogonal projection $\Phi_{xy}$ exists on $\phi_x\left(\DR{3h_1}\right)$ and satisfies $\dhimp{x'}{\Phi_{xy}(x')}\leq e^{-2R}$, for $x'\in\phi_x\left(\DR{3h_1}\right)$.
      \item \label{singHD} If $\set{T}_i$ is a singular transversal, $D\in\mathcal{V}_i$ and $x,y\in\wo{2D}{U_{\sing}(R)}$, then the orthogonal projection $\Phi_{xy}$ exists in a neighbourhood of $x$ and satisfies $\dhimp{x}{\Phi_{xy}(x)}\leq e^{-2R}$. Moreover, $x$ and $\Phi_{xy}(x)$ are $(3h_1,e^{-3R})$-relatively close following the flow.
      \item \label{cardHD} $\max_{i\in I_{\set{T}}}\card\mathcal{V}_i\leq e^{C_6R}$.
      \end{enumerate}
    \end{prop}

    In the regular case, we can simply cover $\set{T}_{r,\Lambda}$ by $Ce^{4R}$ disks of radius $Be^{-2R}$, for $B$ sufficiently small. The difficulty is again the singular case and we address particularly different types of singularities. For a linearizable or Briot--Bouquet singularity, we consider the covering $\mathcal{D}$ of Subsection~\ref{subseccell} of a transversal $\set{T}_{a,j,k,u}$. More precisely, we consider only the disks of $\mathcal{D}$ that intersect $\set{T}_{a,j,k,u}$.

For a Poincar\'e--Dulac singularity, we consider the same covering if $u=2$. On the other hand, if $u=1$, we consider a covering of $\set{T}_{a,j,k,1}$ by $O\left(e^{2C_4R}\right)$ disks of radius $\left(r_j^{\set{T}}\right)^m\Cmod{\ln r_j^{\set{T}}}e^{-C_4R}$. Note that $\Cmod{\ln r_j^{\set{T}}}e^{-C_4R}\leq e^{-3R}$ because $C_4>C_3+3$. Actually, we need something stronger and need to enlarge again $C_4$. 

\begin{proof}[Proof of Proposition~\ref{propD}] Points~\eqref{regHD} and~\eqref{cardHD} are clear. We have to prove~\eqref{singHD}. For a transversal $\set{T}_i=\set{T}_{a,j,k,u}$, $D\in\mathcal{V}_i$ and $x=(z_1,z_2),y=(w_1,w_2)\in\wo{2D}{U_{\sing}(R)}$, we have $z_u=w_u$. For the other coordinate, the definition of the disks covering $\set{T}_{a,j,k,u}$ makes it easy to check case by case that $\frac{\norm{z-w}_{\infty}}{\norm{z}_{\infty}}=O\left(\Cmod{\ln\norm{z}_{\infty}}e^{-C_4R}\right)=O\left(e^{-3R}\right)$. In fact, it is $O\left(e^{-C_4R}\right)$ except for the Poincar\'e--Dulac case with $u=1$ and $m=1$ (see~\eqref{z-w/zDnk} below for the less trivial case). Thus, the orthogonal projection exists and $y'=(w'_1,w'_2)=\Phi_{xy}(x)$ is given by $y'=\flot{y}(t)$, for some $t=O(\Cmod{\ln\norm{z}_{\infty}}e^{-C_4R})$. To show that~$x$ and~$y'$ are $(3h_1,e^{-3R})$-relatively close following the flow, we have to consider different cases. Note $v\in\{1,2\}$ such that $v\neq u$, so that the coordinate on the transversal is $z_v$ as in Figure~\ref{figHypHD2}.

  \begin{figure}[htb]
  \centering

\def \globalscale {10.000000}

\begin{tikzpicture}[y=0.80pt, x=0.80pt, yscale=-\globalscale, xscale=\globalscale,line cap=butt,line join=miter,miter limit=4.00,line width=0.4pt,draw=black]
\draw[<->]  (13.0886,15.7885) node[below left,xshift=-0.2cm,yshift=-0.3cm] {$z_v$} -- (13.0886,39.4003) -- (44.7355,39.4003) node[right,xshift=0.2cm] {$z_u$};

\draw[->]  (13.0886,39.4003) -- (3.3455,56.2759);

\filldraw[black,fill opacity=0.1]  (20.4859,55.4702) -- (28.7162,41.2149) -- (28.7162,17.7381) node[right,opacity=1,xshift=.1cm,yshift=-.3cm] {$\set{T}$} --
  (20.4859,31.9934) -- cycle;

\draw[line width=0.3pt]  (2.7934,46.6043) .. controls (2.7934,46.6043) and (5.7699,43.5644) ..
  (7.7354,43.2627) .. controls (10.6394,42.8171) and (13.2433,45.7744) ..
  (16.1813,45.7829) .. controls (18.4648,45.7894) and (20.4746,43.9305) ..
  (22.7572,43.8617) .. controls (25.5170,43.7785) and (28.0730,45.7502) ..
  (30.8337,45.7013) .. controls (35.3367,45.6216) and (43.8929,42.2364) ..
  (43.8929,42.2364);

\draw[line width=0.7pt]  (2.7485,42.7359) .. controls (2.7485,42.7359) and (5.5901,39.1431) ..
  (7.6607,38.7834) .. controls (10.8654,38.2269) and (13.6281,42.0114) ..
  (16.8806,41.9790) .. controls (18.9733,41.9581) and (20.6786,39.8564) ..
  (22.7708,39.8051) .. controls (25.8920,39.7284) and (28.5908,42.4101) ..
  (31.7063,42.6129) .. controls (35.6746,42.8714) and (43.4554,40.5428) ..
  (43.4554,40.5428) node[right,xshift=.1cm] {$\leafu{y}$};

\draw[line width=0.3pt]  (4.7423,35.0363) .. controls (4.7423,35.0363) and (9.0531,31.6527) ..
  (11.6303,31.4920) .. controls (14.0941,31.3384) and (16.1188,34.1386) ..
  (18.5853,34.0365) .. controls (21.3246,33.9231) and (23.3079,30.5756) ..
  (26.0496,30.5818) .. controls (28.5409,30.5874) and (30.3584,33.3555) ..
  (32.8180,33.7517) .. controls (36.0545,34.2731) and (42.5835,32.5888) ..
  (42.5835,32.5888);

\draw[line width=0.3pt]  (3.8999,27.5229) .. controls (3.8999,27.5229) and (9.6259,24.1806) ..
  (12.7928,24.1797) .. controls (15.6379,24.1789) and (17.9396,27.1091) ..
  (20.7839,27.1788) .. controls (23.2246,27.2386) and (25.3228,24.7851) ..
  (27.7590,24.9443) .. controls (30.2651,25.1081) and (32.0954,27.7521) ..
  (34.5740,28.1573) .. controls (37.0455,28.5614) and (42.0362,27.2843) ..
  (42.0362,27.2843);

\filldraw[fill opacity=0.1] (25,38) node[above right,opacity=1,yshift=1cm,xshift=.3cm] {$D$} ellipse (0.0350cm and 0.1398cm);

\draw[line width=0.7pt]  (3.1254,39.0934) .. controls (3.1254,39.0934) and (6.7792,35.9789) ..
  (9.0339,35.6821) .. controls (11.9631,35.2965) and (14.6408,37.9761) ..
  (17.5953,37.9761) .. controls (19.3176,37.9761) and (20.8660,36.5519) ..
  (22.6861,36.6388) .. controls (26.3164,36.8172) and (29.0655,40.3920) ..
  (32.5003,40.4022) .. controls (36.1434,40.4128) and (42.7405,37.5829) ..
  (42.7405,37.5829) node[right,xshift=.1cm] {$\leafu{x}$};

\filldraw[black] (23.2305,43.9053) circle (0.1);

\filldraw[black] (25.7985,40.4845) node[right,xshift=.1cm,yshift=.05cm] {$y$} circle (0.1);

\filldraw[black] (25.7985,37.5725) node[right,xshift=.1cm,yshift=.05cm] {$x$} circle (0.1);

\filldraw (23.8026,31.2679) circle (0.1);

\filldraw (25.2586,25.5658) circle (0.1);

\draw[line width=0.7pt] (25.7535,37.5742) -- (25.0578,40.1706);

\filldraw[black] (25.0578,40.1706) node[above left,xshift=.05cm,yshift=.1cm] {$y'$} circle (0.1);

\draw[line width=0.7pt] (25.2969,40.2586) -- (25.3603,40.0220) -- (25.1188,39.9573);

\draw[dashed] (25.0578,40.1706) -- (25.0578,46.3);

\draw[dashed] (25.7985,46.3) -- (25.7985,37.5725);

\draw[dashed] (25.7985,40.5107) -- (9.1547,40.5107);

\draw[dashed] (25.7985,37.5725) -- (9.1547,37.5725);

\draw[dashed] (25.0578,40.1706) -- (9.3,40.1706);

\filldraw[black] (25.0578,46.3) node[below left,xshift=.05cm,yshift=0.15cm] {$w'_u$} circle (0.1);

\filldraw[black] (25.7985,46.3) node[below right,xshift=-.05cm,yshift=0.1cm] {$z_u=w_u$} circle (0.1);

\filldraw[black] (9.1547,40.5107) node[below,yshift=-.1cm] {$w_v$} circle (0.1);

\filldraw[black] (9.1547,37.5725) node[above,yshift=.05cm] {$z_v$} circle (0.1);

\filldraw[black] (9.3,40.1706) node[above,yshift=0.05cm] {$w'_v$} circle (0.1);
\end{tikzpicture}
 \caption{Proof of~\eqref{singHD} in Proposition~\ref{propD}.\label{figHypHD2}}
\end{figure}

\emph{$D=D_0$ for a linearizable, Briot--Bouquet or Poincar\'e--Dulac singularity with $u=2$.} By definition of the orthogonal projection, note that $d(x,y')\leq d(x,y)\leq r_0$. Since $C_4>C_3$ and $\Cmod{z_u}\geq\left(r_{\sing}(R)\right)^m$, it is easy to check that $z_u$ and $w_u'$ are in configuration~\eqref{Cnseplin} and $z_v$,~$w_v'$ in configuration~\eqref{Cseplin} of Lemmas~\ref{deccelllin},~\ref{deccellPD} and~\ref{deccellBB}.

\emph{$D=D_{nk}$ with the same singularities.} Recall that $D_{nk}=D\left(r_{n-1}e^{i\theta_k},r_n-r_{n-1}\right)$, with the notations at the beginning of paragraph~\ref{statementcell}. So, we have
\begin{equation}\label{z-w/zDnk}\frac{\norm{z-w}_{\infty}}{\norm{z}_{\infty}}\leq 2\frac{r_n-r_{n-1}}{2r_{n-1}-r_n}=2\frac{\exp\left(e^{-C_4R}\right)-1}{2-\exp\left(e^{-C_4R}\right)}=O\left(e^{-C_4R}\right).\end{equation}
Hence, $t$ is a $O\left(e^{-C_4R}\right)$. Since $\{z_v=0\}$ is a separatrix, that is $\frac{dz_v}{dt}$ is proportional to $z_v$, we have that $\Cmod{w_v-w'_v}\leq C\Cmod{t}\Cmod{w_v}$. It follows that $z_v$ and $w_v'$ are in configuration~\eqref{Cnseplin}. For a linearizable or Briot--Bouquet singularity, this is the same for $z_u$ and $w_u'$ with the same argument. For the Poincar\'e--Dulac case,
    \[z_2-w_2'=z_2-\left(z_2+\mu tw_1^m\right)e^{mt},\]
    because $u=2$. The conditions $\Cmod{w_1}^m\leq C\Cmod{z_2}\Cmod{\ln\norm{z}_{\infty}}^{-1}$ and $t=O\left(e^{-C_4R}\right)$ give us that $\frac{\Cmod{z_2-w'_2}}{\Cmod{z_2}}$ is a $O(\Cmod{t})$, so that $z_2$ and $w'_2$ are in configuration~\eqref{casperiphPD2} if $C_4>C_3+3$.

\emph{A Poincar\'e--Dulac singularity with $u=1$.} Since $t=O\left(\Cmod{\ln\norm{z}_{\infty}}e^{-C_4R}\right)$ and $z_1=w_1$, we have $\frac{\Cmod{z_1-w'_1}}{\Cmod{z_1}}=\Cmod{e^t-1}=O(\Cmod{t})$. Hence, if $C_4>2C_3$, $z_1$ and $w'_1$ are in configuration~\eqref{CnsepPD1}. On the other hand,
    \[z_2-w'_2=z_2-\left(w_2+\mu tz_1^m\right)e^{mt}.\]
    This time, $\Cmod{z_2-w_2}=O\left(\Cmod{z_1}^m\Cmod{\ln\norm{z}_{\infty}}e^{-C_4R}\right)$ and $\Cmod{w_2}=O\left(\Cmod{z_1}^m\Cmod{\ln\norm{z}_{\infty}}\right)$ by definition of the disks and of the transversal $\set{T}_{a,j,k,1}$. Again, one also has $t=O\left(\Cmod{\ln\norm{z}_{\infty}}e^{-C_4R}\right)$. Hence, the above equation can be written
    \[\Cmod{z_2-w_2'}\leq\Cmod{z_2-w_2}+\Cmod{w_2}O\left(\Cmod{t}\right)+\Cmod{z_1}^mO\left(\Cmod{t}\right)\leq\frac{e^{-3R}}{2}\Cmod{z_1}^m,\]
    if $C_4>2C_3$ and~$R$ is sufficiently large. Therefore, $z_2,w'_2$ are in configuration~\eqref{cascentrPD2}.
    \end{proof}
    
\section{Holonomy and flow in hyperbolic time}\label{secholo}

\subsection{Comparison of three motion processes}

We need to control the behaviour of an orthogonal projection. Since we work on a universal cover, it will be very convenient to have processes that are invariant under homotopy and that is why we study holonomy. Moreover, this motion enables us to carry information step by step from far away transversals to the regular transversals where we want to build orthogonal projections. Still, we need it to be well defined on an initial disk and a way to keep disks in the process. Here, estimating the flow is crucial to guarantee these technical elements. The end of the proof relies deeply on comparing these three motions: following a local orthogonal projection, holonomy and parallel flow in same time. We need some preparation.

\begin{lem}\label{unifboiteaflots} Let $a\in\mani{E}$ be some singularity of $\fol$. There exists a constant $\eps_3>0$ such that for any $z_0\in\frac{3}{4}U_a$, there exists a regular flow box $U\simeq\set{D}\times\set{T}$ with $B\left(z_0,\eps_3\norm{z_0}_{\infty}\right)\subset U$.
\end{lem}

\begin{proof} If $z\in U_a$ satisfies $\norm{z-z_0}_{\infty}\leq\frac{1}{2C_0C_1}\norm{z_0}_{\infty}$, then $\norm{X(z)-X(z_0)}_{\infty}\leq\frac{1}{4}\norm{X(z_0)}_{\infty}$ by~\eqref{X(z)=z} and~\eqref{X(z-w)=z-w}. Write $X=X_1\der{}{z_1}+X_2\der{}{z_2}$. Without loss of generality, we can suppose that $\norm{X(z_0)}_{\infty}=\Cmod{X_1(z_0)}$. Considering the exponential map of $\frac{X}{X_1}$, we conclude.
\end{proof}

This preparation enables us to ensure that the holonomy maps we consider are not only germs, but are defined on the whole disks of the initial covering.

\begin{prop}\label{holodefD} Let $\set{T}_i,\set{T}_j\in\wt{\set{T}}$, $z\in\set{T}_i$ and $\lambda\colon\intcc{0}{1}\to\leafu{z}$ be such that $\lambda(0)=z$, $\lambda(1)\in\set{T}_j$ and $\lPC(\lambda)<2h_1$. Let $D\in\mathcal{V}_i$ be a disk of the initial covering containing $z$. If $R$ is sufficiently large, the holonomy map $\Hol_{\lambda}$ along $\lambda$ from $\set{T}_i$ to $\set{T}_j$ is well defined on $2D$.

  If $w\in2D$, there exists a unique map $\wt{\Phi}_{zw}\colon\DR{2h_1}\to\leafu{w}$ such that near $\zeta\in\DR{2h_1}$, if $z'=\phi_z(\zeta)$ and $w'=\wt{\Phi}_{zw}(\zeta)$, then $\wt{\Phi}_{zw}=\Phi_{z'w'}\circ\phi_z$, where $\Phi_{z'w'}$ is the orthogonal projection from $\leafu{z'}$ near~$z'$ onto $\leafu{w'}$ near $w'$; and moreover $\wt{\Phi}_{zw}=\Phi_{zw}\circ\phi_z$ in a neighbourhood of $0$.

  Finally, if $w''=\Hol_{\lambda}(w)$ and $\zeta\in\DR{2h_1}$ is such that $\phi_z(\zeta)=z'=\lambda(1)$, then $\wt{\Phi}_{zw}=\Phi_{z'w''}\circ\phi_z$ in a neighbourhood of $\zeta$.
\end{prop}

\begin{proof} Since $\fol$ is Brody-hyperbolic and $h_1$ is small, the conclusions are clear for $\set{T}_i$ a regular transversal. If $\set{T}_i$ is a singular transversal, then note that $z$ and $w$ are $(3h_1,e^{-3R})$-relatively close following the flow by Proposition~\ref{propD}. Considering a flow path~$\gamma$ for $z$ representing~$\lambda$ and using Lemma~\ref{unifboiteaflots}, we can cut $\gamma$ in parts where $\lambda(t)$ and $\flot{w}(\gamma(t))$ are in the same flow box to ensure that the holonomy is well defined. The existence of $\wt{\Phi}_{zw}$ and its link with $\Phi_{z'w''}$ can be addressed similarly using Lemma~\ref{floworthproj}.
\end{proof}

\subsection{Image of disks by holonomy}

Now that we know the holonomy maps are well defined on the disks of the covering, we need a method to keep disks to apply Lemma~\ref{lemref}. We need the following notion, which is close to one that may be found in property (H3) in~\cite[p.~615]{DNSII}.

\begin{defn} Let $U$ be an open subset of $\set{C}$ and $\sigma>1$. We call $U$ \emph{$\sigma$-quasi-round} if there exists a disk $D$ with $\sigma^{-1}D\subset U\subset\sigma D$.
\end{defn}

Given a $\sigma$-quasi-round open set, we use the following to make it round.

\begin{lem}\label{rhoqr4disks} There exists $\sigma_0>1$ such that for $\sigma\in\intoo{1}{\sigma_0}$ and $U$ $\sigma$-quasi-round, there exist four disks $D_1,D_2,D_3,D_4$ with $U\subset\cup_{k=1}^4D_k$ and $2D_k\subset2U$, for $k\in\intent{1}{4}$.
\end{lem}

\begin{proof} Let $D$ be a disk with $\sigma^{-1}D\subset U\subset\sigma D$. By an affine transformation of $\set{C}$, one can suppose that $D=\set{D}$. Il is enough to build the $D_k$ satisfying $2D_k\subset2\sigma^{-1}\set{D}$ and $\sigma\set{D}\subset\cup_{k=1}^4D_k$. Easy computations show that for $\sigma<\sigma_0=\sqrt{\sqrt{6}-\sqrt{2}}\approx1.02$ and
  \[D_k=\disk{\left(4\frac{\sqrt{3}}{3}-2\right)\sigma^{-1}i^k}{\left(2-2\frac{\sqrt{3}}{3}\right)\sigma^{-1}},\]
    we have the result.
\end{proof}

Fix $\sigma_1\in\intoo{0}{\sigma_0}$. We want to prove the following result, which is somehow a control of the distortion of the holonomy mappings. See also similar results formulated in properties (H2), (H3), (H2)', (H3)', (H2)", (H3)" in \cite[pp.~615--617]{DNSII}. Note that this is where being in dimension~$2$ is crucial for the case of linearizable singularities~\cite{DNSII}. Indeed, for transversals of higher dimension, one would have polydisks of the form $\Delta=r_1\set{D}\times\dots\times r_n\set{D}$, with possibly very degenerate quotients $\frac{r_i}{r_j}$ and no refinement lemma such as Lemma~\ref{lemref} would be available.

\begin{prop}\label{imholrhoqr} We keep the notations of Proposition~\ref{holodefD}. If $h_1$ is sufficiently small, $R$ is sufficiently large and $D'\subset2D$ is a disk, then $\Hol_{\lambda}(D')$ is $\sigma_1$-quasi-round.
\end{prop}

Let us begin with the easiest cases.

\begin{proof}[Beginning of the proof of Proposition~\ref{imholrhoqr}] First, suppose that $\set{T}_i$ or $\set{T}_j$ is a \emph{regular} transversal. Note that $\Hol_{\lambda}$ is holomorphic without critical points. Hence, it is close to homotheties on small disks and we conclude. Here, we study a bounded number of maps, independently of $R$. Moreover, these arguments still work if we consider singular transversals that do not approach too much (independently of $R$) the singularities.

  Now, suppose that $\set{T}_i$ and $\set{T}_j$ are singular transversals, for a \emph{linearizable} singularity. Let us introduce many notations, that we keep for the other types of singularities. Note $z^0=(z_1^0,z_2^0)$ the center of the disk $D'$, and $z^1=(z^1_1,z_2^1)=\Hol_{\lambda}(z^0)$. Let \mbox{$\gamma\colon\intcc{0}{1}\to\leafu{z^0}$} be a flow path corresponding to $\lambda$. Here, we can suppose that $\lambda$ is a path on $\leafu{z^0}$, since we have a sequence of flow boxes that contain the image of $\lambda$ and corresponding plaques for~$z^0$. Note \mbox{$w^0=(w_1^0,w_2^0)\in D'$}, $w^1=(w_1^1,w_2^1)=\Hol_{\lambda}(w^0)$, which is well defined by Proposition~\ref{holodefD}. Let  us also denote by $w^{'1}=(w^{'1}_1,w^{'1}_2)=\flot{w^0}(\gamma(1))$. Finally, denote $\set{T}_i=\set{T}_{a,j_0,k_0,u_0}$ and $\set{T}_j=\set{T}_{a,j_1,k_1,u_1}$. By definition of the transversals, note that $z_{u_0}^0=w^0_{u_0}$ and $z^1_{u_1}=w^1_{u_1}$. Let us go back to the linearizable case. For $j\in\{1,2\}$, $\lambda_1=1$ and $\lambda_2=\lambda$, we have $w_j^{'1}-z_j^1=(w_j^0-z_j^0)e^{\lambda_j\gamma(1)}$. Hence, if $u_0=u_1$, $w^{'1}=w^1$ and $\Hol_{\lambda}(D')$ is actually a disk. On the other hand, if $u_0\neq u_1$,
  \[w^1=\flot{w^{'1}}\left(\frac{1}{\lambda_{u_1}}\ln\frac{z^1_{u_1}}{w^{'1}_{u_1}}\right)=\flot{w^{'1}}\left(\frac{1}{\lambda_{u_1}}\ln\frac{z_{u_1}^0}{w_{u_1}^0}\right).\]
      Denote by $t_w=\frac{1}{\lambda_{u_1}}\ln\frac{z_{u_1}^0}{w_{u_1}^0}$. This time is chosen such that $w_{u_1}^1=z_{u_1}^1=w_{u_1}^{'1}e^{\lambda_{u_1}t_w}$. We keep this notation such that $w^1=\flot{w^{'1}}(t_w)$ in the other cases. Here, we have the estimate $t_w=\frac{1}{\lambda_{u_1}}\frac{z^0_{u_1}-w_{u_1}^0}{z^0_{u_1}}+O\left(\Cmod{\frac{z_{u_1}^0-w_{u_1}^0}{z_{u_1}^0}}^2\right)$. This implies that
      \[w_{u_0}^1-z_{u_0}^1=z_{u_0}^0e^{\lambda_{u_0}\gamma(1)}\left(e^{\lambda_{u_0}t_w}-1\right)=\frac{\lambda_{u_0}}{\lambda_{u_1}}z_{u_0}^0e^{\lambda_{u_0}\gamma(1)}\frac{z_{u_1}^0-w_{u_1}^0}{z^0_{u_1}}\left(1+O\left(\Cmod{\frac{w^0_{u_1}-z_{u_1}^0}{z^0_{u_1}}}\right)\right).\]
      Finally, since $u_0\neq u_1$,  we claim that we can not come from a separatrix disk, \emph{i.e.} $D\neq D_0$. Else, $\Cmod{z_{u_1}^0}\leq e^{-\exp(C_4R)}$ and $\Cmod{z_{u_0}^0}\geq r_{\sing}(R)$. By the usual Gr\"{o}nwall Lemma in logarithmic time, $\Cmod{z_{u_1}^1}\leq\Cmod{z_{u_1}^0}\norm{z}_{\infty}^{-1/2}$ and $\Cmod{z_{u_0}^1}\geq\Cmod{z_{u_0}^0}\norm{z}^{1/2}_{\infty}$. Hence,
      \[\frac{\Cmod{z_{u_1}^1}}{\Cmod{z_{u_0}^1}}\leq e^{-\exp(C_4R)}r_{\sing}(R)^{-2}<\frac{2}{3}.\]
      This is a contradiction with $u_1=1$ and the definition of the transversals. Therefore, we are in some~$D_{nk}$ and as in~\eqref{z-w/zDnk}, $\Cmod{\frac{w^0_{u_1}-z_{u_1}^0}{z^0_{u_1}}}\leq Ce^{-C_4R}$. We conclude because $w_{u_1}^0$ runs through the disk~$D'$ of center $z^0_{u_1}$. The notations above are summarized in Figure~\ref{figholoflot}.

\begin{figure}[htb]
  \centering

\def \globalscale {9.5000000}
\begin{tikzpicture}[y=0.80pt, x=0.80pt, xscale=\globalscale,yscale=-\globalscale,line cap=butt,line join=miter,miter limit=4.00,line width=0.4pt,draw=black]
\draw  (-3.6848,16.5616) node[above] {$\set{T}_{a,j_0,k_0,u_0}$} -- (-3.6848,35.1264) node[right] {$z_{u_0}=\cst$};

\draw
  (8.6619,10.3503) node[below] {$z_{u_1}=\cst$} -- (32.2626,10.3503) node[right] {$\set{T}_{a,j_1,k_1,u_1}$};

\draw[line width=0.7pt]
  (-3.6808,24.6885) .. controls (-3.6808,24.6885) and (4.4880,24.2064) ..
  (8.1997,22.6706)  node[below,rotate=20,yshift=.225cm] {\tiny $\blacktriangleright$} .. controls (11.0405,21.4952) and (13.6088,19.4283) ..
  (16.0490,17.5580) node[left,yshift=-1.4cm,xshift=-2.3cm] {$\lambda(t)=\flot{z^0}(\gamma(t))$} .. controls (17.7743,16.2357) and (18.3097,15.8031) ..
  (19.4685,13.9639) .. controls (19.9861,13.1425) and (21.4560,10.3547) ..
  (21.4560,10.3547);

\draw[line width=0.7pt]
  (-3.6820,26.6636) .. controls (-3.6820,26.6636) and (1.9171,26.0181) ..
  (4.6902,25.5258)  node[below] {$\flot{w^0}(\gamma(t))$}.. controls (6.5508,25.1955) and (8.4150,24.6925) ..
  (10.1461,23.9346) node[below,rotate=30,yshift=.245cm] {\tiny $\blacktriangleright$}.. controls (12.5718,22.8726) and (14.7219,21.4755) ..
  (16.7429,19.7647) .. controls (18.9463,17.8995) and (20.3039,16.9217) ..
  (21.6567,14.3713) node[xshift=.19cm,yshift=.4cm,rotate=60] {\tiny $\blacktriangleright$}.. controls (22.1701,13.4033) and (23.4568,10.3550) ..
  (23.4568,10.3550) node[below right,xshift=-.35cm,yshift=-.4cm] {$\flot{w^{'1}}(tt_w)$};

\filldraw[black] (-3.6848,26.6488) node[left] {$w^0$} circle (0.1);

\filldraw[black] (-3.6848,24.6627) node[left] {$z^0$} circle (0.1);

\filldraw[black] (21.4560,10.3503) node[above left] {$z^1$} circle (0.1);

\draw[dashed]  (21.4560,10.3503) -- (21.4560,25.8339);

\filldraw[black] (21.4560,14.7252) node[right] {$w^{'1}$} circle (0.1);

\filldraw[black] (23.4452,10.3503) node[above right] {$w^1$} circle (0.1);

\draw (-4,46) node[above,yshift=.4cm] {$D'$} circle (0.0460cm);

\draw (13,46) circle (0.0542cm);

\filldraw[black] (-4,46) node[left,xshift=.15cm] {$z^0_{u_1}$} circle (0.1);

\filldraw[black] (-3.2,46.8) node[below right,xshift=-0.05cm,yshift=.3cm] {$w^0_{u_1}$} circle (0.1);

\filldraw[black] (13,46) node[left,xshift=.15cm] {$z_{u_1}^1$} circle (0.1);

\filldraw[black] (13.9426,46.9426) node[below right,xshift=-0.05cm,yshift=.3cm] {$w^{'1}_{u_1}$} circle (0.1);

\filldraw[black] (31,46) node[left,xshift=.15cm] {$z_{u_0}^1$} circle (0.1);

\filldraw[black] (32.5,47.8) node[left,xshift=.1cm,yshift=.1cm] {$w_{u_0}^1$} circle (0.1);

\draw[dashed] (31,46) circle (0.0896cm);

\draw[dashed] (31,46) circle
  (0.0597cm);

  \begin{scope}[xshift=-.65,yshift=-.3]
\draw
  (29.9965,44.4145) .. controls (30.2083,44.1862) and (30.2086,44.2159) ..
  (30.4875,44.0935) .. controls (30.6494,44.0224) and (30.8020,43.8745) ..
  (30.9773,43.8519) .. controls (31.2741,43.8135) and (31.4591,43.9572) ..
  (31.7574,43.9323) .. controls (32.0506,43.9077) and (32.3558,43.8569) ..
  (32.6402,43.9323) .. controls (32.9221,44.0069) and (33.1749,44.1803) ..
  (33.4003,44.3653) .. controls (33.6232,44.5482) and (33.8164,44.7730) ..
  (33.9668,45.0189) .. controls (34.1664,45.3453) and (34.3491,45.7035) ..
  (34.3962,46.0830) node[right,xshift=-.1cm,yshift=.05cm] {$\Hol_{\lambda}(D')$} .. controls (34.4351,46.3969) and (34.2863,46.9767) ..
  (34.2669,47.0231) .. controls (34.2476,47.0695) and (34.1132,47.6249) ..
  (33.9379,47.8730) .. controls (33.7285,48.1692) and (33.4024,48.3641) ..
  (33.1124,48.5820) .. controls (32.8585,48.7727) and (32.6176,49.0068) ..
  (32.3151,49.1033) .. controls (31.9176,49.2301) and (31.4677,49.2580) ..
  (31.0644,49.1512) .. controls (30.8143,49.0850) and (30.5775,48.9345) ..
  (30.4049,48.7416) .. controls (30.3009,48.6253) and (30.2967,48.4585) ..
  (30.2035,48.3191) .. controls (30.0124,48.0333) and (29.7970,48.0685) ..
  (29.6086,47.9225) .. controls (29.4523,47.8015) and (29.3407,47.7074) ..
  (29.1669,47.5269) .. controls (28.9931,47.3464) and (28.8290,46.9213) ..
  (28.8229,46.5881) .. controls (28.8191,46.3821) and (28.8926,46.1667) ..
  (29.0141,46.0003) .. controls (29.1473,45.8179) and (29.4279,45.7846) ..
  (29.5641,45.6043) .. controls (29.7170,45.4020) and (29.8225,45.0340) ..
  (29.8232,44.8891) .. controls (29.8237,44.7442) and (29.9965,44.4145) ..
  (29.9965,44.4145) -- cycle;
\end{scope}

\draw[line width=0.7pt,->,>=stealth] (-1.5,46) node[above,xshift=1.7cm] {$w^0\mapsto w^{'1}=\flot{w^0}(\gamma(1))$}-- (10.5,46);

\draw[line width=0.7pt,->,>=stealth] (15.5,46) node[above,xshift=1.6cm] {$w^{'1}\mapsto w^1=\flot{w^{'1}}(t_w)$} -- (27,46);

\draw[line width=0.3pt,dash pattern=on 0.06pt off 0.06pt,->,>=stealth]
  (-3.6857,25.6193) .. controls (-5.0449,29.1989) and (-6.4041,32.7785) ..
  (-6.5789,35.9371) .. controls (-6.7538,39.0957) and (-5.7443,41.8332) ..
  (-4.7349,44.5707);

\draw[line width=0.3pt,dash pattern=on 0.06pt off 0.06pt,->,>=stealth]
  (21.4608,13.8733) .. controls (17.9035,19.4999) and (14.3461,25.1264) ..
  (13.0240,30.2541) .. controls (11.7018,35.3817) and (12.6148,40.0101) ..
  (13.0278,44.1386);

\draw[line width=0.3pt,dash pattern=on 0.06pt off 0.06pt,->,>=stealth]
  (22.6551,10.3510) .. controls (26.1269,17.7645) and (29.5986,25.1779) ..
  (31.2023,30.6651) .. controls (32.8060,36.1524) and (32.5417,39.7133) ..
  (32.2773,43.0742);

  \draw[line width=0.3pt]  (21.4560,10.3547) -- (23.0722,11.2878);

\draw[line width=0.3pt]  (22.8246,11.1418) -- (22.6798,11.3926) -- (22.9460,11.5463);
\end{tikzpicture}
 \caption{Proof of Proposition~\ref{imholrhoqr}. Here, the intermediate transversal and its disk are specific to the linearizable case. \label{figholoflot}}
\end{figure}

      Next, consider a \emph{Poincar\'e--Dulac} type singularity. We need to distinguish all four cases for starting and arrival transversals.

      \emph{$u_0=u_1=1$}. As in the linearizable case, $w^{'1}=w^1$ and $w_2^1-z_2^1=(w_2^0-z_2^0)e^{m\gamma(1)}$. It follows that $\Hol_{\lambda}(D')$ is a disk.

      \emph{$u_0=2$ and $u_1=1$}. Here, we have $t_w=\ln\frac{z_1^0}{w_1^0}$ and $D\neq D_0$. As before, we obtain the estimate $t_w=\frac{z_1^0-w_1^0}{z_1^0}\left(1+O\left(\Cmod{\frac{z_1^0-w_1^0}{z_1^0}}\right)\right)$ and $\Cmod{\frac{z_1^0-w_1^0}{z_1^0}}\leq Ce^{-C_4R}$. Combining the different conditions on $z^0,w^0,z^1$ and $w^{'1}$, we can compute
      \begin{equation}\label{u02u11}\begin{aligned}w_2^{'1}=\left(z_2^0+\mu\gamma(1)w_1^{0m}\right)e^{m\gamma(1)}=&z_2^1-\mu\gamma(1)z_1^{1m}\left(1-\left(\frac{z_1^0}{w_1^0}\right)^m\right)\\
      =&z_2^1-m\mu\gamma(1)z_1^{1m}\frac{z_1^0-w_1^0}{z_1^0}\left(1+O\left(e^{-C_4R}\right)\right),\end{aligned}\end{equation}
      where the last inequality is obtained by Taylor expanding $1-\left(1-t\right)^m$. Note that since $u_0=2$ and $\Cmod{\gamma(1)}=O\left(h\Cmod{\ln r_{j_0^{\set{T}}}}\right)$,
      \begin{equation}\label{u02u112}\Cmod{\frac{m\mu\gamma(1)z_1^{1m}}{z_2^1}}=\Cmod{\frac{m\mu\gamma(1)z_1^{0m}}{z_2^0+m\mu\gamma(1)z_1^{0m}}}=O(1),\end{equation}
      if~$h$ is sufficiently small. Hence, we can also write~\eqref{u02u11} as $w_2^{'1}=z_2^1\left(1+O\left(e^{-C_4R}\right)\right)$. Finally, let us use the explicit flow $w^1=\flot{w^{'1}}(t_w)$ and the condition $u_1=1$ to get
      \begin{equation}\label{u02u113}\begin{aligned}w_2^1-w_2^{'1}&=\left(w_2^{'1}+\mu t_ww^{'1m}_1\right)e^{mt_w}-w_2^{'1}=w_2^{'1}\left(e^{mt_w}-1\right)+\mu t_w z_1^{1m}\\
      &=mw_2^{'1}t_w\left(1+O(t_w)\right)+\mu z_1^{1m}t_w=\left(mz_2^1+\mu z_1^{1m}\right)t_w\left(1+O(t_w)\right).\end{aligned}\end{equation}
      For the last estimate, we use $\Cmod{\frac{mz_2^1}{mz_2^1+\mu z_1^{1m}}}=O(1)$, which is proven similarly to~\eqref{u02u112}. Putting together~\eqref{u02u11},~\eqref{u02u113} and the estimate of~$t_w$, we conclude that
      \[w_2^1-z_2^1=\left(-m\mu\gamma(1)z_1^{1m}+mz_2^1+\mu z_1^{1m}\right)\frac{z_1^0-w_1^0}{z_1^0}\left(1+O\left(e^{-C_4R}\right)\right).\]
      Again, we use that since $u_0=2$, the first factor is a $O\left(\Cmod{z_2^1}\right)$, similarly to~\eqref{u02u112}.

         \emph{$u_0=1$ and $u_1=2$}. Since $u_1=2$, $\Cmod{m-\mu\frac{z_1^{1m}}{z_2^1}}>\frac{1}{2}$. Arguing as in Lemma~\ref{PDTadense}, we get
\begin{equation}\label{exprtwPD}t_w=\frac{z_2^1-w_2^{'1}}{mw_2^{'1}+\mu w_1^{'1m}}\left(1+O\left(\Cmod{\frac{z_2^1-w_2^{'1}}{z_2^1}}\right)\right)=\frac{z_2^{1}-w_2^{'1}}{mz_2^1+\mu z_1^{1m}}\left(1+O\left(e^{-3R}\right)\right),\end{equation}
since $w_1^{'1}=z_1^1$ and 
\[\Cmod{w_2^{'1}-z_2^1}=\Cmod{z_2^0-w_2^0}e^{m\Re(\gamma(1))}\leq\frac{\Cmod{z_2^0-w_2^0}}{\Cmod{z_1^0}^m}\Cmod{z_1^1}^m=O\left(e^{-3R}\Cmod{z_2^1}\right).\]
 Hence, we conclude with the estimate
 \[\frac{z_1^1-w_1^1}{z_1^1}=\frac{e^{m\gamma(1)}}{mz_2^1+\mu z_1^{1m}}(z_2^0-w_2^0)\left(1+O\left(e^{-3R}\right)\right).\]

 \emph{$u_0=u_1=2$}. This time, $z_2^0=w_2^0$. Thus,
 \[z_2^1-w_2^{'1}=\mu\gamma(1)(z_1^{1m}-w_1^{'1m})\quad\text{and}\quad z_1^1-w_1^{'1}=e^{\gamma(1)}(z_1^0-w_1^0).\]
   Moreover, the following bound and a similar for $z_1^{1m}-w_1^{'1m}$ show that~\eqref{exprtwPD} still holds.
\[\Cmod{z_2^1-w_2^{'1}}=\Cmod{\mu\gamma(1)}\Cmod{z_1^{1m}-w_1^{1m}}\leq C\Cmod{\mu\gamma(1)}\Cmod{z_1^1}^me^{-C_4R}\leq C\Cmod{z_2^1}e^{-C_4R},\]
if $h_1$ is sufficiently small, because $u_1=2$. Here, we have considered the case $D\neq D_0$ but $D=D_0$ is even simpler. It follows that
\[z_1^1-w_1^1=e^{\gamma(1)}(z_1^0-w_1^0)\left(1-\frac{\mu\gamma(1)}{mz_2^1+\mu z_1^{1m}}\sum\limits_{k=0}^{m-1}\left(z_1^1\right)^k\left(w_1^{'1}\right)^{m-k}\left(1+O\left(e^{-3R}\right)\right)\right).\]
Since $u_1=2$, note that $\Cmod{z_2^1}\geq\frac{1}{3}\Cmod{\mu\gamma(1)}\Cmod{z_1^1}^m$ if $h_1$ is sufficiently small. Therefore,
\[\frac{z_1^1-w_1^1}{z_1^1}=\frac{z_1^0-w_1^0}{z_1^0}\times\frac{mz_2^1+\mu(1-m\gamma(1))z_1^{1m}}{mz_2^1+\mu z_1^{1m}}\left(1+O\left(e^{-3R}\right)\right).\]
We have our result for $R$ sufficiently large.
\end{proof}

It leaves us with the Briot--Bouquet case, for which we still need some preparation. We need to ensharpen the estimates of Subsection~\ref{subsecBB}. Let us reintroduce some of its notations. Consider the flow in $\wt{X}_3$, for $z\in\frac{1}{2}\set{D}^2$, denote by $\flot{z}(t)=(z_1(t),z_2(t))$, with $z_1(t)=z_1e^t$. As in Lemma~\ref{flowtimesmallIm}, denote by $\wt{z}_2(t)=z_2(t)e^{\alpha t}$ and $\wh{z}_2(t)=\wt{z}_2(t)-z_2$.

\begin{lem}\label{whz} Take the notations above for $z^0$ and $w^0$ in the proof of Proposition~\ref{imholrhoqr}. If $R$ is sufficiently large and $h_1$ is sufficiently small, then
  \begin{enumerate}[label=(\roman*),ref=\roman*]
  \item \label{whz1}If $z_1^0=w_1^0$, then $\Cmod{\wh{z}_2^0(\gamma(1))-\wh{w}_2(\gamma(1))}=O\left(\Cmod{z_2^0}\Cmod{z_2^0-w_2^0}\right)$.
  \item \label{whz2}If $z_2^0=w_2^0$, then $\Cmod{\wh{z}_2^0(\gamma(1))-\wh{w}_2^0(\gamma(1))}=O\left(\Cmod{z_2^0}\Cmod{z_1^0-w_1^0}\right)$.
  \end{enumerate}
\end{lem}

\begin{proof} Note that $\wh{z}_2^{0'}(t)=\wt{z}_2^{0'}(t)$ and by~\eqref{majwtz'-wtw'},
  \begin{equation}\label{eqGronwhz2w2}\Cmod{\wh{z}_2^{0'}(t)-\wh{w}_2^{0'}(t)}\leq C\left(\Cmod{z_2^0}^{\frac{3}{2}}\Cmod{z_1^0(t)-w_1^0(t)}+\Cmod{w_1^0}^{\alpha}\max\left(\Cmod{z_2^0},\Cmod{w_2^0}\right)\Cmod{\wt{z}_2^0(t)-\wt{w}_2^0(t)}\right).\end{equation}
  Now, consider the two cases of our statement. In the first one, $z_1^0(t)=w_1^0(t)$ and using Lemma~\ref{majzt-wt}, $\Cmod{\wt{z}_2^0(t)-\wt{w}_2^0(t)}\leq2\Cmod{z_2^0-w_2^0}$. Hence,
  \[\Cmod{\wh{z}_2^0(\gamma(1))-\wh{w}_2^0(\gamma(1))}\leq C\Cmod{\ln\norm{z^0}_{\infty}}\Cmod{z_1^0}^{\alpha}\Cmod{z_2^0}\Cmod{z_2^0-w_2^0}.\]
  Then, our result is a consequence of the fact that $\Cmod{z_1^0}^{\alpha}\Cmod{\ln\norm{z^0}_{\infty}}$ is bounded.

  In the second case, we have $\wh{z}_2^0(t)-\wh{w}_2^0(t)=\wt{z}_2^0(t)-\wt{w}^0_2(t)$. Applying the refined Gr\"{o}nwall Lemma~\ref{Gronwall} to~\eqref{eqGronwhz2w2}, denoting by $\beta=C\Cmod{w_1^0}^{\alpha}\max\left(\Cmod{z_2^0},\Cmod{w_2^0}\right)$, we obtain
  \[\Cmod{\wh{z}_2^0(\gamma(1))-\wh{w}_2^0(\gamma(1))}\leq C\Cmod{z_2^0}^{\frac{3}{2}}e^{\Cmod{\gamma(1)}}\Cmod{z_1^0-w_1^0}\int_0^{\Cmod{\gamma(1)}}e^{\beta\left(\Cmod{\gamma(1)}-s\right)}ds.\]
  Since $\beta\Cmod{\gamma(1)}$ is uniformly bounded, the last integral is bounded above by $C\Cmod{\gamma(1)}$. On the other hand $e^{\Cmod{\gamma(1)}}\leq\Cmod{z_2^0}^{-\frac{1}{2}}$ if $h_1$ is sufficiently small. The result follows.
\end{proof}

\begin{proof}[End of proof of Proposition~\ref{imholrhoqr}] We just have to consider a Briot--Bouquet case and with our arguments in the regular case, we can suppose that $\norm{z^0}_{\infty}$ is sufficiently small (independently of $R$). Considering $\wh{X}_3$ if necessary, we can suppose that $u_1=1$. We have to distinguish whether $u_0=1$ or $u_0=2$.

  \emph{Case $u_0=1$}. In that case, $t_w=0$.  It follows that
  \[z_2^1-w_2^1=e^{-\alpha\gamma(1)}\left(z_2^0-w_2^0+\wh{z}_2^0(\gamma(1))-\wh{w}_2^0(\gamma(1))\right)=e^{-\alpha\gamma(1)}\left(z_2^0-w_2^0\right)\left(1+O\left(\Cmod{z_2^0}\right)\right).\]
  For $O\left(\Cmod{z_2^0}\right)\leq\sigma_1-1$, we have our result.

  \emph{Case $u_0=2$}. As for other singularities, $t_w=\ln\frac{z_1^0}{w_1^0}=\frac{z_1^0-w_1^0}{z_1^0}\left(1+O\left(e^{-C_4R}\right)\right)$, since $D\neq D_0$. On the other hand, $z_2^1-w_2^{'1}=e^{-\alpha\gamma(1)}(\wh{z}_2^0(\gamma(1))-\wh{w}_2^0(\gamma(1)))$. We get
  \[z_2^1-w_2^1=e^{-\alpha\gamma(1)}\left(\wh{z}_2^0(\gamma(1))-\wh{w}_2^0(\gamma(1))-\alpha\wt{z}_2^1(1+z_1^1z_2^1f(z_1^1,z_2^1))\frac{z_1^0-w_1^0}{z_1^0}\left(1+O\left(e^{-3R}\right)\right)\right),\]
  because $z_2^0$ and $w_2^0$ are $(3h_1,e^{-3R})$-relatively close following the flow. Here, we have denoted by $\wt{z}_2^1=\wt{z}_2^0(\gamma(1))$. By Lemma~\ref{whz} and Lemma~\ref{majwtz2},
  \[z_2^1-w_2^1=-\alpha e^{-\alpha\gamma(1)}\wt{z}_2^1(1+z_1^1z_2^1f(z_1^1,z_2^1))\frac{z_1^0-w_1^0}{z_1^0}\left(1+O\left(e^{-3R}\right)+O\left(\Cmod{z_1^0}\right)\right).\]
  Choosing $\Cmod{z_1^0}$ sufficiently small and $R$ sufficiently large, we conclude.
\end{proof}

\section{Hyperbolic motion trees}\label{sechyp}

\subsection{Covering the Poincar\'e disk}

We still face some serious problems. Even though a transversal can only interact with a finite number of others by~\eqref{Tsep}, it can do so in many ways. Indeed, a transversal that is close to the singularity has around $e^{C_3R}$ monodromy paths of Poincar\'e length lower than $3h_1$. To tackle this, we consider some standard flow motion. We need a more precise version of Lemma~\ref{Cmodgammat}, the proof of which is the same.

\begin{lem}\label{compllPC} Let $z\in\frac{1}{2}\wo{U_a}{\{a\}}$, $\gamma\colon\intcc{0}{1}\to\set{C}$ be a flow path for~$z$ and $h_1$ be sufficiently small. We denote by $\lPC(\gamma)$ the Poincar\'e length of $\flot{z}\circ\gamma$ and $\ell(\gamma)$ the length of~$\gamma$ for the usual Hermitian metric of $\set{C}$. If $\lPC(\gamma)\leq h_1$, then there exists a constant $C_7>1$ with
\[C_7^{-1}\ell(\gamma)\leq\lPC(\gamma)\Cmod{\ln\norm{z}_{\infty}}\leq C_7\ell(\gamma).\]
\end{lem}

Fix $\hbar=\frac{1}{3C_8^2}h_1$ and an integer $p>0$, for $C_8>C_7$ that our further computation specifies. For $z\in\set{T}_{a,j,k,u}$ in a singular transversal, denote by
\[\flot{k}(z)=\flot{z}\left(C_8^{-1}h_1\Cmod{\ln r_j^{\set{T}}}e^{\frac{2ik\pi}{p}}\right),\qquad k\in\intent{0}{p-1}.\]
These maps are standard flows. In what follows, we use criteria close to~\cite[Lemma~4.6]{DNSII}.

\begin{lem}\label{covsing} Let $z\in\set{T}_{a,j,k,u}$ be in a singular transversal. If $p$ is sufficiently large, $h_1$ is sufficiently small and $C_8$ is well chosen, there exist $\zeta_0,\dots,\zeta_{p-1}\in\DR{h_1}$ with $\phi_z(\zeta_k)=\flot{k}(z)$ for $k\in\intent{0}{p-1}$. Moreover,
  \[\DR{h_1+\hbar}\subset\DR{h_1}\cup\bigcup\limits_{k=0}^{p-1}\DR{h_1-\hbar}(\zeta_k).\]
  In particular, if $\xi_0,\dots,\xi_{p-1}\in\DR{h_1+\hbar}$ and satisfy $\dPC{\zeta_k}{\xi_k}\leq\hbar$, $k\in\intent{0}{p-1}$, then
  \[\DR{h_1+\hbar}\subset\DR{h_1}\cup\bigcup\limits_{k=0}^{p-1}\DR{h_1}(\xi_k).\]
\end{lem}

\begin{proof} The second statement is indeed a direct consequence of the first one. Note that by definition of our transversals, we have $c^{-1}\Cmod{\ln r_j^{\set{T}}}\leq\Cmod{\ln\norm{z}_{\infty}}\leq c\Cmod{\ln r_j^{\set{T}}}$. Define $C_8=cC_7$. If $h_1$ is sufficiently small, Gr\"{o}nwall Lemma implies that the flow is defined on the whole disk $C_8h_1\Cmod{\ln r_j^{\set{T}}}\set{D}$. Consider the paths $\lambda_k\colon\intcc{0}{1}\to\leafu{z}$ defined by
  \[\lambda_k(t)=\flot{z}\left(C_8^{-1}th_1\Cmod{\ln r_j^{\set{T}}}e^{\frac{2ik\pi}{p}}\right).\]
    Take the lifting $\wt{\lambda}_k$ in $\set{D}$ \emph{via} $\phi_z$ such that $\wt{\lambda}_k(0)=0$ and define $\zeta_k=\wt{\lambda}_k(1)$. Therefore, we have $\ell(\lambda_k)=C_8^{-1}h_1\Cmod{\ln r_j^{\set{T}}}$, so $\lPC(\lambda_k)\leq h_1$ by Lemma~\ref{compllPC} and by definition of $C_8$. Thus, $\zeta_k\in\DR{h_1}$. Let $\xi\in\DR{h_1+\hbar}$. If $\xi\in\DR{h_1}$, there is nothing to prove. If $\dPC{0}{\xi}\geq h_1$, define the radius $r_1=\frac{e^{h_1}-1}{e^{h_1}+1}$ and $\xi_1=r_1\frac{\xi}{\Cmod{\xi}}$. Let~$\gamma\colon\intcc{0}{\Cmod{\xi}}\to\set{C}$ be the flow path representing $\xi$ and $\lambda(t)=\flot{z}(t\gamma(r_1))$, which is homotopic to the path $\flot{z}\circ\restriction{\gamma}{\intcc{0}{r_1}}$. Hence, its lifting \emph{via} $\phi_z$ with $\wt{\lambda}(0)=0$ satisfies $\wt{\lambda}(1)=\xi_1$. By definition, we have $\ell(\lambda)=\Cmod{\gamma(r_1)}$ and $\lPC(\lambda)\geq h_1$. Thus, $\Cmod{\gamma(r_1)}\geq C_8^{-1}h_1\Cmod{\ln r_j^{\set{T}}}$ and there exists $r_1'\in\intcc{0}{r_1}$ with $\Cmod{\gamma(r_1')}=C_8^{-1}h_1\Cmod{\ln r_j^{\set{T}}}$. Note $\xi_1'=r_1'\frac{\xi}{\Cmod{\xi}}$ and the hyperbolic radius $h_1'=\ln\frac{1+r_1'}{1-r_1'}$. Applying Lemma~\ref{compllPC}, we get
 \[h_1'=\lPC\left(\restriction{\gamma}{\intcc{0}{r_1'}}\right)\geq C_8^{-1}\Cmod{\ln r_j^{\set{T}}}^{-1}\ell\left(\restriction{\gamma}{\intcc{0}{r_1'}}\right)\geq C_8^{-1}\Cmod{\ln r_j^{\set{T}}}^{-1}\Cmod{\gamma(r_1')}=3\hbar.\]
Hence, $\dPC{\xi}{\xi_1'}<h_1-2\hbar$. It follows that it is enough to find $k\in\intent{0}{p-1}$ with $\dPC{\xi_1'}{\zeta_k}\leq\hbar$. It is clear that we can find~$k$ such that
\[\Cmod{\gamma(r_1')-\Cmod{\gamma(r_1')}e^{\frac{2ik\pi}{p}}}\leq \frac{\pi}{p}\Cmod{\gamma(r_1')}=\frac{\pi}{p}C_8^{-1}h_1\Cmod{\ln r_j^{\set{T}}}.\]
Denote by $z_k=\flot{k}(z)$, $u_1'=\gamma(r_1')$ and $t_k=C_8^{-1}h_1\Cmod{\ln r_j^{\set{T}}}e^{\frac{2ik\pi}{p}}$. Consider the flow path which goes straight from~$0$ to $u_1'$ and then travels the arc from $u_1'$ to $t_k$. It is homotopic to the straight line from~$0$ to $t_k$. Thus, $\phi_z(\xi_1')=\flot{z_k}(u_1'-t_k)$. By Gr\"{o}nwall Lemma, $\Cmod{\ln\norm{z_k}_{\infty}}\geq\frac{1}{2}\Cmod{\ln\norm{z}_{\infty}}\geq\frac{1}{2c}\Cmod{\ln r_j^{\set{T}}}$ if $h_1$ is sufficiently small. By Lemma~\ref{compllPC},
    \[\dPC{\xi_1'}{\zeta_k}\leq \frac{\pi}{p}C_7C_8^{-1}h_1\Cmod{\ln\norm{z_k}_{\infty}}^{-1}\Cmod{\ln r_j^{\set{T}}}\leq\frac{2\pi h_1}{p}.\]
For $p\geq\frac{2\pi h_1}{\hbar}=6\pi C_8^2$, we have our result.
\end{proof}

Far from the singular set, we use a slightly different criterion.

\begin{lem}\label{covreg} Note $r_1=\frac{e^{h_1}-1}{e^{h_1}+1}$ and $\zeta_k=r_1e^{\frac{2ik\pi}{p}}$, for $k\in\intent{0}{p-1}$. If $h_1$ is sufficiently small,
\[\DR{h_1+\hbar}\subset\DR{h_1}\cup\bigcup\limits_{k=0}^{p-1}\DR{h_1-\hbar}(\zeta_k).\]
  In particular, if $\xi_0,\dots,\xi_{p-1}\in\DR{h_1+\hbar}$ and satisfy $\dPC{\zeta_k}{\xi_k}\leq\hbar$, $k\in\intent{0}{p-1}$, then
  \[\DR{h_1+\hbar}\subset\DR{h_1}\cup\bigcup\limits_{k=0}^{p-1}\DR{h_1}(\xi_k).\]
\end{lem}

\begin{proof} Let $\xi\in\DR{h_1+\hbar}$. If $\dPC{0}{\xi}<h_1$, then $\xi\in\DR{h_1}$. Next, suppose that $\dPC{0}{\xi}\geq h_1$. Define $\xi=\frac{r_1}{\Cmod{\xi}}\xi$ so that $\dPC{\xi}{\xi_1}<\hbar$. Note that $\lPC(\partial\DR{h_1})=\pi(e^{h_1}-e^{-h_1})\leq4\pi h_1$ if $h_1$ is sufficiently small. So, there exists $k\in\intent{0}{p-1}$ with $\dPC{\xi_1}{\zeta_k}\leq\frac{2\pi h_1}{p}<\hbar$. Since $C_8>1$, $\dPC{\xi}{\zeta_k}<2\hbar\leq h_1-\hbar$.
\end{proof}

\subsection{Encoding the hyperbolic dynamics}

Here, we introduce a notion that originates from the notion of tree given in~\cite[pp.~617--621]{DNSII}. It encodes the transversals through which a leaf $\leafu{x}$ goes in time $R$, and ensures that we get a $h_1$-dense subset.

\begin{defn} For $H\in\set{N}$, we define
\[\mathcal{A}_H=\bigsqcup\limits_{j=0}^H\intent{0}{p-1}^j,\]
where by convention $\intent{0}{p-1}^0=\{\emptyset\}$. We see $\mathcal{A}_H$ as a tree, the directed edges of which are $(i_1,\dots,i_k)\to(i_1,\dots,i_{k+1})$, for $k\in\intent{0}{H-1}$ and $i_1,\dots,i_{k+1}\in\intent{0}{p-1}$. A \emph{hyperbolic motion tree of depth $H$} is a map $\Theta\colon\mathcal{A}_H\to\set{D}$ with $\Theta(\emptyset)=0$ and for any vertex $i_1,\dots,i_k\in\intent{0}{p-1}$,
\[\DR{h_1+\hbar}\left(\Theta\left(i_1,\dots,i_k\right)\right)\subset\DR{h_1}\left(\Theta\left(i_1,\dots,i_k\right)\right)\cup\bigcup\limits_{i_{k+1}=0}^{p-1}\DR{h_1}\left(\Theta\left(i_1,\dots,i_{k+1}\right)\right),\]
and
\[\Theta(i_1,\dots,i_{k+1})\in\adhDR{h_1+\hbar}(\Theta(i_1,\dots,i_k)),\quad i_{k+1}\in\intent{0}{p-1}.\]
The \emph{cut-off definition tree of $\Theta$} is defined as the subset $\mathcal{A}_H^{\Theta}\subset\mathcal{A}_H$, made of elements $(i_1,\dots,i_k)\in\mathcal{A}_H$, for $k\in\intent{0}{H}$ such that for all $j\in\intent{0}{k}$, $\Theta(i_1,\dots,i_j)\in\DR{2h_1+H\hbar}$.
\end{defn}

These trees are designed to get $h_1$-covering subsets.
\begin{lem}\label{treecovers} Let $\Theta\colon\mathcal{A}_H\to\set{D}$ be a hyperbolic motion tree of depth $H\in\set{N}$. Then,
        \[\DR{h_1+H\hbar}\subset\bigcup_{S\in\mathcal{A}_H^{\Theta}}\DR{h_1}(\Theta(S)).\]
\end{lem}

\begin{proof} The proof is by induction on $H$. If $H=0$, this is a definition. Suppose that the lemma is true for depth up to $H$, let $\Theta$ be a hyperbolic motion tree of depth $H+1$ and $\Xi=\restriction{\Theta}{\mathcal{A}_H}$. Note $\DR{h_1}(\Theta)=\cup_{S\in\mathcal{A}_{H+1}^{\Theta}}\DR{h_1}(\Theta(S))$. Let $\zeta\in\DR{h_1+(H+1)\hbar}$. If $\zeta$ belongs to $\DR{h_1+H\hbar}$, then the induction hypothesis applies on $\Xi$ and gives $\zeta\in\DR{h_1}(\Theta)$. Next, suppose that $\dPC{0}{\zeta}\geq h_1+H\hbar$, define $r_H=\frac{e^{h_1+H\hbar}-1}{e^{h_1+H\hbar}+1}$, $\zeta=re^{i\theta}$ and $\zeta_H=r_He^{i\theta}$. The induction hypothesis on $\zeta_H$ and the condition $\dPC{\zeta}{\zeta_H}<\hbar$ give us $\xi=\Theta(i_1,\dots,i_k)$, with $(i_1,\dots,i_k)\in\mathcal{A}_H^{\Xi}$ such that $\dPC{\zeta}{\xi}\leq h_1+\hbar$. By definition, there exists $i_{k+1}\in\intent{0}{p-1}$ such that $\xi'=\Theta(i_1,\dots,i_{k+1})$ satisfies $\dPC{\xi}{\zeta}<h_1$. Since $(i_1,\dots,i_k)\in\mathcal{A}_H^{\Xi}$ and
\[\dPC{\xi'}{0}\leq\dPC{\zeta}{\xi'}+\dPC{\zeta}{0}<2h_1+(H+1)\hbar,\]
$(i_1,\dots,i_{k+1})\in\mathcal{A}_{H+1}^{\Theta}$ and the induction is finished. 
\end{proof}

These trees emerge from our work on the mesh of transversals.

\begin{lem}\label{arbrex} Let $x\in\set{T}$ and $H=\sentp{\frac{R-h_1}{\hbar}}$. If $h_1$ is small enough, then there is a hyperbolic motion tree $\Theta_x\colon\mathcal{A}_H\to\set{D}$ such that for $S\in\mathcal{A}_H^{\Theta_x}$, there exists $\set{T}_S\in\wt{\set{T}}$ with $\phi_x(\Theta_x(S))\in\set{T}_S$. Moreover, if $\phi_x(S)\in2\rho U_a$ for $S\in\mathcal{A}_H^{\Theta_x}$, then $\set{T}_S$ is a singular transversal $\set{T}_{a,j,k,u}$.
\end{lem}

\begin{proof} The proof is by induction on $H$. If $\set{T}_S$ is regular, then we apply Lemma~\ref{covreg} to first find~$\zeta_k$ for $k\in\intent{0}{p-1}$. If $\set{T}_S$ is singular, we find similar $\zeta_k$ by Lemma~\ref{covsing}. To find $\xi_k=\Theta_x(S\cdot k)$, we distinguish whether $\norm{\phi_x(\zeta_k)}_{\infty}\leq2\rho$ or not. In the first case, we apply Lemma~\ref{flowtimesmallIm}. In the second case, we choose $\xi_k$ in a transversal such that $\norm{\phi_x(\xi_k)}_{\infty}>2\rho$. Note also that if $\dPC{0}{\zeta_k}\geq2h_1+H\hbar$, we can keep the $\xi_k=\zeta_k$. To apply the lemmas, we need to check that $\norm{\phi_x(\zeta_k)}_{\infty}\geq r_{\sing}(R)$. If $h_1$ is sufficiently small, since we have $\dPC{0}{\zeta_k}<2h_1+H\hbar<R+2h_1$, we get this condition using Lemma~\ref{Bowcella}.
\end{proof}

With the notations of the previous proof, we want to make the choice of $\xi_k$ in the singular case in some sense uniform on a disk of the initial covering. Let us introduce some notations. Let $\set{T}_i$ be a singular transversal, $D\in\mathcal{V}_i$ and $z,w\in2D$. Note $z_k=\flot{k}(z)$ and $\zeta_{k,z}$ the corresponding point in the disk that we obtain in Lemma~\ref{covsing}. Suppose that $z_k\in2\rho U_a$ and follow the procedure of Lemma~\ref{flowtimesmallIm} to find a corresponding $z'_k\in\leafu{z_k}$ and $\xi_{k,z}\in\DR{h_1+\hbar}$, with $z'_k$ on a transversal and $\phi_z(\xi_{k,z})=z'_k$. Denote by $\wt{\gamma}_{k,z}^{(1)}$ the geodesic from~$0$ to $\zeta_{k,z}$, $\wt{\gamma}_{k,z}^{(2)}$ the geodesic from $\zeta_{k,z}$ to $\xi_{k,z}$ and $\wt{\gamma}_{k,z}=\wt{\gamma}_{k,z}^{(2)}\cdot\wt{\gamma}_{k,z}^{(1)}$ the concatenation of the two paths. For each of these, denote without tildas the projection on $\leafu{z}$, that is, for $\star\in\{(1),(2),\emptyset\}$, $\gamma_{k,z}^{\star}=\phi_z\circ\wt{\gamma}_{k,z}^{\star}$.

On the other hand, consider $\lambda_{k,z}^{(1)}(t)=\flot{z}\left(tC_8^{-1}h_1\Cmod{\ln r_j^{\set{T}}}\right)$, for $t\in\intcc{0}{1}$. That way, we have $\lambda_{k,z}^{(1)}(0)=z$, $\lambda_{k,z}^{(1)}(1)=z_k$. Since it is how we have built $\zeta_{k,z}$, $\gamma_{k,z}^{(1)}$ and $\lambda_{k,z}^{(1)}$ are homotopic. Consider then the lifting $\wt{\lambda}_{k,z}^{(1)}$ of $\lambda_{k,z}^{(1)}$ \emph{via} $\phi_z$ such that $\wt{\lambda}_{k,z}^{(1)}(0)=0$. It follows that $\wt{\lambda}_{k,z}^{(1)}(1)=\zeta_{k,z}$. Finally, the procedure of Lemma~\ref{flowtimesmallIm} gives us a flow path and then a path $\lambda_{k,z}^{(2)}$ on $\leafu{z}$ joining $z_k$ and $z'_k$, by definition of $\xi_{k,z}$ . Its lifting $\wt{\lambda}_{k,z}^{(2)}$ \emph{via} $\phi_z$ such that $\wt{\lambda}_{k,z}^{(2)}(0)=\zeta_{k,z}$ clearly satisfies $\wt{\lambda}_{k,z}^{(2)}(1)=\xi_{k,z}$. Let us also denote $\wt{\lambda}_{k,z}=\wt{\lambda}_{k,z}^{(2)}\cdot\wt{\lambda}_{k,z}^{(1)}$ and $\lambda_{k,z}=\lambda_{k,z}^{(2)}\cdot\lambda_{k,z}^{(1)}$ the concatenation of these paths.

We use analogous notations for~$w$, with slight subtelty. To be more precise, let us denote by $w_k=\flot{k}(w)$, $\zeta_{k,w}$ its preimage by $\phi_w$ obtained by Lemma~\ref{covsing}, $\wt{\gamma}_{k,w}^{(1)}$ the geodesic joining~$0$ and $\zeta_{k,w}$, $\gamma_{k,w}=\phi_w\circ\wt{\gamma}_{k,w}^{(1)}$, $\lambda_{k,w}^{(1)}(t)=\flot{w}\left(tC_8^{-1}h_1\Cmod{\ln r_j^{\set{T}}}\right)$ and $\wt{\lambda}_{k,w}^{(1)}$ its lifting \emph{via} $\phi_w$ such that $\wt{\lambda}_{k,w}^{(1)}(0)=0$. We stop here with the strict analogy. Indeed, we wish to show that the construction for~$z$ also holds for~$w$ to have a uniform choice. Note $w_k'=\Hol_{\gamma_{k,z}}(w)$, which is well defined by Proposition~\ref{holodefD}. We want to find a path $\lambda_{k,w}^{(2)}$ joining $w_k$ and $w'_k$ with peculiar properties (see below). Once we have it, define $\wt{\lambda}_{k,w}^{(2)}$ its lifting \emph{via} $\phi_w$ such that $\wt{\lambda}_{k,w}^{(2)}(0)=\zeta_{k,w}$ and denote by $\xi_{k,w}$ its endpoint. Then, take analogous notations with $\wt{\gamma}_{k,w}^{(2)}$ the geodesic joining $\zeta_{k,w}$ with $\xi_{k,w}$, $\gamma_{k,w}^{(2)}=\phi_w\circ\wt{\gamma}_{k,w}^{(2)}$ and $\lambda_{k,w}$, $\wt{\lambda}_{k,w}$, $\gamma_{k,w}$, $\wt{\gamma}_{k,w}$ the usual concatenations. See the diagramm below for a summary of all these notations.

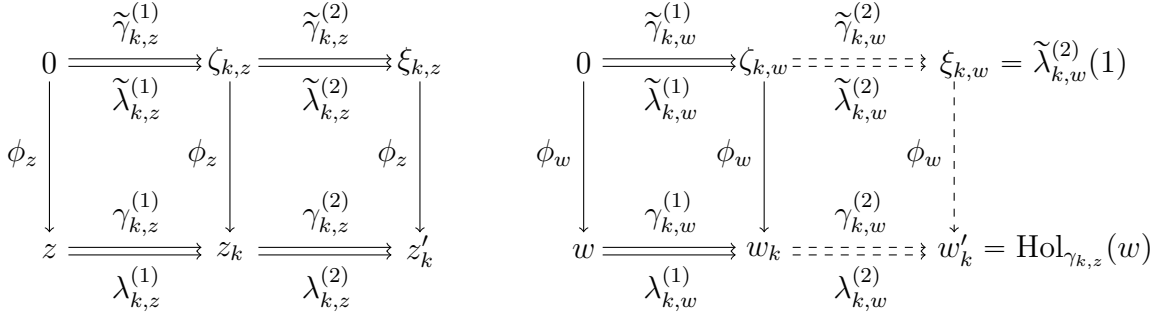
\begin{figure}[ht]
  \centering
  \begin{tikzpicture}[scale=2.5,line cap=butt,line join=miter,miter limit=4.00,line width=0.4pt,draw=black]
    \begin{scope}
      \draw (0.05,0) node {$0$};
      \draw (0.05,-1) node {$z$};
      \draw (1,0) node {$\zeta_{k,z}$};
      \draw (1,-1) node {$z_k$};
      \draw (2,0) node {$\xi_{k,z}$};
      \draw (2,-1) node {$z'_k$};
      \draw[->] (0.15,0.02) node[above,xshift=0.9cm] {$\wt{\gamma}_{k,z}^{(1)}$} -- (0.85,0.02);
      \draw[->] (0.15,-0.02) node[below,xshift=0.9cm] {$\wt{\lambda}_{k,z}^{(1)}$} -- (0.85,-0.02);
      \draw[->] (0.15,-0.98) node[above,xshift=0.9cm] {$\gamma_{k,z}^{(1)}$} -- (0.85,-0.98);
      \draw[->] (0.15,-1.02) node[below,xshift=0.9cm] {$\lambda_{k,z}^{(1)}$} -- (0.85,-1.02);
      \draw[->] (1.15,0.02) node[above,xshift=0.9cm] {$\wt{\gamma}_{k,z}^{(2)}$} -- (1.85,0.02);
      \draw[->] (1.15,-0.02) node[below,xshift=0.9cm] {$\wt{\lambda}_{k,z}^{(2)}$} -- (1.85,-0.02);
      \draw[->] (1.15,-0.98) node[above,xshift=0.9cm] {$\gamma_{k,z}^{(2)}$} -- (1.85,-0.98);
      \draw[->] (1.15,-1.02) node[below,xshift=0.9cm] {$\lambda_{k,z}^{(2)}$} -- (1.85,-1.02);
      \draw[->] (0.05,-0.1) node[left,yshift=-0.95cm] {$\phi_z$}-- (0.05,-0.9);
      \draw[->] (1,-0.1) node[left,yshift=-0.95cm] {$\phi_z$}-- (1,-0.9);
      \draw[->] (2,-0.1) node[left,yshift=-0.95cm] {$\phi_z$}-- (2,-0.9);
    \end{scope}

    \begin{scope}[xshift=75]
      \draw (0.05,0) node {$0$};
      \draw (0.05,-1) node {$w$};
      \draw (1,0) node {$\zeta_{k,w}$};
      \draw (1,-1) node {$w_k$};
      \draw (2,0) node[xshift=1.05cm] {$\xi_{k,w}=\wt{\lambda}_{k,w}^{(2)}(1)$};
      \draw (2,-1) node[xshift=1.2cm] {$w'_k=\Hol_{\gamma_{k,z}}(w)$};
      \draw[->] (0.15,0.02) node[above,xshift=0.9cm] {$\wt{\gamma}_{k,w}^{(1)}$} -- (0.85,0.02);
      \draw[->] (0.15,-0.02) node[below,xshift=0.9cm] {$\wt{\lambda}_{k,w}^{(1)}$} -- (0.85,-0.02);
      \draw[->] (0.15,-0.98) node[above,xshift=0.9cm] {$\gamma_{k,w}^{(1)}$} -- (0.85,-0.98);
      \draw[->] (0.15,-1.02) node[below,xshift=0.9cm] {$\lambda_{k,w}^{(1)}$} -- (0.85,-1.02);
      \draw[->,dashed] (1.15,0.02) node[above,xshift=0.9cm] {$\wt{\gamma}_{k,w}^{(2)}$} -- (1.85,0.02);
      \draw[->,dashed] (1.15,-0.02) node[below,xshift=0.9cm] {$\wt{\lambda}_{k,w}^{(2)}$} -- (1.85,-0.02);
      \draw[->,dashed] (1.15,-0.98) node[above,xshift=0.9cm] {$\gamma_{k,w}^{(2)}$} -- (1.85,-0.98);
      \draw[->,dashed] (1.15,-1.02) node[below,xshift=0.9cm] {$\lambda_{k,w}^{(2)}$} -- (1.85,-1.02);
      \draw[->] (0.05,-0.1) node[left,yshift=-0.95cm] {$\phi_w$}-- (0.05,-0.9);
      \draw[->] (1,-0.1) node[left,yshift=-0.95cm] {$\phi_w$}-- (1,-0.9);
      \draw[->,dashed] (2,-0.1) node[left,yshift=-0.95cm] {$\phi_w$}-- (2,-0.9);
    \end{scope}
    
    \end{tikzpicture}
\end{figure}

\begin{lem}\label{unifholo} With the notations above, there exists a path $\lambda_{k,w}^{(2)}$ with the following properties.
  \begin{enumerate}[label=(\roman*),ref=\roman*]
  \item \label{smallpathw}If $C_5$ is sufficiently small, $\lPC\left(\lambda_{k,w}^{(2)}\right)\leq\hbar$.
  \item \label{idholo}$\Hol_{\gamma_{k,z}}=\Hol_{\gamma_{k,w}}$ on $2D$.
  \end{enumerate}
\end{lem}

\begin{proof} Note $\delta_{k,z}$ the flow path for~$z$ corresponding to $\lambda_{k,z}$, and cut it in two parts according to $\lambda_{k,z}^{(j)}$, for $j\in\{1,2\}$ that we denote by $\delta_{k,z}^{(j)}\colon\intcc{0}{1}\to\set{C}$. Note that $w_k=\flot{w}\left(\delta_{k,z}^{(1)}(1)\right)$.  We build $\lambda_{k,w}^{(2)}$ by concatenating the two following parts.
  \begin{enumerate}[label=(\alph*),ref=\alph*]
  \item $\lambda_{k,w}^{(2,1)}(t)=\flot{w_k}\left(\delta_{k,z}^{(2)}(t)\right)$, for $t\in\intcc{0}{1}$. Note $w_k''$ its endpoint.
  \item $\lambda_{k,w}^{(2,2)}(t)=\flot{w_k''}(tt_w)$, for $t\in\intcc{0}{1}$ and the notations of the proof of Proposition~\ref{imholrhoqr}. Its endpoint is $w'_k$ by definition of $t_w$.
  \end{enumerate}
Since $t_w=O(e^{-3R})$ and $\frac{\norm{z}_{\infty}-\norm{w}_{\infty}}{\norm{z}_{\infty}}=O\left(e^{-C_4R}\right)$, the second part is of length smaller than~$\frac{\hbar}{2}$ if~$R$ is sufficiently large. Moreover, using Lemma~\ref{compllPC}, we can compare the Poincar\'e length of the second part as a flow path for~$z_k$ or~$w_k$ in $\leafu{z}$ or $\leafu{w}$. Choosing $C_5$ small enough, the four lemmas when we had not fixed $\hbar$ show that~\eqref{smallpathw} holds.

To prove~\eqref{idholo}, note that $\lambda_{k,w}$ is the same flow path as the one for~$z$, just completed by adding $t_w$. Hence, we can cut $\lambda_{k,z}$ and $\lambda_{k,w}$ into parts on which they stay in the same flow box. This implies that $\Hol_{\lambda_{k,w}}=\Hol_{\lambda_{k,z}}$ on $2D$. Note that on the one hand $\gamma_{k,z}$ and $\lambda_{k,z}$ and on the other hand $\gamma_{k,w}$ and $\lambda_{k,w}$ have same starting and ending point in their lifting in $\set{D}$. Therefore, they are homotopic and we also have $\Hol_{\gamma_{k,z}}=\Hol_{\gamma_{k,w}}$ on $2D$.
\end{proof}

In practice, to determine a hyperbolic motion tree for a point $x$, we do not follow exactly the proof of Lemma~\ref{arbrex}. Instead, we fix a point in the cell (say the center), follow the proof for this point and apply Lemma~\ref{unifholo} to all points in the same cell. Lemma~\ref{unifholo} implies that we do not break any symmetry with this choice because we could have done the same with any point by~\eqref{smallpathw}, with the same holonomy map by~\eqref{idholo}. We need to check the following. In our refinement process, this plays a similar role to~\cite[Lemma~2.9]{DNSII}.

\begin{lem}\label{tpsdeflottotIm} Let $z,w\in\set{T}_{a,j_0,k_0,u_0}$ and $k\in\intent{0}{p-1}$ be such that $z_k',w_k'\in\set{T}_{a,j_1,k_1,u_1}$ arrive on the same transversal, with the notations above. Here, we do not suppose that $z,w$ are in the same disk, so $w_k'$ is indeed obtained by exactly the same process as $z_k'$. Suppose that the two points $z'_k,w'_k\in2D\in\mathcal{V}_{a,j_1,k_1,u_1}$ belong to a same disk of the initial covering. Then, $\Hol_{\lambda_{k,z}}^{-1}=\Hol_{\lambda_{k,w}}^{-1}$ on $2D$, with $\lambda_{k,z},\lambda_{k,w}$ defined before.
\end{lem}

\begin{proof} Note that both holonomies are defined on $2D$ by Proposition~\ref{holodefD}. Let us introduce some notations. Denote by $z'=\Hol_{\lambda_{k,w}}^{-1}(z_k')$. Note $\delta_z$ (resp. $\delta_w$) the flow time for $X^{u_0}$ such that $z'_k=\flot{z}(\delta_z)$ (resp. $w'_k=\flot{w}(\delta_w)$). It is clear that $z'=\flot{z}(\delta_z-\delta_w+t)$, with $\Cmod{t}=O(e^{-3R})$. Moreover, $\delta_z$ and $\delta_w$ can be written as a sum of two terms: $\delta_z=t_k+t_z$, $\delta_w=t_k+t_w$, where $t_k=C_8^{-1}h_1\Cmod{\ln r_j^{\set{T}}}e^{\frac{2ik\pi}{p}}$, and $t_z$, $t_w$ are defined by Lemma~\ref{flowtimesmallIm}. In the linearizable case, since $z,z'$ belong to the same transversal, we have $\Re(t_z-t_w+t)=0$. It follows that $\Cmod{\Im(t_z-t_w+t)}<2\pi$ and hence is~$0$. So, $t_z=t_w-t$ and the condition on~$t$ imply that we have the same holonomy, by the same argument as Proposition~\ref{holodefD}. For the Briot--Bouquet case or the Poincar\'e--Dulac case with $u_0=1$, we argue the same. Consider a Poincar\'e--Dulac singularity with $u_0=2$. Note $z=(z_1,z_2)$. It is enough to show that
  \[z_2=\left(z_2+\mu tz_1^m\right)e^{mt},\]
  with $\Im(t)=O(C_5\hbar)$ and $C_5\hbar$ small enough implies that $t=0$. If $\Re(t)\geq\Cmod{\Im(t)}$, we have
  \[1=\Cmod{1+\mu t\frac{z_1^m}{z_2}}e^{m\Re(t)}\geq\left(1-\frac{3}{2}\Re(t)\left(\Cmod{\ln r_{j_0}^{\set{T}}}\right)^{-1}\right)e^{m\Re(t)}.\]
  Here, we have used that $\frac{\Cmod{z_1}^m}{\Cmod{z_2}}\leq\frac{3}{2}\Cmod{\ln r_{j_0}^{\set{T}}}^{-1}$, since $u_0=2$. Studying the real function $x\mapsto\left(1-3x/\left(2\Cmod{\ln r_{j_0}^{\set{T}}}\right)\right)e^{mx}$, it is easy to see that it is not possible if $h_1$ is sufficiently small. Using similarly that $\Cmod{1+\mu t\frac{z_1^m}{z_2}}\leq 1+\Cmod{\mu t\frac{z_1^m}{z_2}}$, one can show that $\Re(t)\leq-\Cmod{\Im(t)}$ is not possible either. Finally, we argue as in Lemma~\ref{PDTadense} to show that $t=0$.
\end{proof}

    
\section{End of proof of Theorem~\ref{mainthm}}\label{secproof}

\subsection{Refining the initial covering} We are ready to finish the proof by exposing the refinement algorithm. Our different setting makes it slightly different to the one of the three authors~\cite[Section~5, pp.~614--617]{DNSII} but the ideas are similar. Consider the initial covering $\mathcal{V}_i$ of each transversal $\set{T}_i\in\wt{\set{T}}$, for $i\in I_{\set{T}}$. We denote it by $\mathcal{V}_i^0$, because we refine it by induction. Name also $H=\sentp{\frac{R-h_1}{\hbar}}$. For $H'\in\intent{0}{H}$, we build a covering $\mathcal{V}_i^{H'}$ by disks that satisfies properties to ensure the existence of the orthogonal projection. What follows is the induction process, and contains some peremptory assertions to define the construction that we prove just after. Suppose that $\mathcal{V}_i^{H'}$ is already built.

\begin{enumerate}[label=\Roman*),ref=\Roman*]
\item \label{caseI}If $\set{T}_i\cap\left(\wo{\mani{M}}{\cup_{a\in\mani{E}}\frac{\rho}{2}U_a}\right)\neq\emptyset$,
  \begin{enumerate}[label=(\arabic*),ref=\theenumi.\arabic*]
  \item \label{Istep1} For each $D\in\mathcal{V}_i^0$, we denote by $J_D$ the set of indices $j\in I_{\set{T}}$ such that there exists $x\in D$, $x'\in\set{T}_j\cap\leafu{x}$ with $\dPC{x}{x'}\leq2h_1$. If $h_1$ is sufficiently small, the holonomy map along the geodesic from $x$ to $x'$ only depends on $\set{T}_i$ and $\set{T}_j$ and not on $x$ and $x'$. We denote it by $\pi_{ij}$, which is well defined on~$D$. 
  \item \label{Istep2} For all disk $D'\in\mathcal{V}_j^{H'}$ such that $\pi_{ij}(D)\cap D'\neq\emptyset$, $\pi_{ij}^{-1}(D')$ is $\sigma_1$-quasi-round, for~$\sigma_1$ fixed below Lemma~\ref{rhoqr4disks}. We cover it by four disks $D'_1,\dots,D'_4$ using Lemma~\ref{rhoqr4disks}. Denote by $\mathcal{V}_{i,j,D}^{H'}=\left\{D_l';\,l\in\intent{1}{4},\,\pi_{ij}(D)\cap D'\neq\emptyset\right\}$.
  \item \label{Istep3} Define $J'_D$ as the set of indices $j\in J_D$ such that $\mathcal{V}_{i,j,D}^{H'}$ is a covering of $D$. Let $J'=\cup_{D\in\mathcal{V}_i^0}J'_D$ and $\mathcal{V}_{i,j}^{H'}=\cup_{D\in\mathcal{V}_i^0}\mathcal{V}_{i,j,D}^{H'}$ for $j\in J'$. Let $\mathcal{V}_i^{H'+1}$ be the covering of~$\set{T}_i$ obtained by applying Lemma~\ref{lemref} to $\mathcal{V}_i^{H'}$ and $\left(\mathcal{V}_{i,j}^{H'}\right)_{j\in J'}$. Note that each $\mathcal{V}_{i,j}^{H'}$ does not necessarily cover entirely $\set{T}_i$ (it is only the inverse image of~$\set{T}_j$ by $\pi_{ij}$ which is covered by $\mathcal{V}_{i,j}^{H'}$, or more precisely $\cup_{D:j\in J'_D}D$) but we apply successively Lemma~\ref{lemref} on the subset each covers.
  \end{enumerate}
\item \label{caseII} If $\set{T}_i\subset\frac{\rho}{2}U_a$, for some $a\in\mani{E}$,
  \begin{enumerate}[label=(\arabic*),ref=\theenumi.\arabic*]
  \item \label{IIstep1} For each $D\in\mathcal{V}_i^0$, fix one $z\in D$. For $k\in\intent{0}{p-1}$, name by $z_k=\flot{k}(z)$, and if $\norm{z_k}_{\infty}\geq r_{\sing}(R)$, name by $\set{T}_{k,D}$ the transversal obtained by applying Lemma~\ref{flowtimesmallIm} and $z'_k=\varphi_z^{v}(\gamma^{v}(1))\in\set{T}_{k,D}$ the point obtained by the same lemma, with its notations. Consider the path $\gamma_{k,D}$ defined before Lemma~\ref{unifholo} and note $\pi_{k,D}$ the holonomy map along $\gamma_{k,D}$, which is well defined on $D$. The holonomy map $\pi_{k,D}^{-1}$ only depends on $\set{T}_i$, $\set{T}_{k,D}$ and $k$, but not on $D$. Name $J_D=\left\{k\in\intent{0}{p-1};\,\norm{z_k}_{\infty}\geq r_{\sing}(R)\right\}$.
  \item \label{IIstep2} For all disk $D'\in\mathcal{V}_{k,D}^{H'}$ such that $\pi_{k,D}(D)\cap D'\neq\emptyset$, $\pi_{k,D}^{-1}(D')$ is $\sigma_1$-quasi-round and we cover it by four disks $D_1',\dots,D_4'$ using Lemma~\ref{rhoqr4disks}. We denote by $\mathcal{V}_{i,k,D}^{H'}=\left\{D_l';\,l\in\intent{1}{4},\,\pi_{k,D}(D)\cap D'\neq\emptyset\right\}$ and $\mathcal{V}_{i,k}^{H'}=\cup_{D\in\mathcal{V}_i^0}\mathcal{V}_{i,k,D}^{H'}$.
  \item \label{IIstep3} For each $k\in J_D$, $\mathcal{V}_{i,k,D}^{H'}$ is a covering of $D$. Note $\mathcal{V}_i^{H'+1}$ the covering of $\set{T}_i$ obtained by applying Lemma~\ref{lemref} to $\mathcal{V}_i^{H'}$ and $\left(\mathcal{V}_{i,k}^{H'}\right)_{k\in\intent{0}{p-1}}$. Note that $\mathcal{V}_{i,k}^{H'}$ covers at least the disks $D$ for which $\norm{z_k}_{\infty}\geq r_{\sing}(R)$. 
  \end{enumerate}
\end{enumerate}

\begin{lem}[See also~{\cite[Proposition~5.1]{DNSII}}]\label{refbiendef} This algorithm works well. That is, the peremptory assertions used to define it hold. Moreover, there exists a constant $C_9>0$ such that $\max\limits_{i\in I_{\set{T}}}\card\mathcal{V}_i^H\leq e^{C_9R}$.
\end{lem}

\begin{proof} Let us shortly list our assertions: in~\eqref{Istep1} and~\eqref{IIstep1}, the holonomy does not depend on~$x$ and~$x'$ or on $D$; in~\eqref{Istep1} and~\eqref{IIstep1}, the holonomy is well defined on~$D$; in~\eqref{Istep2} and~\eqref{IIstep2}, the inverse images of $D'$ are $\sigma_1$-quasi-round; in~\eqref{IIstep3}, $\mathcal{V}_{i,k,D}^{H'}$ is a covering of~$D$. Note that by Lemma~\ref{lemref} and by induction, we always have $2D'\subset2D^0$, for some $D^0\in\mathcal{V}_i^0$. By Propositions~\ref{holodefD} and~\ref{imholrhoqr}, it follows that the holonomies are well defined on $2D$ and that the inverse images of $D'$ are $\sigma_1$-quasi-round. The fact that the holonomy does not depend on~$x$ and~$x'$ in~\eqref{Istep1} follows from the following observation. Since in case~\eqref{caseI}, we have $\dhimpsing{E}{x}>c$ for some $c>0$ and all $x\in\set{T}_i$, we can cover $\set{T}_i$ by a finite number of flow boxes. Reducing $h_1$ if necessary, we can suppose that $\DR{2h_1}$ is still contained in these flow boxes. The uniqueness of a point of a plaque belonging to a transversal gives us that the holonomy only depends on $\set{T}_i$ and $\set{T}_j$. In~\eqref{IIstep1}, the similar result follows from Lemma~\ref{tpsdeflottotIm}. Finally, note that Lemmas~\ref{linTadense} and~\ref{PDTadense} make us end far from the boundary of $\set{T}_{k,D}$ (this is the whole point of the $\frac{3}{2}$ in the definition of the transversals) if~$R$ is sufficiently large. Hence, $\pi_{k,D}(D)$ is fully contained in $\set{T}_{k,D}$ and $\mathcal{V}_{i,k,D}^{H'}$ is a covering of $D$. Therefore, the algorithm works well.

  It remains to prove the control of the cardinality. Note $K_{H'}=\max_{i\in I_{\set{T}}}\card\mathcal{V}_i^{H'}$. By Proposition~\ref{propD}, $K_0\leq e^{C_6R}$. Note that $\card J'\leq K'$, by~\eqref{Tsep}. Since the holonomy maps only depend on the transversals, we have by construction $\card\mathcal{V}_{i,j}^{H'}\leq 4K_{H'}$ and $\card\mathcal{V}_{i,k}^{H'}\leq 4K'K_{H'}$.  By Lemma~\ref{lemref}, we obtain for $C=\max\left(4\times200^{K'+1},4K'\times200^{p+1}\right)$, $\card\mathcal{V}_i^{H'+1}\leq CK_{H'}$. By definition of $H$, we get $K_H\leq C^{R/\hbar}e^{C_6R}$.
\end{proof}

\subsection{Proof of the existence of an orthogonal projection}

The construction of the covering $\mathcal{V}_i^H$ and of the hyperbolic motion tree clearly imply the following result. 

\begin{lem}\label{raffarbrex} Let $\set{T}_i\in\set{T}$ be a regular transversal, $D\in\mathcal{V}_i^H$ and $x\in2D$.
  Then, there exists a hyperbolic motion tree $\Theta_x\colon\mathcal{A}_H\to\set{D}$ satisfying the conditions of Lemma~\ref{arbrex} and the following. If $S_k=(i_1,\dots,i_k)\in\mathcal{A}_H^{\Theta_x}$, for $j\in\intent{0}{k-1}$, we denote by $S_j=(i_1,\dots,i_j)$. Let $\xi_j=\Theta_x(S_j)$, $\gamma_j$ be the geodesic from $\xi_j$ to $\xi_{j+1}$, $\lambda_j=\phi_x\circ\gamma_j$ and $\lambda_{S_k}=\lambda_{k-1}\dots\lambda_0$. Then, $\Hol_{\lambda_{S_k}}$ is well defined on $2D$, with image in $2D_k$, for some $D_k\in\mathcal{V}_{S_k}$.
\end{lem}




This enables us to conclude the proof of Theorem~\ref{mainthm} by checking our criterion.

\begin{prop}\label{endprop} The covering $\left(\mathcal{V}_i^H\right)_{i\in I_{\set{T}}}$ satisfies the hypotheses of Proposition~\ref{reducorthproj}.
\end{prop}

\begin{figure}[bht]
  \centering

\def \globalscale {4.800000}
\begin{tikzpicture}[y=0.80pt, x=0.80pt, yscale=-\globalscale, xscale=\globalscale,line cap=butt,line join=miter,miter limit=4.00,line width=0.4pt,draw=black]

  \begin{scope}
\draw (-3.6513,36.5654) node[below left] {$0$} circle  (0.1941cm);

\draw (1.3758,32.3321) node[above left,yshift=.4cm,xshift=.45cm] {$\DR{h_1}(\Theta_x(S_{1,1}))$} circle  (0.1941cm); 

\draw (-3.9158,42.3862) circle  (0.1941cm);

\draw (0.0529,45.5613) node[below left,yshift=.1cm,xshift=.5cm] {$\Theta_x(S_{2,2})$} circle  (0.1941cm);

\draw (6.6675,31.0092) circle  (0.1941cm);

\draw (6.6675,46.6196) circle  (0.1941cm);

\draw (12.2238,31.5383) circle  (0.1941cm);

\draw (11.9592,47.4133) circle  (0.1941cm);

\draw (16.9862,32.5967) circle  (0.1941cm);

\draw (17.2508,46.8842) circle  (0.1941cm);

\draw (21.2196,46.0904) circle  (0.1941cm);

\draw (22.8773,35.5071)  circle (0.1941cm);

\filldraw[black] (22.1456,40.4019) node[right,yshift=-.1cm] {$\xi$}circle  (0.007cm);

\filldraw[black] (-3.6513,36.5654) node[left,xshift=-.15cm,yshift=.3cm] {$\DR{h_1}$} circle  (0.007cm);

\filldraw[black] (1.3758,32.3321) circle  (0.005cm);%

\filldraw[black] (6.6675,31.0092) circle  (0.005cm);%

\filldraw[black] (12.2238,31.5383) circle  (0.005cm);%

\filldraw[black] (16.9862,32.5967) circle  (0.005cm);%

\filldraw[black] (22.8773,35.5071) node[above right,xshift=-.3cm,yshift=-0.05cm] {$\Theta_x(S_1)$} circle  (0.007cm);

\filldraw[black] (-3.9158,42.3862) circle  (0.005cm);%

\filldraw[black] (0.0529,45.5613) circle  (0.005cm);%

\filldraw[black] (6.6675,46.6196) circle  (0.005cm);%

\filldraw[black] (11.9592,47.4133) circle (0.005cm);%

\filldraw[black] (17.2508,46.8842) circle  (0.005cm);%

\filldraw[black] (21.2196,46.0904) node[below,xshift=0.4cm] {$\Theta_x(S_2)$} circle  (0.007cm);

\draw[line width=0.7pt]
  (-3.9158,42.3862) -- (0.0529,45.5613) -- (6.6675,46.6196) node[below right,xshift=-0.1cm] {$\wt{\lambda}_{S_2}$}-- (11.9592,47.4133) node[left,rotate=-7,yshift=0.01cm] {\tiny $\blacktriangleright$} 
  -- (17.2508,46.8842) -- (21.2196,46.0904) -- (22.1456,40.4019) --
  (22.8773,35.5071) -- (16.9862,32.5967) -- (12.2238,31.5383) node[left,rotate=-7] {\tiny $\blacktriangleright$}--
  (6.6675,31.0092) node[above right,yshift=-0.1cm] {$\wt{\lambda}_{S_1}$}-- (1.3758,32.3321) -- (-3.6513,36.5654) -- cycle;
\end{scope}

  \begin{scope}
\filldraw[black,fill opacity=0.1]
  (34.9250,36.7771) -- (34.9250,48.6833) -- (41.2750,48.6833) --
  (41.2750,36.7771) -- cycle;

\filldraw[black,fill opacity=0.1]
  (51.8583,27.2521) -- (51.8583,37.3062) -- (55.0333,40.4813) --
  (55.0333,30.4271) -- cycle;


\filldraw[black,fill opacity=0.1]
  (66.4104,28.8396) -- (66.4104,39.9521) -- (70.1146,37.5708) --
  (70.1146,26.4583) -- cycle;


\filldraw[black,fill opacity=0.1]
  (64.5583,40.2167) -- (64.5583,51.3292) -- (68.2625,53.7104) --
  (68.2625,42.5979) -- cycle;

  \filldraw[black,fill opacity=0.1]
  (50.2708,45.7729) -- (50.5354,55.8271) -- (53.4458,52.6521) --
  (53.4458,42.5979) -- cycle;


\filldraw[black,fill opacity=0.1]
  (84.9313,38.1000) -- (84.9313,50.0063) -- (91.5458,50.0063) --
  (91.5458,38.1000) node[above left,opacity=1] {$\set{T}_{\xi}$}-- cycle;

\filldraw[black,fill opacity=0.15] (37.0946,41.8571) circle  (0.0149cm);

\filldraw[black,fill opacity=0.15] (37.3592,41.5925) circle  (0.0597cm);

\filldraw[black,fill opacity=0.15] (68.4099,35.2571)  ellipse (0.0224cm and 0.0672cm);

\filldraw[black,fill opacity=0.15] (68.4997,34.8728)  ellipse (0.0097cm and 0.0291cm);

\filldraw[black,fill opacity=0.15] (67.0583,49.5890)  ellipse (0.0187cm and 0.0560cm);

\filldraw[black,fill opacity=0.15] (67.2877,49.0215)  ellipse (0.0060cm and 0.0179cm);

\filldraw[black,fill opacity=0.15] (52.1277,48.3277)  ellipse (0.0299cm and 0.0597cm);

\filldraw[black,fill opacity=0.15] (52.2729,48.6468)  ellipse (0.0112cm and 0.0224cm);

\filldraw[black,fill opacity=0.15] (53.5978,34.1582)  ellipse (0.0299cm and 0.0597cm);

\filldraw[black,fill opacity=0.15] (54.0770,33.5193)  ellipse (0.0075cm and 0.0149cm);

\filldraw[black,fill opacity=0.3] (88.3,44.2) circle (0.02cm);

\draw[line width=0.7pt]
  (36.8300,41.8571) .. controls (36.8300,41.8571) and (50.4875,33.7345) ..
  (54.1602,33.5227) node[above right,xshift=.5cm] {$\lambda_{S_1}$}.. controls (61.1620,33.1189) and (63.1711,33.2333) ..
  (68.5800,34.8456) node[left,rotate=-5,xshift=-.5cm,yshift=.11cm] {\tiny $\blacktriangleright$}.. controls (72.7606,36.0918) and (75.1031,36.5207) ..
  (79.1481,37.9136) .. controls (82.4306,39.0439) and (82.1252,38.7758) ..
  (85.2172,40.3544) .. controls (86.9039,41.2156) and (88.5560,44.1060) ..
  (88.5560,44.1060) .. controls (88.5560,44.1060) and (86.9204,46.1198) ..
  (84.7256,46.8264) .. controls (81.0371,48.0138) and (79.6863,48.2542) ..
  (75.8178,48.4777) .. controls (72.6672,48.6597) and (70.5455,48.7103) ..
  (67.2570,49.0008) node[left,xshift=-.7cm,yshift=-.055cm] {\tiny $\blacktriangleright$}.. controls (62.7964,49.3949) and (57.1740,49.6860) ..
  (52.3081,48.6040) node[below right,xshift=.5cm] {$\lambda_{S_2}$} .. controls (46.8141,47.3823) and (36.8300,41.8571) ..
  (36.8300,41.8571) -- cycle;

\filldraw[black] (36.8300,41.8571) node[above] {$x$} circle  (0.007cm);

\filldraw[black] (88.5560,44.1060) node[above right,xshift=-.1cm] {$\phi_x(\xi)$} circle  (0.007cm);

\filldraw[black] (54.1602,33.5227) circle  (0.005cm);%

\filldraw[black] (68.5800,34.8456) circle  (0.005cm);%

\filldraw[black] (52.3081,48.6040) circle  (0.005cm);%

\filldraw[black] (67.2570,49.0008) circle  (0.005cm);
\end{scope}

\end{tikzpicture}
 \caption{Proof of Proposition~\ref{endprop} with the notations of Lemma~\ref{raffarbrex} for the $S_{i,j}$ and the notations before Lemma~\ref{unifholo} for $\lambda_{S_i}$ and $\wt{\lambda}_{S_i}$. On the left hand side, the setup in the universal cover $\set{D}$. On the right hand side, the setup in the ambient variety and in~$\leafu{x}$. The biggest disks in the transversals are the one of the initial cover and the smallest are the images of the final disk containing~$x$ under holonomy. \label{figpreuveentropie}}
\end{figure}

\begin{proof} The cardinality condition is an immediate consequence of Lemma~\ref{refbiendef}. Let $D$ be in $\mathcal{V}_i^H$ with $\set{T}_i\in\set{T}$ and $x,y\in D$. 
  Let $\Theta_x\colon\mathcal{A}_H\to\set{D}$ be the hyperbolic motion tree of Lemma~\ref{raffarbrex} and $F=\{\Theta_x(S);\,S\in\mathcal{A}_H^{\Theta_x}\}$. The set $F$ is $h_1$-dense by Lemma~\ref{treecovers}. For $S\in\mathcal{A}_H^{\Theta_x}$, denote by $x_S=\phi_x(\Theta_x(S))$, $y_S=\Hol_{\lambda_S}(y)$, with the notations of Lemma~\ref{raffarbrex}. We build~$\psi$ by gluing all the $\wt{\Phi}_{x_Sy_S}$ of Proposition~\ref{holodefD}. If we can do so, Lemma~\ref{raffarbrex},~\eqref{regHD} and~\eqref{singHD} imply points~\eqref{critexorthproj} and~\eqref{critrelclose} of Proposition~\ref{reducorthproj}. Moreover, point~\eqref{dxyleqe2R} is given by the fact that the final covering is a refinement of the first one, which is of diameter less than $e^{-2R}$. By Proposition~\ref{holodefD} and Lemma~\ref{raffarbrex}, these maps patch up well on a branch of~$\Theta_x$. On the other hand, given two $S_1$ and $S_2$ such that $\DR{h_1}(\Theta_x(S_1))\cap\DR{h_1}(\Theta_x(S_2))\neq\emptyset$, we can build homotopic paths from~$x$ to the image of a point on the intersection because they have the same starting and ending point in~$\set{D}$. The corresponding holonomies coincide as germs and coincide on $2D$ by analytic continuation. Hence, it is clear that $\wt{\Phi}_{x_{S_1}y_{S_1}}$ and $\wt{\Phi}_{x_{S_2}y_{S_2}}$ coincide. Figure~\ref{figpreuveentropie} summarizes this proof.
\end{proof}

\bibliography{Finiteness-hyperbolic-entropy_non-deg}

\bibliographystyle{plain}

\end{document}